\documentclass[reqno]{amsart}
\usepackage[left=1in,right=1in,top=1in,bottom=1in]{geometry}
\usepackage{tikz}
\usetikzlibrary{shapes,snakes,calc,arrows}
\usetikzlibrary{decorations.pathreplacing,shapes.geometric}
\usetikzlibrary{calc,positioning}
\usepackage{amsthm}
\usepackage{amscd}
\usepackage{amsfonts}
\usepackage{amsmath}
\usepackage{amssymb}
\usepackage{mathrsfs}
\usepackage{multirow}
\usepackage{verbatim}
\usepackage{url}
\usepackage{graphicx}
\usepackage{yfonts}
\usepackage{color}

\tikzset{Box/.style={very thick, rounded corners}}
\tikzset{marked/.style={star, star point height = .75mm, star points =5, fill=black,minimum size=2mm, inner sep=0mm} }
\tikzset{verythickline/.style = {line width=7pt}}
\tikzset{thickline/.style = {line width=5pt}}
\tikzset{medthick/.style = {line width=3pt}}
\tikzset{med/.style = {line width=2pt}}
\tikzset{count/.style = {fill=white,circle,draw,thin, inner sep=2pt}}
\tikzset{rcount/.style = {fill=white,rectangle,draw,thin,inner sep=2pt, rounded corners}}
\tikzset{cpr/.style = {draw,fill=white,rectangle,thin, rounded corners}}

\definecolor{ggreen}{HTML}{00BB33}
\definecolor{ianeditcolour}{HTML}{2266FF}

\begin{document}


\newcommand{\supp}{\text{supp}}
\newcommand{\Aut}{\text{Aut}}
\newcommand{\Gal}{\text{Gal}}
\newcommand{\Inn}{\text{Inn}}
\newcommand{\Irr}{\text{Irr}}
\newcommand{\Ker}{\text{Ker}}
\newcommand{\N}{\mathbb{N}}
\newcommand{\Z}{\mathbb{Z}}
\newcommand{\Q}{\mathbb{Q}}
\newcommand{\R}{\mathbb{R}}
\newcommand{\C}{\mathbb{C}}
\renewcommand{\H}{\mathcal{H}}
\newcommand{\B}{\mathcal{B}}
\newcommand{\A}{\mathcal{A}}
\newcommand{\K}{\mathcal{K}}
\newcommand{\M}{\mathcal{M}}

\newcommand{\J}{\mathscr{J}}
\newcommand{\D}{\mathscr{D}}

\newcommand{\ul}[1]{\underline{#1}}

\newcommand{\I}{\text{I}}
\newcommand{\II}{\text{II}}
\newcommand{\III}{\text{III}}

\newcommand{\<}{\left\langle}
\renewcommand{\>}{\right\rangle}
\renewcommand{\Re}[1]{\text{Re}\ #1}
\renewcommand{\Im}[1]{\text{Im}\ #1}
\newcommand{\dom}[1]{\text{dom}\,#1}
\renewcommand{\i}{\text{i}}
\renewcommand{\mod}[1]{(\operatorname{mod}#1)}
\newcommand{\mb}[1]{\mathbb{#1}}
\newcommand{\mc}[1]{\mathcal{#1}}
\newcommand{\mf}[1]{\mathfrak{#1}}
\newcommand{\im}{\operatorname{im}}

\newcommand{\lrleq}{\leq_{\mathrm{lat}}}
\newcommand{\lrgeq}{\geq_{\mathrm{lat}}}

\newcommand{\paren}[1]{\left(#1\right)}
\newcommand{\ang}[1]{\langle#1\rangle}
\newcommand{\set}[1]{\left\{#1\right\}}
\newcommand{\sq}[1]{\left[#1\right]}


\newtheorem{thm}{Theorem}[subsection]
\newtheorem{prop}[thm]{Proposition}
\newtheorem{lem}[thm]{Lemma}
\newtheorem{cor}[thm]{Corollary}
\newtheorem{innercthm}{Theorem}
\newenvironment{cthm}[1]
  {\renewcommand\theinnercthm{#1}\innercthm}
  {\endinnercthm}
\newtheorem{innerclem}{Lemma}
\newenvironment{clem}[1]
  {\renewcommand\theinnerclem{#1}\innerclem}
  {\endinnerclem}

\theoremstyle{definition}
\newtheorem{defi}[thm]{Definition}
\newtheorem{ex}[thm]{Example}
\newtheorem*{exs}{Examples}
\newtheorem{rem}[thm]{Remark}
\newtheorem{innercdefi}{Definition}
\newenvironment{cdefi}[1]
  {\renewcommand\theinnercdefi{#1}\innercdefi}
  {\endinnercdefi}
\newtheorem{cons}[thm]{Construction}


\title{On two-faced families of non-commutative random variables}

\author{Ian Charlesworth}
\author{Brent Nelson}
\author{Paul Skoufranis}
\address{Department of Mathematics, UCLA, Los Angeles, California, 90095, USA}
\email{ilc@math.ucla.edu\\bnelson6@math.ucla.edu\\pskoufra@math.ucla.edu}
\thanks{This research was supported in part by NSF grants DMS-1161411, DMS-0838680 and by NSERC
PGS-6799-438440-2013.}

\maketitle

\begin{abstract}
We demonstrate that the notions of bi-free independence and combinatorial-bi-free independence of two-faced families are equivalent using a diagrammatic view of bi-non-crossing partitions.
These diagrams produce an operator model on a Fock space suitable for representing any two-faced family of non-commutative random variables.
Furthermore, using a Kreweras complement on bi-non-crossing partitions we establish the expected formulas for the multiplicative convolution of a bi-free pair of two-faced families.
\end{abstract}

\section{Introduction}

Free probability for pairs of faces, or simply bi-free probability, was introduced by Voiculescu in \cite{voiculescu2013freei} as a generalization of the notion of free probability to allow the simultaneous study of ``left-handed'' and ``right-handed'' variables. Prior to this work, the left and right actions were only considered separately.
Voiculescu demonstrated that many results in free probability, such as the existence of the free cumulants and the free central limit theorem, have direct analogues in the bi-free setting.
However, free independence is equivalent to a variety of computational conditions, such as vanishing alternating moments of centered variables, or vanishing mixed cumulants.
It was shown in Proposition 5.6 of \cite{voiculescu2013freei} that such computational conditions for bi-freeness exist as a collection of universal polynomials on the mixed moments of a bi-free pair of two-faced families, but their explicit formulas were unknown.

Seeking an alternate approach to bi-free probability, Mastnak and Nica in \cite{1312.0269} defined the $(\ell, r)$-cumulant functions, which they predicted to be the universal polynomials of Voiculescu.
Such cumulant functions were defined by considering permutations applied to non-crossing diagrams.
Taking inspiration from the free case, they defined defined a pair of two-faced families $z'$ and $z''$ to be combinatorially-bi-free if all mixed cumulants are zero, and posed the question of whether their definition was equivalent to the definition of bi-free independence of Voiculescu.

In this paper, we will provide an affirmative answer to their question, demonstrating the equivalence of bi-free independence and combinatorial-bi-free independence.
Analyzing \cite{1312.0269}, one can take a diagrammatic view of the desired partitions which is more natural to the study of two-faced families of non-commutatitve random variables.
In Section \ref{sec:prelim}, after some preliminaries, we introduce this view via the notion of bi-non-crossing partitions.
Such partitions are designed to encapsulate information about whether a variable should be considered on the left or on the right.
One main goal of this paper is to demonstrate that bi-non-crossing partitions play the same role in bi-free probability as non-crossing partitions play in free probability.

Following Speicher in \cite{speicher1994}, we introduce the incidence algebra on bi-non-crossing partitions in Section \ref{sec:incalg}.
The algebra enables an analysis of left and right variables simultaneously, and provides a method of M\"{o}bius inversion.
This allows us to directly obtain the bi-free cumulant functions.

In Section \ref{sec:unify} we will prove our main theorem, Theorem \ref{main_theorem}, which demonstrates that the two notions of bi-free independence are equivalent.
To do so, we analyze the action of operators on free product spaces as in \cite{voiculescu2013freei} to obtain explicit descriptions of Voiculescu's universal polynomials.
We given equivalent formulae for these polynomials using the bi-non-crossing M\"{o}bius function.

Using the combinatorially-bi-free approach, we will develop further results.
In Section \ref{sec:mult} we will describe a multiplicative free convolution of two-faced families.
By extending the Kreweras complement approach of \cite{nicaspeicher1997} to bi-non-crossing diagrams, we show that the bi-free cumulants of a product of two-faced families can be written as a convolution of the individual bi-free cumulants.

Finally, in Section \ref{sec:model} we construct an operator model in the linear operators on a Fock space for a two-faced family of non-commutative random variables.
This generalizes the model from \cite{Nica1996271} and provides a bi-free analogue of Voiculescu's non-commutative $R$-series.


\section{Preliminaries}
\label{sec:prelim}

\subsection{Free probability for pairs of faces}
\label{voiculescu}

Throughout, $z=((z_i)_{i\in I}, (z_j)_{j\in J})$ will denote a two-faced family in a non-commutative probability space $(\mc{A},\varphi)$ with the left face indexed by $I$, the right face indexed by $J$, and $I$ and $J$ disjoint. 
We will also let $z'$ and $z''$ be two-faced families, similarly indexed. 

Recall that in \cite{voiculescu2013freei}, $z'$ and $z''$ are said to be \emph{bi-freely independent} (or simply \emph{bi-free}) if there exists a free product $(\mc{X}, p,\xi)=(\mc{X}',p',\xi')*(\mc{X}'',p'', \xi'')$ of vector spaces with specified state-vectors and unital homomorphisms
	\begin{align*}
		&l^\epsilon\colon \C\<z_i^\epsilon\colon i\in I\>\rightarrow \mc{L}(\mc{X}^\epsilon),\qquad\   \text{and}\\		
		&r^\epsilon\colon \C\<z_j^\epsilon\colon j\in J\>\rightarrow \mc{L}(\mc{X}^\epsilon),\qquad \epsilon\in\{',''\},
	\end{align*}
such that the two-faced families $T^\epsilon=( (\lambda^\epsilon\circ l^\epsilon(z_i^\epsilon))_{i\in I}, (\rho^\epsilon\circ r^\epsilon(z_j^\epsilon))_{j\in J})$ with $\epsilon\in\{',''\}$ have the same joint distribution in $(\mc{L}(\mc{X}),\varphi)$ as $z'$ and $z''$.
Here $\lambda^\epsilon$ and $\rho^\epsilon$ are the left and right representations of $\mc{L}(\mc{X}^\epsilon)$ in $\mc{L}(\mc{X})$ (\emph{cf.} Section 1.9 in \cite{voiculescu2013freei}).
For $T\in \mc{L}(\mc{X}^\epsilon)$, we will often repress the $\epsilon$ notation on $\lambda^\epsilon$, $\rho^\epsilon$, and $\varphi^\epsilon$ (the state on $\mc{L}(\mc{X}^\epsilon)$ induced by $p^\epsilon$) as it will be clear which is meant by noting which vector space $T$ is defined on.

Given $\alpha\colon\{1,\ldots, n\}\rightarrow I\sqcup J$, we will refer to the ``$\alpha$-moment'' of a two-faced family $z$:
	\begin{align*}
		\varphi_\alpha(z):=\varphi(z_{\alpha(1)}\cdots z_{\alpha(n)}).
	\end{align*}

It was shown in Theorem 5.7 of \cite{voiculescu2013freei} that for each $\alpha$ there exists a universal polynomial $R_\alpha$ on indeterminates $X_K$ indexed by non-empty subsets $K\subset \{1,\ldots, n\}$ satisfying:
	\begin{enumerate}
	\item[(i)] $R_\alpha=X_{\{1,\ldots, n\}} + \tilde{R}_\alpha$, where $\tilde{R}_\alpha$ is a polynomial on indeterminates $X_K$ indexed by non-empty strict subsets $K\subsetneq\{1,\ldots, n\}$;
	
	\item[(ii)] $R_\alpha$ and $\tilde{R}_\alpha$ are homogeneous of degree $n$ when $X_K$ is given degree $|K|$; and
	
	\item[(iii)] if $R_\alpha(z)$ denotes $R_\alpha$ evaluated at $X_K=\varphi(z_{\alpha(k_1)}\cdots z_{\alpha(k_r)})$ with $K=\{k_1<\cdots <k_r\}$, then
		\begin{align*}
			R_\alpha(z'+z'')=R_\alpha(z')+R_\alpha(z''),
		\end{align*}
	when $z'$ and $z''$ are bi-free two-faced families.
	\end{enumerate}
The number $R_\alpha(z)$ is called the \emph{$\alpha$-cumulant of $z$}. Property (iii) above is referred to as the \emph{cumulant property}.

\subsection{Partitions, ordering, and non-crossing partitions}

A partition $\pi$ is a set $\pi=\{V_1,\ldots, V_k\}$, where $V_1,\ldots, V_k$ (called the \emph{blocks of $\pi$}) are non-empty sets satisfying $V_i\cap V_j=\emptyset$ for $i\neq j$ and $\bigcup_{i=1}^k V_i=\{1,\ldots, n\}$.
We traditionally order the blocks of $\pi$ so that $\min(V_1)<\cdots <\min(V_k)$.
Let $\mc{P}(n)$ denote the set of partitions of $\{1,\ldots, n\}$.

For $\pi,\sigma\in \mc{P}(n)$ we say $\pi$ is a \emph{refinement} of $\sigma$ and write $\pi\leq \sigma$ if every block of $\pi$ is contained in a block of $\sigma$. This defines a partial ordering on $\mc{P}(n)$ with minimum and maximum elements
	\begin{align*}
		0_n:=\left\{\vphantom{\sum} \{1\},\ldots, \{n\} \right\}, \qquad 1_n:=\left\{\vphantom{\sum} \{1,\ldots, n\} \right\},
	\end{align*}
respectively. We will also consider the following action of the symmetric group $S_n$ on $\mc{P}(n)$: if $\pi=\{V_1,\ldots, V_k\} \in \mc{P}(n)$ and $s\in S_n$ then
	\begin{align*}
		s\cdot \pi=\{s(V_1),\ldots, s(V_k)\} \in \mc{P}(n).
	\end{align*}
Observe that this action is order-preserving.
	
A partition $\pi \in \mc{P}(n)$ is said to be \emph{non-crossing} if for any two distinct blocks $V=\{v_1<\ldots< v_r\}, W=\{w_1<\ldots< w_s\}\in \pi$ we have $v_l<w_1<v_{l+1}$ if and only if $v_l<w_s<v_{l+1}$ ($l\in\{1,\ldots, r-1\}$). The term ``non-crossing'' refers to the fact that any such partition can be represented as a non-crossing diagram. For example, the non-crossing partition $\left\{\vphantom{\sum} \{1,5, 6\},\{2,3,4\}, \{7\} \right\}\in \mc{P}(7)$ corresponds to the diagram
	
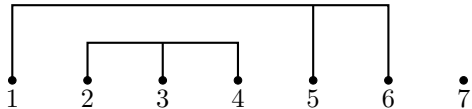
\begin{align*}
	\begin{tikzpicture}[baseline]
	\draw[thick] (1,0) -- (1,1) -- (6,1) -- (6,0);
	\draw[thick] (5,0) -- (5,1);
	\draw[thick] (2,0) -- (2,.5) -- (4,.5) -- (4,0);
	\draw[thick] (3,0) -- (3,.5);
	\node[below] at (1, 0) {1};
	\draw[fill=black] (1,0) circle (0.05);
	\node[below] at (2, 0) {2};
	\draw[fill=black] (2,0) circle (0.05);
	\node[below] at (3, 0) {3};
	\draw[fill=black] (3,0) circle (0.05);
	\node[below] at (4, 0) {4};
	\draw[fill=black] (4,0) circle (0.05);
	\node[below] at (5, 0) {5};
	\draw[fill=black] (5,0) circle (0.05);
	\node[below] at (6, 0) {6};
	\draw[fill=black] (6,0) circle (0.05);
	\node[below] at (7, 0) {7};
	\draw[fill=black] (7,0) circle (0.05);
	\end{tikzpicture}
	\end{align*}
We denote set of non-crossing partitions in $\mc{P}(n)$ by $NC(n)$.

The horizontal segments connected the nodes of a block $V\in \pi$ will be referred to as the \emph{spine of $V$}, and the segements connecting the nodes to the spine of $V$ will be referred to as the \emph{ribs of $V$}. In the following diagrammatic representation of $\{\{1,4\}, \{2,3\} \}\in NC(4)$ we have highlighted the spine of $\{1,4\}$ in red and its ribs in green:
	
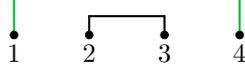
\begin{align*}
	\begin{tikzpicture}[baseline]
	\draw[thick,red] (1,.5) -- (4,.5);
	\draw[thick, ggreen] (1,0) -- (1,.5);
	\draw[thick, ggreen] (4,0) -- (4,.5);
	\draw[thick] (2,0) -- (2,.25) -- (3,.25) -- (3,0);
	\node[below] at (1,0) {1};
	\draw[fill=black] (1,0) circle (0.05);
	\node[below] at (2,0) {2};
	\draw[fill=black] (2,0) circle (0.05);
	\node[below] at (3,0) {3};
	\draw[fill=black] (3,0) circle (0.05);
	\node[below] at (4,0) {4};
	\draw[fill=black] (4,0) circle (0.05);
	\end{tikzpicture}
	\end{align*}
For a singleton block $V\in \pi$, $|V|=1$, the spine of $V$ will simply refer to the corresponding node itself.

\subsection{Combinatorial-bi-free independence}
\label{mastnak_nica}

For consistency, we note the following definitions of Mastnak and Nica.
Given $\chi : \set{1, \ldots, n} \to \set{\ell, r}$ let $\{i_1<\cdots <i_p\}=\chi^{-1}(\ell)$ and $\{j_1<\cdots<j_{n-p}\}=\chi^{-1}(r)$ and consider $\sigma_\chi\in S_n$ defined by
\begin{align*}
	\sigma_\chi(k)=\left\{\begin{array}{cl} i_k & \text{if }k\leq p\\
	     j_{n+1-k} & \text{if }k>p\end{array}\right..
\end{align*}
The class of partitions $\mc P^{(\chi)}(n) \subset \mc P(n)$ is defined as
$$\mc P^{(\chi)}(n) := \set{\sigma_\chi\cdot \pi | \pi \in NC(n) }.$$

\begin{defi}[Definition 5.2 of \cite{1312.0269}]
Let $(\mathcal{A}, \varphi)$ be a non-commutative probability space.
There exists a family of multilinear functionals 
\[
\left( \kappa_\chi : \mathcal{A}^n \to \mathbb{C}\right)_{n\geq 1, \chi: \{1,\ldots, n\} \to \{\ell, r\}}
\]
which are uniquely determined by the requirement
\[
\varphi(z_1 \cdots z_n) = \sum_{\pi \in \mathcal{P}^{(\chi)}(n)} \left( \prod_{V \in \pi} \kappa_{\chi|_V}((z_1, \ldots, z_n)|V)   \right)
\]
for every $n \geq 1$, $\chi \in \{\ell, r\}^n$, and $z_1 ,\ldots, z_n \in \mathcal{A}$.
These $\kappa_\chi$'s will be called the $(\ell, r)$-cumulant functionals of $(\mathcal{A}, \varphi)$.
\end{defi}

\begin{defi}[\cite{1312.0269}]
Let $z'$ and $z''$ each be two-faced families in $(\mc{A},\varphi)$. We say that $z'$ and $z''$ are \emph{combinatorially-bi-free} if 
\[
\kappa_\chi\left(z_{\alpha(1)}^{\epsilon_1}, \ldots, z_{\alpha(n)}^{\epsilon_n}\right) = 0
\]
whenever $\alpha : \{1,\ldots, n\} \to I \sqcup J$, $\chi: \{1,\ldots, n\} \to \{\ell, r\}$ is such that $\alpha^{-1}(I) = \chi^{-1}(\{\ell\})$, and $\epsilon\in\{',''\}^n$ is non-constant.
\end{defi}

\begin{rem}
Note that the condition $\alpha^{-1}(I) = \chi^{-1}(\{\ell\})$ completely determines $\chi$ and so we may denote
\[
	\kappa_\alpha(z):=\kappa_\chi\left(z_{\alpha(1)}, \ldots, z_{\alpha(n)}\right).
\]
Then if $z'$ and $z''$ are combinatorially-bi-free, it is easy to see that
\[
	\kappa_\alpha(z'+z'')=\kappa_\alpha(z')+\kappa_\alpha(z'');
\]
that is, $\kappa_\alpha$ has the cumulant property.
\end{rem}

\subsection{Bi-non-crossing partitions}
\label{sec:bncp}

For $\alpha\colon \{1,\ldots, n\}\rightarrow I\sqcup J$ we let $\{i_1<\cdots <i_p\}=\alpha^{-1}(I)$ and $\{j_1<\cdots<j_{n-p}\}=\alpha^{-1}(J)$ and consider $s_\alpha\in S_n$ defined by
	\begin{align*}
		s_\alpha(k)=\left\{\begin{array}{cl} i_k & \text{if }k\leq p\\
									     j_{n+1-k} & \text{if }k>p\end{array}\right..
	\end{align*}
We say a partition $\pi\in\mc{P}(n)$ is \emph{bi-non-crossing} (with respect to $\alpha$) if $s_\alpha^{-1}\cdot \pi\in NC(n)$. We denote the set of such partitions by $BNC(\alpha)$.
The minimum and maximum elements of $BNC(\alpha)$ are given by $0_\alpha := s_\alpha\cdot 0_n$ and $1_\alpha := s_\alpha\cdot1_n$, respectively.

To each partition $\pi \in BNC(\alpha)$ we can associate a ``bi-non-crossing diagram'' as follows. For each $k=1,\ldots, n$ place a node labeled $k$ at the position $(-1,n-k)$ if $\alpha(k)\in I$ and at the position $(1,n-k)$ if $\alpha(k)\in J$. Connect nodes whose labels form a block of $\pi$ similar to how one would for the diagrams associated to $NC(n)$, except now the spines of blocks are vertically oriented and the ribs extend horizontally from the spine to the left or right, emphasizing the left-right nature of a two-faced family.

\begin{ex}
\label{ex:paulsneed}
If $\alpha^{-1}(I)=\{1,2,4\}$, $\alpha^{-1}(J)=\{3,5\}$, and
	\begin{align*}
		\pi=\left\{\vphantom{\sum}\{1,3\}, \{2,4,5\} \right\}= s_\alpha\cdot \left\{\vphantom{\sum} \{1,5\}, \{2,3,4\} \right\},
	\end{align*}
then the bi-non-crossing diagram associated to $\pi$ is
	
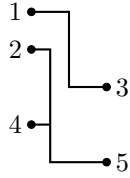
\begin{align*}
	\begin{tikzpicture}[baseline]
	\node[ left] at (-.5, 2) {1};
	\draw[fill=black] (-.5,2) circle (0.05);
	\node[left] at (-.5, 1.5) {2};
	\draw[fill=black] (-.5,1.5) circle (0.05);
	\node[right] at (.5, 1) {3};
	\draw[fill=black] (.5,1) circle (0.05);
	\node[left] at (-.5, .5) {4};
	\draw[fill=black] (-.5,.5) circle (0.05);
	\node[right] at (.5,0) {5};
	\draw[fill=black] (.5,0) circle (0.05);
	\draw[thick] (-.5,2) -- (0,2) -- (0, 1) -- (.5,1);
	\draw[thick] (-.5,1.5) -- (-.25,1.5) -- (-.25, 0) -- (.5,0);
	\draw[thick] (-.5,.5) -- (-.25,.5);
	\end{tikzpicture}
	\end{align*}
\end{ex}

That the diagram can always be drawn to be non-crossing is easily seen through its relationship to the diagram of $s_\alpha^{-1}\cdot \pi\in NC(n)$. Indeed, rotate the line $x=-1$ counter-clockwise a quarter turn about the point $(-1,0)$, rotate the line $x=1$ clockwise a quarter turn about the point $(1,0)$, and adjust the spines and ribs so that they remain connected. Then after relabeling node $k$ as $s_\alpha^{-1}(k)$ the resulting diagram is precisely the one associated to $s_\alpha^{-1}\cdot \pi$ as an element of $NC(n)$ (modulo some extra space between the nodes). Performing this operation to the above diagram yields
	
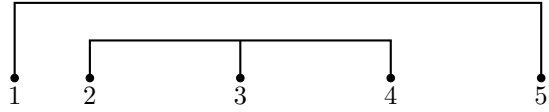
\begin{align*}
	\begin{tikzpicture}[baseline]
	\node[below] at (0,0) {1};
	\draw[fill=black] (0,0) circle (0.05);
	\node[below] at (1,0) {2};
	\draw[fill=black] (1,0) circle (0.05);
	\node[below] at (3,0) {3};
	\draw[fill=black] (3,0) circle (0.05);
	\node[below] at (5,0) {4};
	\draw[fill=black] (5,0) circle (0.05);
	\node[below] at (7,0) {5};
	\draw[fill=black] (7,0) circle (0.05);
	\draw[thick] (0,0) -- (0,1) -- (7,1) -- (7,0);
	\draw[thick] (1,0) -- (1,.5)-- (5,.5) -- (5,0);
	\draw[thick] (3,0) -- (3,.5);
	\end{tikzpicture}
	\end{align*}
Conversely given the diagram corresponding to $\sigma\in NC(n)$ we obtain the diagram for $\pi=s_\alpha\cdot \sigma$ as follows. Initially, the nodes occupy positions $(1,0),\ldots, (n,0)$, so we first widen the space between nodes so that node $k$ now occupies position $(s_\alpha(k),0)$ if $k\leq |\alpha^{-1}(I)|$ and position $(n+1-s_\alpha(k),0)$ if $k>|\alpha^{-1}(I)|$. Given the definition of $s_\alpha$, it is clear that this does not change the order of the nodes. Next, we rotate the segment from $(1,0)$ to $(n,0)$ clockwise a quarter turn about $(n,0)$, we rotate the segment from $(n+1,0)$ to $(2n, 0)$ counter-clockwise a quarter turn about $(n+1,0)$, and homotopically vary the spines and ribs so that they remain connected. Relabeling node $k$ as node $s_\alpha(k)$ then yields the diagram corresponding to $\pi$.

\begin{rem}
\label{respect_to_nica}
Given $\alpha\colon\{1,\ldots, n\}\rightarrow I\sqcup J$, define $\chi\in \{\ell,r\}^n $ by $\chi_k=\ell$ if $\alpha(k)\in I$ and $\chi_k=r$ if $\alpha(k)\in J$. Then $BNC(\alpha)$ is precisely the class of partitions $\mc{P}^{(\chi)}(n)$ defined in \cite{1312.0269} since $s_\alpha$ defined above is exactly the permutation $\sigma_\chi$ used to define the class $\mc{P}^{(\chi)}(n)$.
Moreover, the notation $BNC(\alpha)$ suggests that the lattice of partitions depends on $\alpha$ more than it actually does.
In fact, if $\beta\colon\{1,\ldots, n\}\rightarrow I\sqcup J$ is such that $\beta(j)$ and $\alpha(j)$ are in the same face for each $j=1,\ldots, N$, then $BNC(\alpha)=BNC(\beta)$
Because of this we may write $BNC(\chi)$ for $BNC(\alpha)$.
In order to emphasize the diagrammatic viewpoint pervading this paper, we will continue to use the alternate notation of $BNC(\alpha)$ for this class of partitions.
\end{rem}

\subsection{Shaded bi-non-crossing diagrams and partitions}

Let $z'$ and $z''$ be a bi-free pair of two-faced families. Let $\chi\colon\set{1,\ldots, n}\to \{\ell,r\}$ and $\epsilon\in \{',''\}^n$. We recursively define a collection of diagrams $LR(\chi,\epsilon)$.
For $n=1$, $LR(\chi,\epsilon)$ consists of two parallel, vertical, transparent segments with a single node on the left segment if $\chi(1)=\ell$ or a single node on the right segment if $\chi(1)=r$. We assign a shade to $'$ and $''$ and shade this node the shade associated to $\epsilon_1$. Then either this node remains isolated or a rib and spine of the node's shade are drawn connecting to the top of the diagram, between the two segments.
For convenience, we will refer to the space between the two vertical segments at the top of a diagram as its \emph{top gap}, through which strings may exit.

For $n > 1$ we define $LR(\chi,\epsilon)$ as follows.
Let $\chi_0=\chi\mid_{\{2,\ldots,n\}}$ and $\epsilon_0=(\epsilon_2,\ldots, \epsilon_n)$. Then a diagram of $LR(\chi,\epsilon)$ is an extension of a diagram $D\in LR(\chi_0,\epsilon_0)$: place an additional $\epsilon_1$-shaded node $p$ above $D$, on the left if $\chi(1)=\ell$ and on the right otherwise.
Extend any spines from $D$ to the new top gap.
If at least one spine was extended and the one nearest $p$ shares its shade, then connect it to $p$ with a rib and optionally terminate the spine at $p$.
Otherwise, either connect $p$ with a rib to a new spine extending to the top gap or leave $p$ isolated.

Given its impact on the diagrams, we refer to $\epsilon \in \{',''\}^n$ as a \emph{choice of shading} or simply a \emph{shading}.

Note that each diagram in $LR(\chi,\epsilon)$ is created from a unique diagram in $LR(\chi_0,\epsilon_0)$, which we can recover by simply erasing the top portion of the diagram.
Also, these rules imply that among the chords extending to the top gap, adjacent chords will always be of differing shades.
We take the convention that the nodes are labeled numerically from top to bottom.

For $0\leq k\leq n$, let $LR_k(\chi,\epsilon)\subseteq LR(\chi,\epsilon)$ consist of those diagrams with precisely $k$ chords extending to the top gap. Then $LR(\chi,\epsilon)=\bigcup_k LR_k(\chi,\epsilon)$.

We consider a few examples. In each example, we assign the shade red to $'$ and the shade green to $''$ and have a dashed line in place of the normally transparent left and right segments.

\begin{ex}\label{2_node_ex}
Consider $\chi=(\ell,r)$ and $\epsilon=(','')$.
Then $LR(\chi,\epsilon)$ consists of the following diagrams:
	
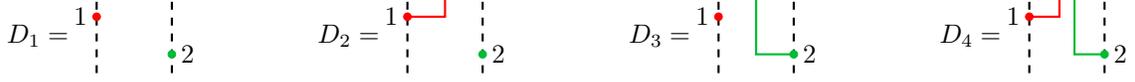
\begin{align*}
		\begin{tikzpicture}[baseline]
			\node[left] at (-.75,.25) {$D_1=$};
			\draw[thick,dashed] (-.5,-.25) -- (-.5,.75);
			\draw[thick,dashed] (.5,-.25) -- (.5,.75);
			\node[left] at (-.5,.5) {1};
			\draw[red,fill=red] (-.5,.5) circle (0.05);
			\node[right] at (.5,0) {2};
			\draw[ggreen,fill=ggreen] (.5,0) circle (0.05);
		\end{tikzpicture}
		\qquad\qquad
		\begin{tikzpicture}[baseline]
			\node[left] at (-.75,.25) {$D_2=$};
			\draw[thick,dashed] (-.5,-.25) -- (-.5,.75);
			\draw[thick,dashed] (.5,-.25) -- (.5,.75);
			\node[left] at (-.5,.5) {1};
			\draw[red,fill=red] (-.5,.5) circle (0.05);
			\draw[thick,red] (-.5,.5) -- (0,.5) -- (0, .75);
			\node[right] at (.5,0) {2};
			\draw[ggreen,fill=ggreen] (.5,0) circle (0.05);
		\end{tikzpicture}
		\qquad\qquad
		\begin{tikzpicture}[baseline]
			\node[left] at (-.75,.25) {$D_3=$};
			\draw[thick,dashed] (-.5,-.25) -- (-.5,.75);
			\draw[thick,dashed] (.5,-.25) -- (.5,.75);
			\node[left] at (-.5,.5) {1};
			\draw[red,fill=red] (-.5,.5) circle (0.05);
			\node[right] at (.5,0) {2};
			\draw[ggreen,fill=ggreen] (.5,0) circle (0.05);
			\draw[thick,ggreen] (.5,0) -- (0,0) -- (0,.75);
		\end{tikzpicture}
		\qquad\qquad
		\begin{tikzpicture}[baseline]
			\node[left] at (-.75,.25) {$D_4=$};
			\draw[thick,dashed] (-.5,-.25) -- (-.5,.75);
			\draw[thick,dashed] (.5,-.25) -- (.5,.75);
			\node[left] at (-.5,.5) {1};
			\draw[red,fill=red] (-.5,.5) circle (0.05);
			\draw[thick,red] (-.5,.5) -- (-.1, .5)-- (-.1,.75);
			\node[right] at (.5,0) {2};
			\draw[ggreen,fill=ggreen] (.5,0) circle (0.05);
			\draw[thick,ggreen] (.5,0) -- (.1,0) -- (.1,.75);
		\end{tikzpicture}
	\end{align*}
Also $LR_0(\chi,\epsilon)=\{D_1\}$, $LR_1(\chi,\epsilon)=\{D_2,D_3\}$, and $LR_2(\chi,\epsilon)=\{D_4\}$.
\end{ex}

\begin{ex}\label{3_node_ex}
For a slightly more robust example we consider $\chi=(r,\ell,r)$ and $\epsilon=(',','')$. Then $LR(\chi,\epsilon)$ consists of the following diagrams:
	
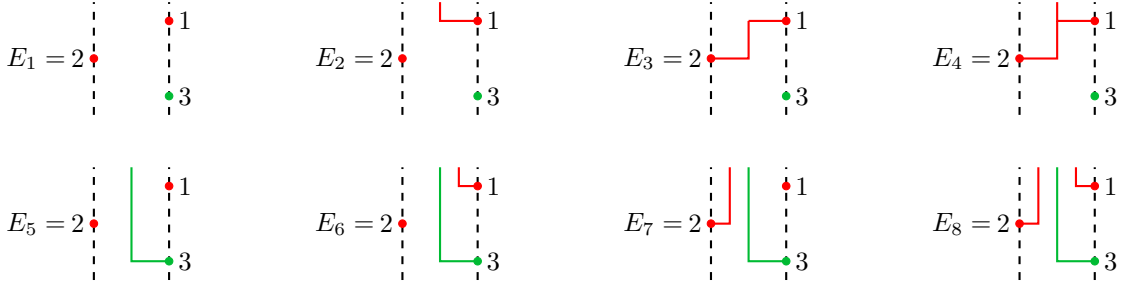
\begin{align*}
		\begin{tikzpicture}[baseline]
			\node[left] at (-.75,.5) {$E_1=$};
			\draw[thick,dashed] (-.5,-.25) -- (-.5,1.25);
			\draw[thick,dashed] (.5,-.25) -- (.5,1.25);
			\draw[red,fill=red] (.5, 1) circle (0.05);
			\draw[red,fill=red] (-.5,.5) circle (0.05);
			\draw[ggreen,fill=ggreen] (.5,0) circle (0.05);
			\node[right] at (.5,1) {1};
			\node[left] at (-.5,.5) {2};
			\node[right] at (.5,0) {3};
		\end{tikzpicture}
		\qquad\qquad
		\begin{tikzpicture}[baseline]
			\node[left] at (-.75,.5) {$E_2=$};
			\draw[thick,dashed] (-.5,-.25) -- (-.5,1.25);
			\draw[thick,dashed] (.5,-.25) -- (.5,1.25);
			\draw[red,fill=red] (.5, 1) circle (0.05);
			\draw[thick,red] (.5,1) -- (0,1) -- (0,1.25);
			\draw[red,fill=red] (-.5,.5) circle (0.05);
			\draw[ggreen,fill=ggreen] (.5,0) circle (0.05);
			\node[right] at (.5,1) {1};
			\node[left] at (-.5,.5) {2};
			\node[right] at (.5,0) {3};
		\end{tikzpicture}
		\qquad\qquad		
		\begin{tikzpicture}[baseline]
			\node[left] at (-.75,.5) {$E_3=$};
			\draw[thick,dashed] (-.5,-.25) -- (-.5,1.25);
			\draw[thick,dashed] (.5,-.25) -- (.5,1.25);
			\draw[red,fill=red] (.5, 1) circle (0.05);
			\draw[thick,red] (.5,1) -- (0,1);
			\draw[red,fill=red] (-.5,.5) circle (0.05);
			\draw[thick,red] (-.5,.5)--(0,.5)--(0,1);
			\draw[ggreen,fill=ggreen] (.5,0) circle (0.05);
			\node[right] at (.5,1) {1};
			\node[left] at (-.5,.5) {2};
			\node[right] at (.5,0) {3};
		\end{tikzpicture}
		\qquad\qquad
		\begin{tikzpicture}[baseline]
			\node[left] at (-.75,.5) {$E_4=$};
			\draw[thick,dashed] (-.5,-.25) -- (-.5,1.25);
			\draw[thick,dashed] (.5,-.25) -- (.5,1.25);
			\draw[red,fill=red] (.5, 1) circle (0.05);
			\draw[thick,red] (.5,1) -- (0,1) -- (0,1.25);
			\draw[red,fill=red] (-.5,.5) circle (0.05);
			\draw[thick,red] (-.5,.5)--(0,.5)--(0,1);
			\draw[ggreen,fill=ggreen] (.5,0) circle (0.05);
			\node[right] at (.5,1) {1};
			\node[left] at (-.5,.5) {2};
			\node[right] at (.5,0) {3};
		\end{tikzpicture}
		\\
		\\
		\begin{tikzpicture}[baseline]
			\node[left] at (-.75,.5) {$E_5=$};
			\draw[thick,dashed] (-.5,-.25) -- (-.5,1.25);
			\draw[thick,dashed] (.5,-.25) -- (.5,1.25);
			\draw[red,fill=red] (.5, 1) circle (0.05);
			\draw[red,fill=red] (-.5,.5) circle (0.05);
			\draw[ggreen,fill=ggreen] (.5,0) circle (0.05);
			\draw[thick, ggreen] (.5,0) -- (0,0) -- (0, 1.25);
			\node[right] at (.5,1) {1};
			\node[left] at (-.5,.5) {2};
			\node[right] at (.5,0) {3};
		\end{tikzpicture}
		\qquad\qquad
		\begin{tikzpicture}[baseline]
			\node[left] at (-.75,.5) {$E_6=$};
			\draw[thick,dashed] (-.5,-.25) -- (-.5,1.25);
			\draw[thick,dashed] (.5,-.25) -- (.5,1.25);
			\draw[red,fill=red] (.5, 1) circle (0.05);
			\draw[thick,red] (.5,1) -- (.25, 1) -- (.25,1.25);
			\draw[red,fill=red] (-.5,.5) circle (0.05);
			\draw[ggreen,fill=ggreen] (.5,0) circle (0.05);
			\draw[thick, ggreen] (.5,0) -- (0,0) -- (0, 1.25);
			\node[right] at (.5,1) {1};
			\node[left] at (-.5,.5) {2};
			\node[right] at (.5,0) {3};
		\end{tikzpicture}
		\qquad\qquad
		\begin{tikzpicture}[baseline]
			\node[left] at (-.75,.5) {$E_7=$};
			\draw[thick,dashed] (-.5,-.25) -- (-.5,1.25);
			\draw[thick,dashed] (.5,-.25) -- (.5,1.25);
			\draw[red,fill=red] (.5, 1) circle (0.05);
			\draw[red,fill=red] (-.5,.5) circle (0.05);
			\draw[thick,red] (-.5,.5) -- (-.25,.5) -- (-.25,1.25);
			\draw[ggreen,fill=ggreen] (.5,0) circle (0.05);
			\draw[thick, ggreen] (.5,0) -- (0,0) -- (0, 1.25);
			\node[right] at (.5,1) {1};
			\node[left] at (-.5,.5) {2};
			\node[right] at (.5,0) {3};
		\end{tikzpicture}
		\qquad\qquad
		\begin{tikzpicture}[baseline]
			\node[left] at (-.75,.5) {$E_8=$};
			\draw[thick,dashed] (-.5,-.25) -- (-.5,1.25);
			\draw[thick,dashed] (.5,-.25) -- (.5,1.25);
			\draw[red,fill=red] (.5, 1) circle (0.05);
			\draw[thick,red] (.5,1) -- (.25,1) -- (.25,1.25);
			\draw[red,fill=red] (-.5,.5) circle (0.05);
			\draw[thick,red] (-.5,.5) -- (-.25,.5) -- (-.25,1.25);
			\draw[ggreen,fill=ggreen] (.5,0) circle (0.05);
			\draw[thick, ggreen] (.5,0) -- (0,0) -- (0, 1.25);
			\node[right] at (.5,1) {1};
			\node[left] at (-.5,.5) {2};
			\node[right] at (.5,0) {3};
		\end{tikzpicture}
	\end{align*}
Observe in terms of the recursive construction of $LR(\chi,\epsilon)$, the diagram $D_k$, $k=1, 2, 3, 4$ from Example \ref{2_node_ex} creates diagrams $E_{2k-1}$ and $E_{2k}$ in the present example.
\end{ex}

For fixed $\chi$ and $\epsilon$ we note that each $D\in LR_0(\chi,\epsilon)$ can be associated to a partition $\pi\in\mc{P}(n)$ by forming blocks according to which nodes are connected via chords in the diagram. Since $D\in LR_0(\chi,\epsilon)$ is completely determined by the connections between nodes, distinct diagrams yield distinct partitions. Moreover, if $\alpha\colon\{1,\ldots, n\}\rightarrow I\sqcup J$ and we define $\chi^\alpha$ by $\chi^\alpha(k)=\ell$ if $\alpha(k)\in I$ and $\chi^\alpha(k)=r$ if $\alpha(k)\in J$ then the partitions we obtain from $LR_0(\chi^\alpha,\epsilon)$ are elements of $BNC(\alpha)$. We denote by $BNC(\alpha,\epsilon)$ the partitions obtained from the diagrams in $LR_0(\chi^\alpha,\epsilon)$. It is not hard to see that given the diagram associated to some $\pi \in BNC(\alpha)$, there exists some shading $\epsilon$ such that $\pi\in BNC(\alpha,\epsilon)$. It then follows that
	\begin{align*}
		BNC(\alpha)=\bigcup_{\epsilon\in \{',''\}^n} BNC(\alpha,\epsilon)
	\end{align*}
As with $BNC(\alpha)$, we may denote $BNC(\alpha, \epsilon)$ by $BNC(\chi, \epsilon)$ when $\chi=\chi^\alpha$.


\begin{defi}
Suppose that $V$ and $W$ are blocks of some $\pi \in BNC(\chi)$.
Then $V$ and $W$ are said to be \emph{piled} if $\max\left(\min(V), \min(W)\right) \leq \min\left(\max(V), \max(W)\right)$.
In terms of the diagram corresponding to $\pi$, the spines of $V$ and $W$ are not entirely above or below each other; there is some horizontal level at which both are present.

Given blocks $V$ and $W$, a third block $U$ \emph{separates} $V$ from $W$ if it is piled with both, and its spine lies between the spines of $V$ and $W$.
Note that $V$ and $W$ need not be piled with each other to have a separator.
Equivalently, $U$ is piled with both $V$ and $W$, and there are $j, k \in U$ such that $s_\alpha^{-1}(V) \subseteq [s_\alpha^{-1}(j), s_\alpha^{-1}(k)]$ and $s_\alpha^{-1}(W) \cap [s_\alpha^{-1}(j), s_\alpha^{-1}(k)] = \emptyset$, or vice versa.
Given any three piled blocks, one always separates the other two.

Finally, piled blocks $V$ and $W$ are said to be \emph{tangled} if there is no block which separates them.
\end{defi}

\begin{ex}
Consider the following diagrams.
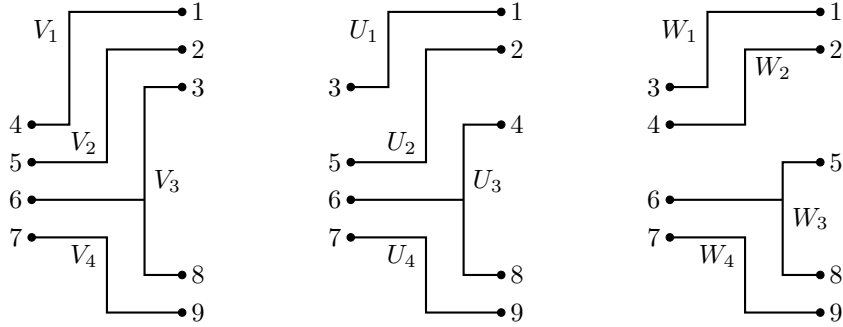
\[
		\begin{tikzpicture}[baseline]
			
			\draw[black,fill=black] (1, 2) circle (0.05);
			\draw[black,fill=black] (1, 1.5) circle (0.05);
			\draw[black,fill=black] (1, 1) circle (0.05);
			\draw[black,fill=black] (-1, .5) circle (0.05);
			\draw[black,fill=black] (-1, 0) circle (0.05);
			\draw[black,fill=black] (-1, -.5) circle (0.05);
			\draw[black,fill=black] (-1, -1) circle (0.05);
			\draw[black,fill=black] (1, -1.5) circle (0.05);
			\draw[black,fill=black] (1, -2) circle (0.05);
			
			\node[right] at (1, 2) {1};
			\node[right] at (1, 1.5) {2};
			\node[right] at (1, 1) {3};
			\node[left] at (-1, .5) {4};
			\node[left] at (-1, 0) {5};
			\node[left] at (-1, -.5) {6};
			\node[left] at (-1, -1) {7};
			\node[right] at (1, -1.5) {8};
			\node[right] at (1, -2) {9};

			\node[left] at (-.5, 1.75) {$V_1$};
			\draw[thick] (1, 2) -- (-.5, 2) -- (-.5, .5) -- (-1, .5);

			\node[left] at (0, .25) {$V_2$};
			\draw[thick] (1, 1.5) -- (0, 1.5) -- (0, 0) -- (-1, 0);
			
			\node[right] at (.5, -.25) {$V_3$};
			\draw[thick] (1, 1) -- (.5, 1) -- (.5, -1.5) -- (1, -1.5);
			\draw[thick] (-1, -.5) -- (.5, -.5);
			
			\node[left] at (0, -1.25) {$V_4$};
			\draw[thick] (-1, -1) -- (0, -1) -- (0, -2) -- (1, -2);
			
		\end{tikzpicture}
		\qquad\qquad
		\begin{tikzpicture}[baseline]
			
			\draw[black,fill=black] (1, 2) circle (0.05);
			\draw[black,fill=black] (1, 1.5) circle (0.05);
			\draw[black,fill=black] (-1, 1) circle (0.05);
			\draw[black,fill=black] (1, .5) circle (0.05);
			\draw[black,fill=black] (-1, 0) circle (0.05);
			\draw[black,fill=black] (-1, -.5) circle (0.05);
			\draw[black,fill=black] (-1, -1) circle (0.05);
			\draw[black,fill=black] (1, -1.5) circle (0.05);
			\draw[black,fill=black] (1, -2) circle (0.05);

			\node[right] at (1, 2) {1};
			\node[right] at (1, 1.5) {2};
			\node[left] at (-1, 1) {3};
			\node[right] at (1, .5) {4};
			\node[left] at (-1, 0) {5};
			\node[left] at (-1, -.5) {6};
			\node[left] at (-1, -1) {7};
			\node[right] at (1, -1.5) {8};
			\node[right] at (1, -2) {9};

			\node[left] at (-.5, 1.75) {$U_1$};
			\draw[thick] (1, 2) -- (-.5, 2) -- (-.5, 1) -- (-1, 1);

			\node[left] at (0, .25) {$U_2$};
			\draw[thick] (1, 1.5) -- (0, 1.5) -- (0, 0) -- (-1, 0);
			
			\node[right] at (.5, -.25) {$U_3$};
			\draw[thick] (1, .5) -- (.5, .5) -- (.5, -1.5) -- (1, -1.5);
			\draw[thick] (-1, -.5) -- (.5, -.5);
			
			\node[left] at (0, -1.25) {$U_4$};
			\draw[thick] (-1, -1) -- (0, -1) -- (0, -2) -- (1, -2);
			
		\end{tikzpicture}
		\qquad\qquad
		\begin{tikzpicture}[baseline]
			
			\draw[black,fill=black] (1, 2) circle (0.05);
			\draw[black,fill=black] (1, 1.5) circle (0.05);
			\draw[black,fill=black] (-1, 1) circle (0.05);
			\draw[black,fill=black] (-1, .5) circle (0.05);
			\draw[black,fill=black] (1, 0) circle (0.05);
			\draw[black,fill=black] (-1, -.5) circle (0.05);
			\draw[black,fill=black] (-1, -1) circle (0.05);
			\draw[black,fill=black] (1, -1.5) circle (0.05);
			\draw[black,fill=black] (1, -2) circle (0.05);

			\node[right] at (1, 2) {1};
			\node[right] at (1, 1.5) {2};
			\node[left] at (-1, 1) {3};
			\node[left] at (-1, .5) {4};
			\node[right] at (1, 0) {5};
			\node[left] at (-1, -.5) {6};
			\node[left] at (-1, -1) {7};
			\node[right] at (1, -1.5) {8};
			\node[right] at (1, -2) {9};

			\node[left] at (-.5, 1.75) {$W_1$};
			\draw[thick] (1, 2) -- (-.5, 2) -- (-.5, 1) -- (-1, 1);

			\node[right] at (0, 1.25) {$W_2$};
			\draw[thick] (1, 1.5) -- (0, 1.5) -- (0, 0.5) -- (-1, 0.5);
			
			\node[right] at (.5, -.75) {$W_3$};
			\draw[thick] (1, 0) -- (.5, 0) -- (.5, -1.5) -- (1, -1.5);
			\draw[thick] (-1, -.5) -- (.5, -.5);
			
			\node[left] at (0, -1.25) {$W_4$};
			\draw[thick] (-1, -1) -- (0, -1) -- (0, -2) -- (1, -2);
			
		\end{tikzpicture}
\]
In the first diagram, $V_2$ separates $V_1$ from $V_3$, and all three are piled with one another.
In the second diagram, $U_2$ still separates $U_1$ and $U_3$, but $U_1$ and $U_3$ are not piled with each other.
In the third diagram, there are no separators.

\end{ex}

\begin{defi}
Suppose $\pi, \sigma \in BNC(\chi)$ are such that $\pi \leq \sigma$.
We say $\pi$ is a \emph{lateral refinement} of $\sigma$ and write $\pi \lrleq \sigma$ if no two piled blocks in $\pi$ are contained in the same block of $\sigma$.

Lateral refinements correspond to making horizontal ``cuts'' along the spines of blocks of $\pi$, between their ribs.
\end{defi}

In the notation of Example \ref{3_node_ex}, $E_1$ is a lateral refinement of $E_3$ made by cutting the block $\{1,2\}$ in between node $1$ and node $2$.

\begin{lem}\label{lateral_refinement}
If $\pi\in BNC(\chi,\epsilon)$ then piled blocks of the same shade in $\pi$ must be separated.
Consequently, if $\sigma\in BNC(\alpha,\epsilon)$ and $\pi\leq \sigma$ then $\pi\lrleq\sigma$.
\end{lem}

\begin{proof}
Suppose $V_1$ and $V_2$ are piled blocks in $\pi \in BNC(\chi, \epsilon)$ which have the same shade.
Without loss of generality, $k := \max(V_2) < \max(V_1)$.
In the construction of the diagram generating $\pi$, when node $k$ is placed the nearest spine must be of a different shade as $k$ begins a new spine.
In particular, this spine sits between the spines of $V_1$ and $V_2$, and so its block is a separator.

If two blocks of the same in $\pi$ are piled, the above argument demonstrates that they are separated by a block of a different shade and so can't be joined in $\sigma$.
\end{proof}

\section{The Incident Algebra on Bi-Non-Crossing Partitions}
\label{sec:incalg}

\setcounter{thm}{0}

\begin{defi}
The lattice of bi-non-crossing partitions is
\[
BNC := \bigcup_{n\geq 1} \bigcup_{\chi : \{1, \ldots, n\} \to \{\ell, r\}} BNC(\chi)
\]
where the lattice structure on $BNC(\chi)$ is as above.
\end{defi}

Given any lattice, there is an algebra of functions associated to the lattice.
\begin{defi}
The incident algebra on $BNC$, denoted $IA(BNC)$, is all functions of the form
\[
f : \bigcup_{n\geq 1} \left(  \bigcup_{\chi : \{1, \ldots, n\} \to \{\ell, r\}} BNC(\chi)\times BNC(\chi)\right) \to \mathbb{C}
\]
such that $f(\pi,\sigma) = 0$ if $\pi \nleq \sigma$ equipped with pointwise addition and a convolution product defined by
\[
(f * g)(\pi, \sigma) = \sum_{\pi \leq \rho \leq \sigma} f(\pi, \rho)g(\rho,\sigma)
\]
for all $\pi, \sigma \in BNC(\chi)$ and $f, g \in IA(BNC)$.
\end{defi}
It is elementary to show that $IA(BNC)$ is an algebra and thus $(f*g)*h = f*(g*h)$.

\subsection{Multiplicative functions on the incident algebra}

In order to construct the notion of multiplicative functions on $BNC$, it is necessary to identify the lattice structure of an interval as a product of full intervals.
\begin{prop}
\label{breakingintervalsintofullintervals}
Let $\pi, \sigma \in BNC(\chi)$ be such that $\pi \leq \sigma$.
The interval
\[
[\pi, \sigma] = \{\rho \in BNC(\chi) \, \mid \, \pi \leq \rho \leq \sigma\}
\]
can be associated to a product of full lattices
\[
\prod^k_{j=1} BNC(\beta_k)
\]
for some $\beta_k : \{1,\ldots, m_k\} \to \{\ell, r\}$ so that the lattice structure is preserved.
\end{prop}

\begin{proof}
The idea behind the decomposition is to take $\pi$ and $\sigma$, view $\pi$ and $\sigma$ as elements of $NC(n)$ by applying $s_\chi^{-1}$, and using the decomposition of intervals in $NC(n)$ given in Proposition 1 of \cite{speicher1994} while maintaining the notion of left and right nodes.

First write $\sigma = \{W_1, \ldots, W_k\}$.
Let $\pi_j $ and $\sigma_j$ be the restrictions of $\pi$ and $\sigma$ to $W_j$.
Then we decompose $[\pi, \sigma]$ into
\[
\prod^k_{j=1} [\pi_j, \sigma_j].
\]
Note each $\sigma_j$ is a full bi-non-crossing partition corresponding to some $\gamma_j : \{1, \ldots, n_j\} \to \{\ell, r\}$ so one may reduce to intervals of the form $[\pi, 1_{\chi}]$.

For a fixed $\chi : \{1, \ldots, n\} \to \{\ell, r\}$, a modification to the recursive argument of Proposition 1 of \cite{speicher1994} under the identification of $BNC(\chi)$ with $NC(n)$ will be described.
First, viewing $\pi \in NC(n)$, examine whether $\pi$ has a block $V=\{k_1 < k_2 < \cdots < k_m\}$ containing non-consecutive elements; that is, there exists an index $t$ such that $k_t +1 \neq k_{t+1}$.
If so, the recursive argument of Proposition 1 of \cite{speicher1994} would decompose $[\pi, 1_\chi]$ into the product of two intervals (removing any trivial intervals that occur): one corresponding to taking $[\pi, 1_\chi]$ and removing all nodes strictly between $k_{t}$ and $k_{t+1}$; and the other corresponding to taking only the nodes strictly between $k_{t}$ and $k_{t+1}$ and adding an isolated node on the left.
The only change made to accommodate $BNC$ is that the isolated node for the second interval should be added to the top left of the bi-non-crossing diagram if the lower of the two nodes of the original diagram corresponding to $k_{t}$ and $k_{t+1}$ is on the left and otherwise on the top right.
For example:
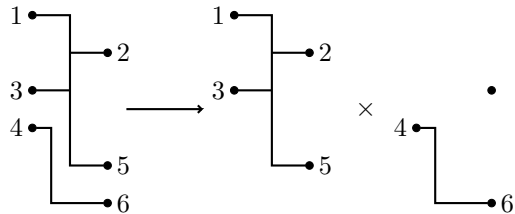
\begin{align*}
	\begin{tikzpicture}[baseline]
	\node[left] at (-.5, 2.5) {1};
	\draw[fill=black] (-.5,2.5) circle (0.05);
	\node[right] at (.5, 2) {2};
	\draw[fill=black] (.5,2) circle (0.05);
	\node[left] at (-.5, 1.5) {3};
	\draw[fill=black] (-.5,1.5) circle (0.05);
	\node[left] at (-.5, 1) {4};
	\draw[fill=black] (-.5,1) circle (0.05);
	\node[right] at (.5, .5) {5};
	\draw[fill=black] (.5,.5) circle (0.05);
	\node[right] at (.5,0) {6};
	\draw[fill=black] (.5,0) circle (0.05);
	\draw[thick] (-.45,2.5) -- (0,2.5) -- (0, 0.5) -- (.45,0.5);
	\draw[thick] (-.45,1) -- (-.25,1) -- (-.25, 0) -- (.45,0);
	\draw[thick] (-.45,1.5) -- (0,1.5);
	\draw[thick] (.45,2) -- (0,2);
	\draw[thick] (.75, 1.25) -- (1.75, 1.25) -- (1.7,1.2);
	\draw[thick] (1.75, 1.25) -- (1.7,1.3);
	\end{tikzpicture}
	\begin{tikzpicture}[baseline]
	\node[left] at (-.5, 2.5) {1};
	\draw[fill=black] (-.5,2.5) circle (0.05);
	\node[right] at (.5, 2) {2};
	\draw[fill=black] (.5,2) circle (0.05);
	\node[left] at (-.5, 1.5) {3};
	\draw[fill=black] (-.5,1.5) circle (0.05);
	\node[right] at (.5, .5) {5};
	\draw[fill=black] (.5,.5) circle (0.05);
	\draw[thick] (-.45,2.5) -- (0,2.5) -- (0, 0.5) -- (.45,0.5);
	\draw[thick] (-.45,1.5) -- (0,1.5);
	\draw[thick] (.45,2) -- (0,2);
	\node[above] at (1.25, 1) {$\times$};
	\end{tikzpicture}
	\begin{tikzpicture}[baseline]
	\draw[fill=black] (.5,1.5) circle (0.05);
	\node[left] at (-.5, 1) {4};
	\draw[fill=black] (-.5,1) circle (0.05);
	\node[right] at (.5,0) {6};
	\draw[fill=black] (.5,0) circle (0.05);
	\draw[thick] (-.45,1) -- (-.25,1) -- (-.25, 0) -- (.45,0);
	\end{tikzpicture}
\end{align*}
Note that the first term in the product will be ignored as it is a full partition.

 This recursive process eventually terminates leaving only partitions $\pi$ such that the blocks of $\sigma_\chi^{-1}\cdot\pi$ are intervals.
For such a bi-non-crossing partition, we associate the zero bi-non-crossing partition corresponding to keeping only the lowest node of each block.
For example:
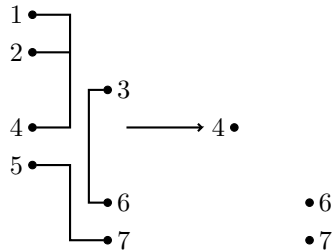
\begin{align*}
	\begin{tikzpicture}[baseline]
	\node[left] at (-.5, 3) {1};
	\draw[fill=black] (-.5,3) circle (0.05);
	\node[left] at (-.5, 2.5) {2};
	\draw[fill=black] (-.5,2.5) circle (0.05);
	\node[right] at (.5, 2) {3};
	\draw[fill=black] (.5,2) circle (0.05);
	\node[left] at (-.5, 1.5) {4};
	\draw[fill=black] (-.5,1.5) circle (0.05);
	\node[left] at (-.5, 1) {5};
	\draw[fill=black] (-.5,1) circle (0.05);
	\node[right] at (.5, .5) {6};
	\draw[fill=black] (.5,.5) circle (0.05);
	\node[right] at (.5,0) {7};
	\draw[fill=black] (.5,0) circle (0.05);
	\draw[thick] (-.45,3) -- (0,3) -- (0, 1.5) -- (-.45,1.5);
	\draw[thick] (-.45,2.5) -- (0,2.5);
	\draw[thick] (-.45,1) -- (0,1) -- (0, 0) -- (.45,0);
	\draw[thick] (.45,2) -- (.25,2) -- (.25,.5) -- (.45,.5);
	\draw[thick] (.75, 1.5) -- (1.75, 1.5) -- (1.7,1.45);
	\draw[thick] (1.75, 1.5) -- (1.7,1.55);
	\end{tikzpicture}
	\begin{tikzpicture}[baseline]
	\node[left] at (-.5, 1.5) {4};
	\draw[fill=black] (-.5,1.5) circle (0.05);
	\node[right] at (.5, .5) {6};
	\draw[fill=black] (.5,.5) circle (0.05);
	\node[right] at (.5,0) {7};
	\draw[fill=black] (.5,0) circle (0.05);
	\end{tikzpicture}
\end{align*}
Thus we have reduced $[\pi,\sigma]$ to products of full lattices in $BNC$.
\end{proof}
Note that as in Proposition 1 of \cite{speicher1994} we make no claim that this association is unique.
However, this ambiguity does not affect the following computations.
\begin{defi}
A function $f \in IA(BNC)$ is said to be \emph{multiplicative} if whenever $\pi, \sigma \in BNC(\chi)$ are such that
\[
[\pi, \sigma] \leftrightarrow \prod^k_{j=1} BNC(\beta_k)
\]
for some $\beta_k : \{1,\ldots, m_k\} \to \{\ell, r\}$, then
\[
f(\pi, \sigma) = \prod^k_{j=1} f(0_{\beta_k}, 1_{\beta_k}).
\]

For a multiplicative function $f \in IA(BNC)$, we will call the collection $\{f([0_{\chi}, 1_\chi]) \, \mid \, n\geq 1, \chi: \{1,\ldots, n\} \to \{\ell, r\}\} \subseteq \mathbb{C}$ the \emph{multiplicative net} associated to $f$.
Note that for any net $\Lambda = \{a_\chi \, \mid \, n\geq 1, \chi: \{1,\ldots, n\} \to \{\ell, r\}\} \subseteq \mathbb{C}$ there is precisely one multiplicative function $f$ with multiplicative sequence $\Lambda$.
\end{defi}
\begin{lem}
If $f, g\in IA(BNC)$ are multiplicative, then $f * g$ is multiplicative.
\end{lem}

See Proposition 2 of \cite{speicher1994} for a proof of the above.

\begin{rem}
\label{mobiusfunctionremarks}
There are three special multiplicative functions to consider; namely
\[
\delta_{BNC}(\pi, \sigma)  = \left\{
\begin{array}{ll}
1 & \mbox{if } \pi = \sigma \\
0 & \mbox{otherwise}
\end{array} \right.
\]
which is called the delta function on $BNC$ and is the identity element in $IA(BNC)$,
\[
\zeta_{BNC}(\pi, \sigma)  = \left\{
\begin{array}{ll}
1 & \mbox{if } \pi \leq \sigma \\
0 & \mbox{otherwise}
\end{array} \right.
\]
which is called the zeta function on $BNC$, and $\mu_{BNC}$ which is called the M\"{o}bius function on $BNC$ which is defined such that
\[
\mu_{BNC} * \zeta_{BNC} = \zeta_{BNC} * \mu_{BNC} = \delta_{BNC}
\]
(as it is clear that $\zeta_{BNC}$ a left and right (and thereby a two-sided) inverse can be recursively defined).
It is clear that $\delta_{BNC}$ is multiplicative with $\delta_{BNC}(0_\chi, 1_\chi)$ being one if $n = 1$ and zero otherwise, and $\zeta_{BNC}$ is multiplicative with  $\zeta_{BNC}(0_\chi, 1_\chi) = 1$ for all $\chi$.
In addition, one can verify that $\mu_{BNC}$ is multiplicative and for any $\pi, \sigma \in BNC(\chi)$
\[
\mu_{BNC}(\pi, \sigma) = \mu(s^{-1}_\chi \cdot \pi, s^{-1}_\chi \cdot \sigma),
\]
where $\mu$ is the M\"{o}bius function in \cite{speicher1994}.
In addition, if $\pi, \sigma \in BNC(\chi)$ and we view $\pi$ and $\sigma$ as elements of $NC(n)$ as in the first paragraph of Proposition \ref{breakingintervalsintofullintervals}, one obtains
by construction.
\end{rem}

%

\begin{rem}
\label{moebius_power}
To consolidate the above with Subsection \ref{mastnak_nica}, for $T_1, \ldots, T_n$ in a non-commutative probability space $(\mathcal{A}, \varphi)$ and $\pi \in BNC(\chi)$ where $\chi : \{1, \ldots, n\} \to \{\ell, r\}$ and $V_t = \{k_{t, 1} < \cdots < k_{t, m_t}\}$ for $t \in \{1, \ldots, k\}$ being the blocks of $\pi$, we define
\[
\varphi_\pi(T_1, \ldots, T_n) := \prod^k_{t=1} \varphi(T_{k_{t,1}} \cdots T_{k_{t, m_t}})
\]
and
\[
\kappa_\pi(T_1, \ldots, T_n) := \sum_{\sigma \in BNC(\chi), \sigma \leq\pi} \varphi_\sigma(T_1, \ldots, T_n) \mu_{BNC}(\sigma, \pi).
\]
Then, as in \cite{speicher1994}, one can show that
\[
\kappa_\pi(T_1, \ldots, T_n) = \prod^k_{t=1} \kappa_{\pi|_{V_t}}(T_{k_{t,1}} \cdots T_{k_{t, m_t}})
\]
where $\kappa_{\pi|_{V_t}}$ should be thought of as the (single block) partition induced by the block $V_t$ of $\pi$, and
\[
\varphi(T_1 \ldots T_n) = \sum_{\pi \in BNC(\chi)} \kappa_\pi(T_1, \ldots, T_n).
\]
In particular, $\kappa_{1_\chi} = \kappa_\chi$ are the bi-free cumulant functions of Definition 5.2 of \cite{1312.0269}.

For a two-faced family $z = ((z_i)_{i\in I}, (z_j)_{j\in J})$, $\alpha : \{1, \ldots, n\} \to I \sqcup J$, and $\pi \in BNC(\alpha)$ we denote
\[\varphi_\pi(z) := \varphi_\pi(z_{\alpha(1)},\ldots,z_{\alpha(n)})
\qquad\text{ and }\qquad
\kappa_\pi(z) := \kappa_\pi(z_{\alpha(1)}, \ldots, z_{\alpha(n)}).\]
In particular, $\varphi_{1_\alpha}(z) = \varphi_\alpha(z)$ and $\kappa_{1_\alpha}(z) = \kappa_\alpha(z)$.
When the faces consist of a single element each, say $z_\ell$ and $z_r$, we define the above quantities for $\chi : \{1, \ldots, n\} \to \{\ell, r\}$ replacing $\alpha$.
In this case we let $m_z,\kappa_z \in IA(BNC)$ be the multiplicative functions with multiplicative nets $(\varphi_{\chi}(z))_\chi$ and $(\kappa_{\chi}(z))_\chi$, respectively.
We call $m_z$ the \emph{moment function} and $\kappa_z$ the \emph{bi-free cumulant function}.
Thus the formulae $m_z*\mu_{BNC} = \kappa_z$ and $\kappa_z * \zeta_{BNC} = m_z$ are obtained.
\end{rem}


\section{Unifying Bi-Free Independence}
\label{sec:unify}

\subsection{Computing bi-free moments}

We will demonstrate how the partitions of $BNC(\chi,\epsilon)$ may be used to compute joint moments of a bi-free pair of two-faced families.

Fix $\chi\colon\{1,\ldots,n\}\rightarrow \{\ell,r\}$ and a shading $\epsilon\in\{',''\}^n$, and let $T_k\in \mc{L}(\mc{X}^{\epsilon_k})$. Given $D\in LR(\chi,\epsilon)$, we will assign a vector weight $\psi(D;T_1,\ldots, T_n)\in \mc{X}$ to $D$. Define $\mu\in\{\lambda,\rho\}^n$ by $\mu_j=\lambda$ if $\chi(j)=\ell$ and $\mu_j=\rho$ if $\chi(j)=r$. Let $V=\{k_1<\cdots<k_r\}$ be a block in $D$ and let $\epsilon(V):=\epsilon_{k_1}=\cdots=\epsilon_{k_r}$. If the spine of $V$ is not connected to the top gap then $V$ contributes a scalar factor of
	\begin{align*}
		\psi(V;T_1,\ldots, T_n):=\psi^{\epsilon(V)}\left( T_{k_1}(1-p^{\epsilon(V)}) T_{k_2} \cdots (1-p^{\epsilon(V)}) T_{k_r} \xi^{\epsilon(V)}\right)
	\end{align*}
to $\psi(D; T_1,\ldots, T_n)$. If the spine does reach the top gap then it contributes a vector factor of
	\begin{align*}
		\psi(V;T_1,\ldots, T_n):=(1-p^{\epsilon(V)})T_{k_1}(1-p^{\epsilon(V)}) T_{k_2} \cdots (1-p^{\epsilon(V)}) T_{k_r} \xi^{\epsilon(V)}.
	\end{align*}
Then $\psi(D;T_1,\ldots, T_n)$ is the product of the scalar factors and the tensor product of the vector factors where the order in the tensor product is determined by the left to right order of the spines reaching the top gap. If all contributions are scalar factors then we multiply this with the state-vector $\xi$, thinking of it as the ``empty tensor word.''

Recalling Example \ref{3_node_ex}, we see that
	\begin{align*}
		\psi(E_3;T_1,T_2,T_3) &= \psi'(T_1(1-p')T_2\xi')\psi''(T_3\xi'')\xi,\qquad \text{ while }\\
		\psi(E_8;T_1,T_2,T_3) &= (1-p')T_2\xi'\otimes (1-p'')T_3\xi''\otimes (1-p')T_1\xi'
	\end{align*}

\begin{prop}
Fix $\chi\colon\{1,\ldots, n\}\rightarrow\{\ell,r\}$ and a shading $\epsilon\in\{',''\}^n$. Let $\mu\in\{\lambda,\rho\}^n$ be as above. If $T_j\in \mc{L}(\mc{X}^{\epsilon_j})$ for $j=1,\ldots, n$, then following formula holds:
	\begin{align}\label{range_computation}
		\mu_1(T_1)\cdots \mu_n(T_n)\xi= \sum_{D\in LR(\chi,\epsilon)} \psi(D;T_1,\ldots, T_n).
	\end{align}
Moreover,
	\begin{align}\label{moment_formula}
		\varphi(\mu_1(T_1)\cdots \mu_n(T_n))=\sum_{\pi \in BNC(\chi)}\left[\sum_{\substack{\sigma\in BNC(\chi,\epsilon)\\ \sigma\lrgeq \pi}}(-1)^{|\pi|-|\sigma|} \right]  \varphi_\pi(T_1,\ldots,T_n)
	\end{align}

\end{prop}

\begin{proof}
We establish (\ref{range_computation}) via induction on $n$. The base case is clear, so we assume the formula holds for $n-1$ operators and apply it as
	\begin{align*}
		\mu_2(T_2)\cdots \mu_n(T_n)\xi =\sum_{D\in LR(\chi_0,\epsilon_0)} \psi(D;T_2,\ldots, T_n),
	\end{align*}
where $\chi_0=\chi\mid_{\{2,\ldots,n\}}$ and $\epsilon_0=(\epsilon_2,\ldots, \epsilon_n)$. Fix a $D\in LR(\chi_0,\epsilon_0)$ and assume $\mu_1=\lambda$. Either there is a leftmost spine in $D$ of the shade $\epsilon_1$ reaching the top gap, or there is not (meaning either the nearest spine is the wrong shade or that $D$ has no spines reaching the top gap). In the former case, writing $\psi(D;T_2,\ldots, T_n)$ as $x_1\otimes\cdots\otimes x_m$ this implies $x_1\in \mc{X}^{\epsilon_1}$. Hence
	\begin{align*}
		\lambda(T_1)x_1\otimes\cdots\otimes x_m &= \psi(T_1(1-p^{\epsilon_1})x_1) x_2\otimes\cdots\otimes x_m + (1-p^{\epsilon_1})T_1(1-p^{\epsilon_1})x_1\otimes x_2\otimes\cdots\otimes x_m\\
			&=\psi(D_1;T_1,\ldots, T_n)+\psi(D_2;T_1,\ldots, T_n),
	\end{align*}
where $D_1,D_2\in LR(\chi,\epsilon)$ are the diagrams constructed from $D$ by adding rib and, respectively, terminating the leftmost spine in $D$ at the new top node or extending the leftmost spine in $D$.

If there is no leftmost spine of the same shade as $\epsilon_1$ then $\psi(D;T_2,\ldots, T_n)$ can be written in the same way as before except $x_1\not\in \mc{X}^{\epsilon_1}$ (if $D$ has no spines reaching the top gap then this is simply a scalar multiple of $\xi$). Hence
	\begin{align*}
		\lambda(T_1)x_1\otimes\cdots\otimes x_m&= \psi^{\epsilon_1}(T_1\xi^{\epsilon_1}) x_1\otimes\cdots \otimes x_m + (1-p^{\epsilon_1})T_1\xi^{\epsilon_1}\otimes x_1\otimes\cdots\otimes x_m\\
			&=\psi(E_1;T_1,\ldots, T_n) + \psi(E_2; T_1,\ldots, T_n),
	\end{align*}
where $E_1, E_2\in LR(\chi,\epsilon)$ are the diagrams constructed from $D$ by, respectively, leaving the new top node isolated or adding a new rib and spine.

Since every $D\in LR(\chi,\epsilon)$ is constructed from exactly one diagram in $LR(\chi_0,\epsilon_0)$ we have
	\begin{align*}
		\lambda(T_1)\mu_2(T_2)\cdots \mu_n(T_n)\xi = \sum_{D\in LR(\chi,\epsilon)} \psi(D;T_1,\ldots T_n).
	\end{align*}
The case $\mu_1=\rho$ is exactly the same upon replacing ``leftmost'' with ``rightmost'' and the considerations about $x_1$ with ones about $x_m$.

Now, $\varphi(\mu_1(T_1)\cdots \mu_n(T_n))$ is given by applying $\psi$ to the left side of (\ref{range_computation}).
So only the terms on the right whose vector parts are $\xi$ will survive, that is, the terms corresponding to $E\in LR_0(\chi,\epsilon)$. Fix such a diagram and let $\sigma\in BNC(\chi,\epsilon)$ be the corresponding partition. We examine
	\begin{align*}
		\psi(E;T_1,\ldots, T_n) = \prod_{W\in \sigma} \psi(W;T_1,\ldots, T_n).
	\end{align*}
For $W=\{l_1<\cdots< l_s\}\in \sigma$ we have
	\begin{align*}
		\psi(W;T_1,\ldots, T_n)&=\psi^{\epsilon(W)}\left(T_{l_1}(1-p^{\epsilon(W)}) T_{l_2} \cdots (1-p^{\epsilon(V)}) T_{l_s}\xi^{\epsilon(V)} \right)\xi\\
					&=\sum_{1\leq q_1<\cdots <q_m\leq s-1} (-1)^m \varphi^{\epsilon(W)}(T_{l_1}\cdots T_{l_{q_1}})\cdots\varphi^{\epsilon(V)}( T_{l_{q_m+1}}\cdots T_{l_s})\xi.
	\end{align*}
Each term in the last sum corresponds to a lateral refinement $\pi_W=\{V_1,\ldots, V_{m+1}\}$ of $W$, weighted by $(-1)^{|\pi_W|-|W|}$. As any lateral refinement of $\sigma$ is simply a collection of lateral refinements of its individual blocks, we see that $\pi=\bigcup_{W\in \sigma} \pi_W$ is a lateral refinement of $\sigma$. The overall weight associated to $\pi$ is $\prod_{W\in \sigma} (-1)^{|\pi_W|-|W|} = (-1)^{|\pi|-|\sigma|}$. Thus we obtain
	\begin{align*}
		\psi(E;T_1,\ldots, T_n)= \sum_{\substack{\pi\in BNC(\chi)\\ \pi\lrleq \sigma}} (-1)^{|\pi|-|\sigma|} \varphi_\pi(T_1,\ldots, T_n).
	\end{align*}
Summing over $E\in LR_0(\chi,\epsilon)$ (or equivalently $\sigma\in BNC(\chi,\epsilon)$) and reversing the order of the two summations yields (\ref{moment_formula}).
\end{proof}

\begin{cor}\label{bi-free_moment_formula}
Let $z'$ and $z''$ be a pair of two-faced families in $(\mc{A},\varphi)$. Then $z'$ and $z''$ are bi-free if and only if for every map $\alpha\colon\{1,\ldots, n\}\rightarrow I\sqcup J$ and $\epsilon\in\{',''\}^n$ we have
	\begin{align}\label{bi-free_moment_formula_formula}
		\varphi_\alpha\left(z^\epsilon\right)= \sum_{\pi\in BNC(\alpha)} \left[\sum_{\substack{\sigma\in BNC(\alpha,\epsilon)\\ \sigma\lrgeq \pi}} (-1)^{|\pi|-|\sigma|}\right] \varphi_\pi(z^\epsilon),
	\end{align}
where $z^\epsilon = \left(z_{\alpha(1)}^{\epsilon_1},\ldots,z_{\alpha(n)}^{\epsilon_n}\right)$.
\end{cor}
\begin{proof}
If $z'$ and $z''$ are bi-free then this immediately follows by applying the previous proposition to the representation guaranteed by the definition of bi-freeness.\par
Conversely, suppose $z'$ and $z''$ satisfy (\ref{bi-free_moment_formula_formula}) for each $\alpha$ and $\epsilon$. As in the proof of Proposition 2.9 of \cite{voiculescu2013freei}, we consider the universal representations of $z'$ and $z''$. That the joint representation in their free product is the same as the joint representation of $z'$ and $z''$ follows precisely from (\ref{bi-free_moment_formula_formula}).
\end{proof}

\subsection{Summation considerations}

For $\chi : \{1, \ldots, n\} \to \{\ell, r\}$, $\epsilon \in \{', ''\}^n$, and $\pi\in BNC(\chi)$, we will write $\pi \leq \epsilon$ where we think of $\epsilon$ as the induced partition in $\mathcal{P}(n)$.
\begin{prop}
\label{annoyingsumresult}
Let $\chi : \{1, \ldots, n\} \to \{\ell, r\}$ and $\epsilon \in \{', ''\}^n$.
Then for every $\pi \in BNC(\chi)$ such that $\pi \leq \epsilon$,
\[
\sum_{\substack{\sigma \in BNC(\chi, \epsilon) \\ \sigma \lrgeq \pi}} (-1)^{|\pi| - |\sigma|} = \sum_{\substack{\sigma \in BNC(\chi) \\ \pi \leq \sigma \leq \epsilon}} \mu_{BNC}(\pi, \sigma).
\]
\end{prop}
To prove Proposition \ref{annoyingsumresult} we will appeal to free probability to handle the following case and reduce all others to it.
\begin{lem}
\label{annoyingsumfreecase}
Let $\chi : \{1, \ldots, n\} \to \{\ell, r\}$ with $\chi \equiv \ell$ and $\epsilon \in \{', ''\}^n$.
Then for every $\pi \in BNC(\chi)$ such that $\pi \leq \epsilon$,
\[
\sum_{\substack{\sigma \in BNC(\chi,\epsilon) \\ \sigma \lrgeq \pi}} (-1)^{|\pi| - |\sigma|} = \sum_{\substack{\sigma \in BNC(\chi) \\ \pi \leq \sigma \leq \epsilon}} \mu_{BNC}(\pi, \sigma).
\]
\end{lem}
\begin{proof}
Let $\{X'_1, ..., X'_n\}$ and $\{X''_1, ..., X''_n\}$ be freely independent sets.
Note by Proposition 2.15b of \cite{voiculescu2013freei} these sets can be viewed as a bi-free pair of two faced families $X'$ and $X''$ with trivial right faces. Hence, by Corollary \ref{bi-free_moment_formula},
\[
\varphi\left(X^{\epsilon_1}_1 \cdots X^{\epsilon_n}_n\right)  = \sum_{\pi \in BNC(\chi)} \left(\sum_{\substack{\sigma \in BNC(\chi, \epsilon) \\ \sigma \lrgeq \pi}} (-1)^{|\pi| - |\sigma|}\right) \varphi_\pi\left(X^{\epsilon_1}_1, \ldots, X^{\epsilon_n}_n\right)
\]
Since $\chi \equiv \ell$, $BNC(\chi) = NC(n)$.
Thus, since $\{X'_1, ..., X'_n\}$ and $\{X''_1, ..., X''_n\}$ are free,
\begin{align*}
\varphi\left(X^{\epsilon_1}_1 \cdots X^{\epsilon_n}_n\right)
 &=  \sum_{\sigma \in BNC(\chi)}  \kappa_\sigma\left(X^{\epsilon_1}_1, \ldots, X^{\epsilon_n}_n\right)\\
 &=  \sum_{\substack{\sigma \in BNC(\chi)\\ \sigma \leq \epsilon}} \kappa_\sigma\left(X^{\epsilon_1}_1, \ldots, X^{\epsilon_n}_n\right)\\
 &=  \sum_{\substack{\sigma \in BNC(\chi)\\ \sigma \leq \epsilon}} \sum_{\substack{\pi \in BNC(\chi) \\ \pi \leq \sigma}}  \mu(\pi, \sigma) \varphi_\pi\left(X^{\epsilon_1}_1, \ldots, X^{\epsilon_n}_n\right)\\
 &=  \sum_{\substack{\pi \in BNC(\chi)\\ \pi \leq \epsilon}} \left(\sum_{\substack{\sigma \in BNC(\chi) \\ \pi\leq \sigma\leq \epsilon}} \mu(\pi, \sigma)\right)\varphi_\pi\left(X^{\epsilon_1}_1, \ldots, X^{\epsilon_n}_n\right).
\end{align*}
Since these expressions agree for any selection of $\{X'_1, ..., X'_n\}$ and $\{X''_1, ..., X''_n\}$ that are freely independent, by selecting $\{X'_1, ..., X'_n\}$ and $\{X''_1, ..., X''_n\}$ that are free and such that $\varphi_\pi\left(X^{\epsilon_1}_1, \ldots, X^{\epsilon_n}_n\right)$ is non-zero for precisely one $\pi$, the desired sums are obtained to be equal (as $\mu = \mu_{BNC}$ in this setting).
\end{proof}
We will use Lemma \ref{annoyingsumfreecase} to show that the desired equations in Proposition \ref{annoyingsumresult} hold.
To do so, we will show that an arbitrary bi-non-crossing partition can be obtained by a sequence of steps, preserving the summations in Proposition \ref{annoyingsumresult}, applied to a partition with all left nodes.
\begin{lem}
\label{changingalefttoaright}
Let $\chi : \{1, \ldots, n\} \to \{\ell, r\}$ with $\chi(n) = \ell$, $\epsilon \in \{', ''\}^n$, and $\pi \in BNC(\chi)$ be such that $\pi \leq \epsilon$.
Let $\hat{\chi} : \{1, \ldots, n\} \to \{\ell, r\}$ be such that
\[
\hat{\chi}(t) = \left\{
\begin{array}{ll}
\chi(t) & \mbox{if } t \neq n \\
r & \mbox{if } t= n
\end{array} \right. ,
\]
and let $\hat{\pi} \in BNC(\hat{\chi})$ be the unique shaded bi-non-crossing partition with the same blocks as $\pi$ (note $\hat{\pi} \leq \epsilon$ by construction).
Then
\[
\sum_{\substack{\sigma \in BNC(\chi, \epsilon) \\ \sigma \lrgeq \pi}} (-1)^{|\pi| - |\sigma|} = \sum_{\substack{\hat\sigma \in BNC(\hat{\chi}, \epsilon) \\ \hat\sigma \lrgeq \hat\pi}} (-1)^{|\hat\pi| - |\hat\sigma|}
\]
and
\[
\sum_{\substack{\sigma \in BNC(\chi) \\ \pi \leq \sigma \leq \epsilon}} \mu_{BNC}(\pi, \sigma) = \sum_{\substack{\hat\sigma \in BNC(\hat{\chi}) \\ \hat\pi \leq \hat\sigma \leq \epsilon}} \mu_{BNC}(\hat\pi, \hat\sigma).
\]
\end{lem}
\begin{proof}
It is clear that the operator which takes an element $\sigma \in BNC(\chi, \epsilon)$ and constructs an element $\hat\sigma \in BNC(\hat{\chi}, \epsilon)$ with the same blocks as $\sigma$ corresponds to taking the bottom node of $\sigma$ which is on the left and placing this node on the right (keeping all strings connected).
 For example, consider the following diagrams.
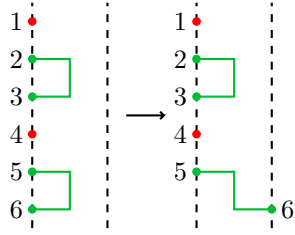
\begin{align*}
		\begin{tikzpicture}[baseline]
			\draw[thick,dashed] (-.5,-.25) -- (-.5,2.75);
			\draw[thick,dashed] (.5,-.25) -- (.5,2.75);
			\node[left] at (-.5,2.5) {1};
			\draw[red,fill=red] (-.5,2.5) circle (0.05);
			\node[left] at (-.5,2) {2};
			\draw[ggreen,fill=ggreen] (-.5,2) circle (0.05);
			\node[left] at (-.5,1.5) {3};
			\draw[ggreen,fill=ggreen] (-.5,1.5) circle (0.05);
			\node[left] at (-.5,1) {4};
			\draw[red,fill=red] (-.5,1) circle (0.05);
			\node[left] at (-.5,.5) {5};
			\draw[ggreen,fill=ggreen] (-.5,.5) circle (0.05);
			\node[left] at (-.5,0) {6};
			\draw[ggreen,fill=ggreen] (-.5,0) circle (0.05);
			\draw[thick, ggreen] (-.5, 0) -- (0,0) -- (0,.5) -- (-.5, .5);
			\draw[thick, ggreen] (-.5, 2) -- (0,2) -- (0,1.5) -- (-.5, 1.5);
			\draw[thick] (.75, 1.25) -- (1.25, 1.25) -- (1.2, 1.2);
			\draw[thick] (1.25, 1.25) -- (1.2, 1.3);
		\end{tikzpicture}
		\begin{tikzpicture}[baseline]
			\draw[thick,dashed] (-.5,-.25) -- (-.5,2.75);
			\draw[thick,dashed] (.5,-.25) -- (.5,2.75);
			\node[left] at (-.5,2.5) {1};
			\draw[red,fill=red] (-.5,2.5) circle (0.05);
			\node[left] at (-.5,2) {2};
			\draw[ggreen,fill=ggreen] (-.5,2) circle (0.05);
			\node[left] at (-.5,1.5) {3};
			\draw[ggreen,fill=ggreen] (-.5,1.5) circle (0.05);
			\node[left] at (-.5,1) {4};
			\draw[red,fill=red] (-.5,1) circle (0.05);
			\node[left] at (-.5,.5) {5};
			\draw[ggreen,fill=ggreen] (-.5,.5) circle (0.05);
			\node[right] at (.5,0) {6};
			\draw[ggreen,fill=ggreen] (.5,0) circle (0.05);
			\draw[thick, ggreen] (.5, 0) -- (0,0) -- (0,.5) -- (-.5, .5);
			\draw[thick, ggreen] (-.5, 2) -- (0,2) -- (0,1.5) -- (-.5, 1.5);
		\end{tikzpicture}
\end{align*}
Such operation is clearly a bijection, maps $BNC(\chi, \epsilon)$ to $BNC(\hat{\chi}, \epsilon)$, $(-1)^{|\pi| - |\sigma|} = (-1)^{|\hat\pi| - |\hat\sigma|}$, and $\sigma \lrgeq \pi$ if and only if $\hat\sigma \lrgeq \hat\pi$.
Hence the first equation holds.
Similarly, by Remarks \ref{mobiusfunctionremarks}, it is clear that the second equation holds.
\end{proof}
\begin{lem}
\label{changingorderofaleftandright}
Let $\chi : \{1, \ldots, n\} \to \{\ell, r\}$ be such that there exists a $k \in \{1,\ldots, n-1\}$ such that $\chi(k) = \ell$ and $\chi(k+1) = r$, $\epsilon \in \{', ''\}^n$, and $\pi \in BNC(\chi)$ be such that $\pi \leq \epsilon$.
 Fix $k \in \{1,\ldots, n-1\}$.
Let $\hat{\epsilon} \in \{', ''\}^n$ be such that
\[
\hat{\epsilon}_t = \left\{
\begin{array}{ll}
\epsilon_t & \mbox{if } t \notin \{k, k+1\} \\
\epsilon_k & \mbox{if } t=k+1 \\
\epsilon_{k+1} & \mbox{if } t=k
\end{array} \right. ,
\]
let $\hat{\chi} : \{1, \ldots, n\} \to \{\ell, r\}$ be such that
\[
\hat{\chi}(t) = \left\{
\begin{array}{ll}
\chi(t) & \mbox{if } t \notin \{k, k+1\} \\
\chi(k) & \mbox{if } t=k+1 \\
\chi(k+1) & \mbox{if } t=k
\end{array} \right. ,
\]
and let $\hat\pi \in BNC(\hat{\chi})$ be the unique shaded bi-non-crossing partition obtained by interchanging $k$ and $k+1$ in $\pi$ (note $\hat{\pi} \leq \hat{\epsilon}$ by construction).
Then
\[
\sum_{\substack{\sigma \in BNC(\chi,\epsilon) \\ \sigma \lrgeq \pi}} (-1)^{|\pi| - |\sigma|} = \sum_{\substack{\hat\sigma \in BNC(\hat\chi, \hat\epsilon) \\ \hat\sigma \lrgeq \hat\pi}} (-1)^{|\hat\pi| - |\hat\sigma|}
\]
and
\[
\sum_{\substack{\sigma \in BNC(\chi) \\ \pi \leq \sigma \leq \epsilon}} \mu_{BNC}(\pi, \sigma) = \sum_{\substack{\hat\sigma \in BNC(\hat\chi) \\ \hat\pi \leq \hat\sigma \leq \hat\epsilon}} \mu_{BNC}(\hat\pi, \hat\sigma).
\]
\end{lem}
\begin{proof}
Since the operation that takes an element $\sigma \in BNC(\chi)$ with $\sigma \leq \epsilon$ and produces an element $\hat\sigma \in BNC(\hat\chi)$ with $\hat\sigma \leq \hat\epsilon$ by interchanging $k$ and $k+1$ in $\sigma$ is a bijection, and since $\mu_{BNC}(\pi, \sigma) = \mu_{BNC}(\hat\pi, \hat\sigma)$ by Remarks \ref{mobiusfunctionremarks}, the second equation clearly holds.

To prove the first equation holds, we break the discussion into several cases.
For the first case, suppose $\epsilon_k \neq \epsilon_{k+1}$; that is, the nodes we desired to change the orders of are of different shades.
For example, see the following diagrams where $k=4$.
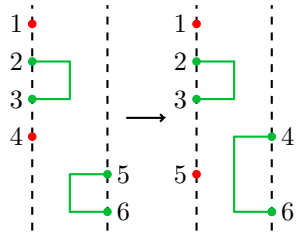
\begin{align*}
		\begin{tikzpicture}[baseline]
			\draw[thick,dashed] (-.5,-.25) -- (-.5,2.75);
			\draw[thick,dashed] (.5,-.25) -- (.5,2.75);
			\node[left] at (-.5,2.5) {1};
			\draw[red,fill=red] (-.5,2.5) circle (0.05);
			\node[left] at (-.5,2) {2};
			\draw[ggreen,fill=ggreen] (-.5,2) circle (0.05);
			\node[left] at (-.5,1.5) {3};
			\draw[ggreen,fill=ggreen] (-.5,1.5) circle (0.05);
			\node[left] at (-.5,1) {4};
			\draw[red,fill=red] (-.5,1) circle (0.05);
			\node[right] at (.5,.5) {5};
			\draw[ggreen,fill=ggreen] (.5,.5) circle (0.05);
			\node[right] at (.5,0) {6};
			\draw[ggreen,fill=ggreen] (.5,0) circle (0.05);
			\draw[thick, ggreen] (.5, 0) -- (0,0) -- (0,.5) -- (.5, .5);
			\draw[thick, ggreen] (-.5, 2) -- (0,2) -- (0,1.5) -- (-.5, 1.5);
			\draw[thick] (.75, 1.25) -- (1.25, 1.25) -- (1.2, 1.2);
			\draw[thick] (1.25, 1.25) -- (1.2, 1.3);
		\end{tikzpicture}
		\begin{tikzpicture}[baseline]
			\draw[thick,dashed] (-.5,-.25) -- (-.5,2.75);
			\draw[thick,dashed] (.5,-.25) -- (.5,2.75);
			\node[left] at (-.5,2.5) {1};
			\draw[red,fill=red] (-.5,2.5) circle (0.05);
			\node[left] at (-.5,2) {2};
			\draw[ggreen,fill=ggreen] (-.5,2) circle (0.05);
			\node[left] at (-.5,1.5) {3};
			\draw[ggreen,fill=ggreen] (-.5,1.5) circle (0.05);
			\node[left] at (-.5,.5) {5};
			\draw[red,fill=red] (-.5,.5) circle (0.05);
			\node[right] at (.5,1) {4};
			\draw[ggreen,fill=ggreen] (.5,1) circle (0.05);
			\node[right] at (.5,0) {6};
			\draw[ggreen,fill=ggreen] (.5,0) circle (0.05);
			\draw[ggreen,fill=ggreen] (.5,0) circle (0.05);
			\draw[thick, ggreen] (.5, 0) -- (0,0) -- (0,1) -- (.5, 1);
			\draw[thick, ggreen] (-.5, 2) -- (0,2) -- (0,1.5) -- (-.5, 1.5);
		\end{tikzpicture}
\end{align*}
In this case it is clear that the operation that takes $\sigma \in BNC(\chi)$ to $\hat\sigma \in BNC(\hat\chi)$ described above is a bijection that maps $BNC(\chi, \epsilon)$ to $BNC(\hat\chi, \hat\epsilon)$, is such that $(-1)^{|\pi| - |\sigma|} = (-1)^{|\hat\pi| - |\hat\sigma|}$, and is such that $\sigma \lrgeq \pi$ if and only if $\hat\sigma \lrgeq \hat\pi$.
Hence the first equation holds in this case.

Otherwise $\epsilon_k = \epsilon_{k+1}$.
Suppose $k$ and $k+1$ are in the same block of $\pi$.
For example, consider the following diagrams where $k=3$.
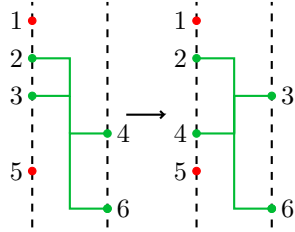
\begin{align*}
		\begin{tikzpicture}[baseline]
			\draw[thick,dashed] (-.5,-.25) -- (-.5,2.75);
			\draw[thick,dashed] (.5,-.25) -- (.5,2.75);
			\node[left] at (-.5,2.5) {1};
			\draw[red,fill=red] (-.5,2.5) circle (0.05);
			\node[left] at (-.5,2) {2};
			\draw[ggreen,fill=ggreen] (-.5,2) circle (0.05);
			\node[left] at (-.5,1.5) {3};
			\draw[ggreen,fill=ggreen] (-.5,1.5) circle (0.05);
			\node[left] at (-.5,.5) {5};
			\draw[red,fill=red] (-.5,.5) circle (0.05);
			\node[right] at (.5,1) {4};
			\draw[ggreen,fill=ggreen] (.5,1) circle (0.05);
			\node[right] at (.5,0) {6};
			\draw[ggreen,fill=ggreen] (.5,0) circle (0.05);
			\draw[ggreen,fill=ggreen] (.5,0) circle (0.05);
			\draw[thick, ggreen] (.5, 0) -- (0,0) -- (0,1) -- (.5, 1);
			\draw[thick, ggreen] (-.5, 2) -- (0,2) -- (0,1.5) -- (-.5, 1.5);
			\draw[thick, ggreen] (0,1) -- (0,1.5);
			\draw[thick] (.75, 1.25) -- (1.25, 1.25) -- (1.2, 1.2);
			\draw[thick] (1.25, 1.25) -- (1.2, 1.3);
		\end{tikzpicture}
		\begin{tikzpicture}[baseline]
			\draw[thick,dashed] (-.5,-.25) -- (-.5,2.75);
			\draw[thick,dashed] (.5,-.25) -- (.5,2.75);
			\node[left] at (-.5,2.5) {1};
			\draw[red,fill=red] (-.5,2.5) circle (0.05);
			\node[left] at (-.5,2) {2};
			\draw[ggreen,fill=ggreen] (-.5,2) circle (0.05);
			\node[left] at (-.5,1) {4};
			\draw[ggreen,fill=ggreen] (-.5,1) circle (0.05);
			\node[left] at (-.5,.5) {5};
			\draw[red,fill=red] (-.5,.5) circle (0.05);
			\node[right] at (.5,1.5) {3};
			\draw[ggreen,fill=ggreen] (.5,1.5) circle (0.05);
			\node[right] at (.5,0) {6};
			\draw[ggreen,fill=ggreen] (.5,0) circle (0.05);
			\draw[ggreen,fill=ggreen] (.5,0) circle (0.05);
			\draw[thick, ggreen] (.5, 0) -- (0,0) -- (0,1.5) -- (.5, 1.5);
			\draw[thick, ggreen] (-.5, 2) -- (0,2) -- (0,1) -- (-.5, 1);
			\draw[thick, ggreen] (0,1) -- (0,1.5);
		\end{tikzpicture}
\end{align*}
It is again clear that the same identifications as the previous case hold and thus the first equation holds in this case.
Hence we have reduced to the case that $k$ and $k+1$ are in different blocks of the same shade.

Let $V_1$ and $V_2$ be the blocks in $\pi$ of $k$ and $k+1$ respectively.
Note that $V_1$ contains a left node and $V_2$ contains a right node and the sum on the left-hand-side of the first equation is
\[
\sum_{\substack{\sigma \in BNC(\chi,\epsilon) \\ \sigma \lrgeq \pi \\ k, k+1 \text{ in separated blocks of } \sigma}} (-1)^{|\pi| - |\sigma|} + \sum_{\substack{\sigma \in BNC(\chi,\epsilon) \\ \sigma \lrgeq \pi \\ k, k+1 \text{ not in separated blocks of } \sigma}} (-1)^{|\pi| - |\sigma|}.
\]
We claim that
\[
\sum_{\substack{\sigma \in BNC(\chi,\epsilon) \\ \sigma \lrgeq \pi \\ k, k+1 \text{ not in separated blocks of } \sigma}} (-1)^{|\pi| - |\sigma|} = 0.
\]
Indeed we will split the discussion into two cases: when $V_1$ and $V_2$ are piled and when they are not.
For an example when $V_1$ and $V_2$ are piled, consider the following diagram.
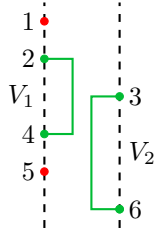
\[
\begin{tikzpicture}[baseline]
			\draw[thick,dashed] (-.5,-.25) -- (-.5,2.75);
			\draw[thick,dashed] (.5,-.25) -- (.5,2.75);
			\node[left] at (-.5,2.5) {1};
			\draw[red,fill=red] (-.5,2.5) circle (0.05);
			\node[left] at (-.5,2) {2};
			\draw[ggreen,fill=ggreen] (-.5,2) circle (0.05);
			\node[left] at (-.5,1) {4};
			\draw[ggreen,fill=ggreen] (-.5,1) circle (0.05);
			\node[left] at (-.5,.5) {5};
			\draw[red,fill=red] (-.5,.5) circle (0.05);
			\node[right] at (.5,1.5) {3};
			\draw[ggreen,fill=ggreen] (.5,1.5) circle (0.05);
			\node[right] at (.5,0) {6};
			\draw[ggreen,fill=ggreen] (.5,0) circle (0.05);
			\draw[ggreen,fill=ggreen] (.5,0) circle (0.05);
			\draw[thick, ggreen] (.5, 0) -- (.125,0) -- (.125,1.5) -- (.5, 1.5);
			\draw[thick, ggreen] (-.5, 2) -- (-0.125,2) -- (-0.125,1) -- (-.5, 1);
			
			\node[left] at (-0.5, 1.5) {$V_1$};
			\node[right] at (.5, .75) {$V_2$};
		\end{tikzpicture}
\]
If $V_1$ and $V_2$ are piled, it is easy to see that any $\sigma \in BNC(\chi,\epsilon)$ such that $\pi \leq \sigma$ and $k$ and $k+1$ are not in separated blocks of $\sigma$ must be such that $V_1$ and $V_2$ are contained in the same block of $\sigma$.
However, this implies that $\pi$ is not a lateral refinement of $\sigma$ as joining piled blocks cannot be undone by a lateral refinement.
Hence the sum is zero in this case.
Otherwise, suppose $V_1$ and $V_2$ are not piled.
For an example where $V_1$ and $V_2$ are not piled, consider the following diagram.
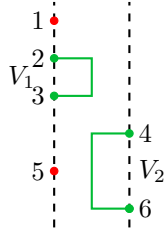
\[
\begin{tikzpicture}[baseline]
			\draw[thick,dashed] (-.5,-.25) -- (-.5,2.75);
			\draw[thick,dashed] (.5,-.25) -- (.5,2.75);
			\node[left] at (-.5,2.5) {1};
			\draw[red,fill=red] (-.5,2.5) circle (0.05);
			\node[left] at (-.5,2) {2};
			\draw[ggreen,fill=ggreen] (-.5,2) circle (0.05);
			\node[left] at (-.5,1.5) {3};
			\draw[ggreen,fill=ggreen] (-.5,1.5) circle (0.05);
			\node[left] at (-.5,.5) {5};
			\draw[red,fill=red] (-.5,.5) circle (0.05);
			\node[right] at (.5,1) {4};
			\draw[ggreen,fill=ggreen] (.5,1) circle (0.05);
			\node[right] at (.5,0) {6};
			\draw[ggreen,fill=ggreen] (.5,0) circle (0.05);
			\draw[ggreen,fill=ggreen] (.5,0) circle (0.05);
			\draw[thick, ggreen] (.5, 0) -- (0,0) -- (0,1) -- (.5, 1);
			\draw[thick, ggreen] (-.5, 2) -- (0,2) -- (0,1.5) -- (-.5, 1.5);
			
			\node[left] at (-0.625, 1.75) {$V_1$};
			\node[right] at (.5, .5) {$V_2$};
		\end{tikzpicture}
\]
This implies $k$ is the lowest element of $V_1$ in the bi-non-crossing diagram of $\pi$ and $k+1$ is the highest element of $V_2$.
If $\sigma \in BNC(\chi, \epsilon)$ is such that $k$ and $k+1$ are not in separated blocks of $\sigma$ and $\sigma \geq \pi$, then if $k$ and $k+1$ are in the same block of $\sigma$, let $\sigma' \lrleq \sigma$ splitting the block containing $k$ and $k+1$ in between these node (note $\sigma' \in BNC(\chi, \epsilon)$).
Otherwise $k$ and $k+1$ are not in the same block of $\sigma$ so letting $\sigma' \lrgeq \sigma$ be the partition made by joining the blocks containing $k$ and $k+1$ together also forms a partition in $BNC(\chi, \epsilon)$.
In either case $(-1)^{|\pi| - |\sigma|} + (-1)^{|\pi| - |\sigma'|} = 0$.
Note that the correspondance between $\sigma$ and $\sigma'$ in each case is one-to-one and thus the sum is zero.

Similar arguments show that
\[
\sum_{\substack{\hat\sigma \in BNC(\hat\chi,\hat\epsilon) \\ \hat\sigma \lrgeq \hat\pi}}(-1)^{|\hat\pi| - |\hat\sigma|} = \sum_{\substack{\hat\sigma \in BNC(\hat\chi,\hat\epsilon) \\ \hat\sigma \lrgeq \hat\pi \\ k, k+1 \text{ in separated blocks of } \hat\sigma}} (-1)^{|\hat\pi| - |\hat\sigma|}.
\]
However, the map taking $\sigma \in BNC(\chi)$ to $\hat\sigma \in BNC(\hat\chi)$ is such that $k$ and $k+1$ are in separated blocks of $\sigma$ if and only if $k$ and $k+1$ are in separated blocks of $\hat\sigma$, and under these conditions $\sigma \in BNC(\chi, \epsilon)$ if and only if $\hat\sigma \in BNC(\hat\chi, \hat\epsilon)$, $\sigma \lrgeq \pi$ if and only if $\hat\sigma \lrgeq \hat\pi$, and $(-1)^{|\pi|-|\sigma|} = (-1)^{|\hat\pi| - |\hat\sigma|}$.
Hence the first equation holds in this final case.
\end{proof}

\begin{proof}[Proof of Proposition \ref{annoyingsumresult}]
Given $\pi$, a $\hat\pi$ in $BNC(\hat\chi)$ where $\hat\chi : \{1, \ldots, n\} \to \{\ell\}$ may be constructed such that $\hat\pi$ can be modified to make $\pi$ via the operations in used in Lemma \ref{changingalefttoaright} and Lemma \ref{changingorderofaleftandright}.
Since the sums are equal for $\hat\pi$ by Lemma \ref{annoyingsumfreecase} and since Lemma \ref{changingalefttoaright} and Lemma \ref{changingorderofaleftandright} preserve the equality of the sums,  the result hold for $\pi$.
\end{proof}

We apply Proposition \ref{annoyingsumresult} to Corollary \ref{bi-free_moment_formula} to immediately obtain the following.
\begin{cor}
\label{something_corollary}
Let $z'$ and $z''$ be a pair of two-faced families in $(\mc{A},\varphi)$. Then $z'$ and $z''$ are bi-free if and only if for every map $\alpha\colon\{1,\ldots, n\}\rightarrow I\sqcup J$ and $\epsilon\in\{',''\}^n$ we have
	\begin{align}\label{bi-free_moment_formula_formula_formula}
		\varphi_\alpha\left(z^\epsilon\right)= \sum_{\pi\in BNC(\alpha)} \left[\sum_{\substack{\sigma\in BNC(\alpha)\\ \pi\leq\sigma\leq \epsilon}} \mu_{BNC}(\pi, \sigma)\right] \varphi_\pi(z^\epsilon),
	\end{align}
where $z^\epsilon = \left(z_{\alpha(1)}^{\epsilon_1},\ldots,z_{\alpha(n)}^{\epsilon_n}\right)$.
\end{cor}

\subsection{Bi-free is equivalent to combinatorially-bi-free}

\begin{thm}
\label{main_theorem}
Let $z'=((z_i')_{i\in I}, (z_j')_{j\in J})$ and $z''=((z_i'')_{i\in I}, (z_j'')_{j\in J})$ be a pair of two-faced families in a non-commutative probability space $(\mathcal{A},\varphi)$. Then $z'$ and $z''$ are bi-free if and only if they are combinatorially-bi-free.
\end{thm}

\begin{proof}
Suppose $z'$ and $z''$ are bi-free, and fix a shading $\epsilon \in \{', ''\}^n$.
By Corollary \ref{something_corollary}, for $\alpha : \set{1, \ldots, n} \to I\sqcup J$ we have
\[
\varphi_\alpha\left( z^{\epsilon} \right) = \sum_{{\pi \in BNC(\alpha)}} \left(   \sum_{\substack{\sigma \in BNC(\alpha) \\ \pi \leq \sigma \leq \epsilon}} \mu_{BNC}(\pi, \sigma) \right) \varphi_\pi\left(z^\epsilon\right).
\]
Therefore
\[
\varphi_\alpha\left( z^{\epsilon} \right) = \sum_{\substack{\sigma \in BNC(\alpha) \\ \sigma \leq \epsilon}} \kappa_\sigma\left(z^{\epsilon}\right)
\]
by Remark \ref{moebius_power}.
Using the above formula, we will proceed inductively to show that $\kappa_\sigma\left(z^{\epsilon}\right) = 0$ if $\sigma \in BNC(\alpha)$ and $\sigma \nleq \epsilon$.
The base case is where $n = 1$ is immediate.


For the inductive case, suppose the result holds for any $\beta : \{1, \ldots, k\} \to I \sqcup J$ with $k < n$.
Let $\alpha : \{1,\ldots, n\} \to I \sqcup J$
Suppose $\epsilon$ is not constant (so in particular, $1_\alpha \nleq \epsilon$).
Then
\[
\sum_{\sigma \in BNC(\alpha)} \kappa_\sigma\left( z^{\epsilon} \right) = \varphi_\alpha\left( z^{\epsilon} \right) = \sum_{\substack{\sigma \in BNC(\alpha) \\ \sigma \leq \epsilon}} \kappa_\sigma\left(z^{\epsilon}\right).
\]
By induction, $\kappa_\sigma\left(z^{\epsilon}\right) = 0$ if $\sigma \in BNC(\alpha)\setminus \{1_\alpha\}$ and $\sigma \nleq \epsilon$.
Consequently
\[
\sum_{\sigma \in BNC(\alpha)} \kappa_\sigma\left(z^{\epsilon} \right) = \kappa_{1_\alpha}\left(z^{\epsilon} \right) + \sum_{\substack{\sigma \in BNC(\alpha) \\ \sigma \leq \epsilon}} \kappa_\sigma\left(z^{\epsilon}\right).
\]
Combining these two equations gives $\kappa_{1_\alpha}\left( z^{\epsilon} \right) = 0$ completing the inductive step.

Now suppose $z'$ and $z''$ are combinatorially-bi-free.
Then, for any $\alpha : \{1,\ldots, n\} \to I \sqcup J$ and $\epsilon \in \{', ''\}^n$,
\begin{align*}
\varphi_\alpha\left( z^{\epsilon} \right)
 &= \sum_{\sigma \in BNC(\alpha)} \kappa_\sigma\left( z^{\epsilon} \right)
 =  \sum_{\substack{\sigma \in BNC(\alpha) \\ \sigma \leq \epsilon}} \kappa_\sigma\left(z^{\epsilon}\right) \\
 &=  \sum_{\substack{\sigma \in BNC(\alpha) \\ \sigma \leq \epsilon}} \sum_{\substack{\pi \in BNC(\alpha)\\ \pi \leq \sigma}}\varphi_\pi\left(z^{\epsilon}\right) \mu_{BNC}(\pi, \sigma) \\
 &=  \sum_{{\pi\in BNC(\alpha) }} \left(\sum_{\substack{\sigma \in BNC(\alpha)\\ \pi \leq \sigma \leq \epsilon}} \mu_{BNC}(\pi, \sigma)\right) \varphi_\pi\left(z^{\epsilon}\right).
\end{align*}
Hence Corollary \ref{something_corollary} implies that $z'$ and $z''$ are bi-free.
\end{proof}

\subsection{Voiculescu's universal bi-free polynomials}

Using the equivalence of bi-free independence and combinatorial-bi-free independence we obtain explicit formulas for several universal polynomials appearing \cite{voiculescu2013freei}.

\begin{prop}\label{moment_cor}
Let $\alpha\colon\{1,\ldots, n\}\rightarrow I\sqcup J$.
For each shading $\epsilon\in\{',''\}^n$ we define a polynomial $P_{\alpha,\epsilon}$ on indeterminates $X'_K$ and $X''_K$ indexed by non-empty subsets $K\subset\{1,\ldots, n\}$ by the formula
	\begin{align*}
		P_{\alpha,\epsilon} := \sum_{\pi\in BNC(\alpha,\epsilon)} \sq{\sum_{\substack{\sigma\in BNC(\alpha)\\ \pi\leq\sigma\leq\epsilon}} \mu_{BNC}(\pi, \sigma)} \prod_{V\in \pi} X_V^{\epsilon(V)}.
	\end{align*}
Then for $z'$ and $z''$ a bi-free pair of two-faced families in $(\mc{A},\varphi)$ we have
	\begin{align*}
		\varphi_\alpha(z^\epsilon)= P_{\alpha,\epsilon}(z',z''),
	\end{align*}
where $P_{\alpha,\epsilon}(z',z'')$ is given by evaluating $P_{\alpha,\epsilon}$ at $X_{\{k_1<\cdots<k_r\}}^{\delta}=\varphi(z_{\alpha(k_1)}^\delta\cdots z_{\alpha(k_r)}^\delta)$, $\delta\in\{',''\}$.

Furthermore, if we define $Q_\alpha$ as the sum of the $P_{\alpha,\epsilon}$ over all possible shadings then
	\begin{align*}
		Q_\alpha = X_{\{1,\ldots, n\}}' + X_{\{1,\ldots, n\}}'' + \sum P_{\alpha,\epsilon},
	\end{align*}
where the summation is over non-constant shadings $\epsilon$, and
	\begin{align*}
		\varphi_\alpha( z' + z'' )=Q_\alpha(z',z''),
	\end{align*}
where $Q_\alpha(z',z'')$ is $Q_\alpha$ evaluated at the same point as the $P_{\alpha,\epsilon}$ above.
\end{prop}

\begin{proof}
The first part of this corollary is immediate from Corollary \ref{something_corollary}. The assertion regarding $Q_\alpha(z',z'')$ is also immediate when expanding the product in the left-hand side. All that remains to show is
	\begin{align*}
		Q_\alpha = X_{\{1,\ldots, n\}}' + X_{\{1,\ldots, n\}}'' + \sum P_{\alpha,\epsilon},
	\end{align*}
which is equivalent to saying $P_{\alpha,\epsilon}= X_{\{1,\ldots, n\}}^\delta$ when $\epsilon$ is the constant shading $\epsilon=(\delta,\ldots, \delta)$, $\delta\in \{',''\}$.
Such a shading induces the full partition $1_\alpha$, and hence
\[
\sum_{\substack{\sigma\in BNC(\alpha)\\ \pi\leq\sigma\leq\epsilon}} \mu_{BNC}(\pi, \sigma)
= \sum_{\substack{\sigma\in BNC(\alpha)\\ \pi\leq\sigma\leq1_\alpha}} \mu_{BNC}(\pi, \sigma)
= \delta_{BNC}(\pi, 1_\alpha).
\]
Then the only term in $P_{\alpha, \epsilon}$ with a non-zero coefficient is the one corresponding to $\pi = 1_\alpha$.
\end{proof}

\begin{prop}
For any $\alpha\colon\set{1,\ldots, n}\to I\sqcup J$, recursively define polynomials $R_\alpha$ on indeterminates $X_K$ indexed by non-empty subsets $K\subseteq\{1,\ldots, n\}$ by the formula
	\begin{align*}
		R_\alpha= \sum_{\pi\in BNC(\alpha)} \mu_{BNC}(\pi, 1_\alpha) \prod_{V\in \pi} X_V
	\end{align*}
If $X_K$ is given degree $|K|$, then $R_\alpha$ is homogeneous with degree $n$.

For $z$ a two-faced family in $(\mc{A},\varphi)$, if $R_{\alpha}(z)$ denotes $R_\alpha$ evaluated at the point $X_{\{k_1<\cdots <k_r\}} =\varphi(z_{\alpha(k_1)}\cdots z_{\alpha(k_r)})$ then $R_\alpha(z)=\kappa_\alpha(z)$. Moreover, if $z'$ and $z''$ are bi-free in $(\mc{A},\varphi)$ then $R_\alpha(z'+z'')=R_\alpha(z')+R_\alpha(z'')$; that is, $R_\alpha$ has the cumulant property.
\end{prop}

\begin{proof}
We see that $R_\alpha(z)$ and $\kappa_\alpha(z)$ are equal by Remark \ref{moebius_power}). Then $R_\alpha$ has the cumulant property simply because $\kappa_\alpha$ does.
\end{proof}

\begin{rem}
The polynomials $P_{\alpha,\epsilon}$, $Q_\alpha$, and $R_\alpha$ are precisely the universal polynomials from Propositions 2.18, 5.6,  and Theorem 5.7, respectively, in \cite{voiculescu2013freei}.
\end{rem}


\section{A Multiplicative Bi-Free Convolution}
\label{sec:mult}

\subsection{Kreweras complement on bi-non-crossing partitions}

In \cite{nicaspeicher1997}, the Kreweras complement $K_{NC}$ on the non-crossing partitions was used to simplify the convolution of multiplicative functions.
In particular, we have the following extension to $BNC$.

\begin{defi}
\label{kreweras}
For any $\chi : \{1,\ldots, n\} \to I \sqcup J$ and $\pi \in BNC(\chi)$, the \emph{Kreweras complement} of $\pi$, denoted $K_{BNC}(\pi)$, is the element of $BNC(\chi)$ obtained by applying $s_\chi$ to the Kreweras complement in $NC(n)$ of $s^{-1}_\chi \cdot \pi$; explicitly
\[
K_{BNC}(\pi) = s_\chi \cdot K_{NC}(s^{-1}_\chi \cdot \pi).
\]
\end{defi}

\begin{rem}
Note that $K_{BNC}(\pi)$ may be obtained by taking the diagram corresponding to $\pi$, placing a node beneath each left node and above each right node of $\pi$, and drawing the largest bi-non-crossing diagram on the new nodes.
\end{rem}

\begin{ex}
In the following diagram, if $\pi$ is the bi-non-crossing partition drawn in black, $K_{BNC}(\pi)$ is the bi-non-crossing partition in red.
\[
\begin{tikzpicture}[baseline]
	\node[left] at (-.75, 7.5) {$1$};
	\draw[fill=black] (-.75,7.5) circle (0.05);
	\node[left] at (-.75, 6.5) {$2$};
	\draw[fill=black] (-.75,6.5) circle (0.05);
	\node[right] at (.5, 5.5) {$\overline3$};
	\draw[red, fill=red] (.5,5.5) circle (0.05);
	\node[left] at (-.75, 4.5) {$4$};
	\draw[fill=black] (-.75,4.5) circle (0.05);
	\node[left] at (-.75, 3.5) {$5$};
	\draw[fill=black] (-.75,3.5) circle (0.05);
	\node[left] at (-.75, 2.5) {$6$};
	\draw[fill=black] (-.75,2.5) circle (0.05);
	\node[right] at (.5,1.5) {$\overline7$};
	\draw[red, fill=red] (.5,1.5) circle (0.05);
	\node[right] at (.5,0.5) {$\overline8$};
	\draw[red, fill=red] (.5,0.5) circle (0.05);
	\node[left] at (-.75, 7) {$\overline{1}$};
	\draw[red, fill=red] (-.75,7) circle (0.05);
	\node[left] at (-.75, 6) {$\overline{2}$};
	\draw[red, fill=red] (-.75,6) circle (0.05);
	\node[right] at (.5, 5) {${3}$};
	\draw[fill=black] (.5,5) circle (0.05);
	\node[left] at (-.75, 4) {$\overline{4}$};
	\draw[red, fill=red] (-.75,4) circle (0.05);
	\node[left] at (-.75, 3) {$\overline{5}$};
	\draw[red, fill=red] (-.75,3) circle (0.05);
	\node[left] at (-.75, 2) {$\overline{6}$};
	\draw[red, fill=red] (-.75,2) circle (0.05);
	\node[right] at (.5,1) {${7}$};
	\draw[fill=black] (.5,1) circle (0.05);
	\node[right] at (.5,0) {${8}$};
	\draw[fill=black] (.5,0) circle (0.05);

	\draw[thick] (-0.7,7.5) -- (0,7.5) -- (0, 3.5) -- (-.7, 3.5);
	\draw[thick] (.5,5) -- (0, 5);
	\draw[thick] (-0.7, 6.5) -- (-.5, 6.5) -- (-.5, 4.5) -- (-.7, 4.5);
	\draw[thick] (-.7,2.5) -- (0, 2.5) -- (0, 0) -- (.45, 0);

	\draw[red, thick] (-0.7, 7) -- (-.25, 7) -- (-.25, 4) -- (-.7, 4);
	\draw[red, thick] (-.7,3) -- (.25, 3) -- (.25, .5) -- (.45, .5);
	\draw[red, thick] (.45,1.5) -- (.25, 1.5);
	\end{tikzpicture}
\]
\end{ex}

\begin{rem}
\label{convolutionofmultiplicativefunctions}
Since $K_{NC}$ is an order reversing and $s_\chi$ is order preserving, $K_{BNC}$ is an order reversing bijection. Thus $[\pi, 1_\alpha] \simeq [K_{BNC}(1_\alpha), K_{BNC}(\pi)] = [0_\alpha, K_{BNC}(\pi)]$ for all $\pi \in BNC(\alpha)$.
Hence, if $f,g \in IA(BNC)$ are multiplicative functions, then
\[
(f*g)(0_\alpha, 1_\alpha) = \sum_{\pi \in BNC(\alpha)} f(0_\alpha, \pi) g(0_\alpha, K_{BNC}(\pi)) = (g*f)(0_\alpha, 1_\alpha)
\]
and thus $f*g = g*f$.
\end{rem}

\subsection{Computing cumulants of a multiplicative bi-free convolution}

Taking inspiration from \cite{nicaspeicher1997}, we use the Kreweras complement to examine the bi-free cumulants of a two-faced family generated by products of a bi-free pair of two-faced families.
\begin{thm}
\label{cumulantsofproductoffaces}
Let $z'=( \{z'_{\ell}\}, \{z'_{r}\})$ and $z''=( \{z''_{\ell}\}, \{z''_{r}\})$ be a bi-free family of  pairs of faces and let $z = (\{z'_{\ell} z''_{\ell}\}, \{z''_{r}z'_{r}\})$.
Then
\[
\kappa_{\chi}(z) = \sum_{\pi \in BNC(\chi)} \kappa_\pi(z') \kappa_{K_{BNC}(\pi)}(z'')
\]
for all $\chi : \{1, \ldots, n\} \to \{\ell, r\}$.
\end{thm}
\begin{proof}
Recall from Remark \ref{moebius_power} the definition of the moment and bi-free cumulant functions $m_x$ and $\kappa_x$, and moreover that these are uniquely determined by the moments and cumulants of the family $(x)$, respectively.
Since the bi-free cumulant functions are multiplicative and by the structure of the convolution of multiplicative functions given in Remark \ref{convolutionofmultiplicativefunctions}, it suffices to show $\kappa_z = \kappa_{z'} \ast \kappa_{z''}$.
Using the relations $m_z*\mu_{BNC} = \kappa_z$ and $\kappa_z * \zeta_{BNC} = m_z$, it suffices to show $m_z = \kappa_{z'} \ast m_{z''}$.

Suppose $\chi : \{1,\ldots, n\} \to \{\ell, r\}$.
Let $\beta : \{1,\ldots, 2n\} \to \{\ell, r\}$ be given by $\beta(2k-1) =\beta(2k) =\chi(k)$.
Take $\epsilon \in \{',''\}^{2n}$ so that $\epsilon_{2k-1} =\ '$ and $\epsilon_{2k} =\ ''$ if $k \in \chi^{-1}(\ell)$, and the opposite if $k \in \chi^{-1}(r)$.
Then
\begin{align*}
m_z(0_\chi, 1_\chi) = \varphi_{\chi}(z)
 &= \varphi\left(z^{\epsilon_1}_{\chi(1)}z^{\epsilon_2}_{\chi(1)} \cdots z^{\epsilon_{2n-1}}_{\chi(n)}z^{\epsilon_{2n}}_{\chi(n)}\right)  \\
 &= \varphi\left(z^{\epsilon_1}_{\beta(1)}z^{\epsilon_2}_{\beta(2)} \cdots z^{\epsilon_{2n-1}}_{\beta(2n-1)}z^{\epsilon_{2n}}_{\beta(2n)}\right)  \\
 &= \sum_{\pi \in BNC(\beta, \epsilon)} \kappa_\pi(z^\epsilon) \\
  &= \sum_{\pi_1 \in BNC(\chi)} \kappa_{\pi_1}(z')  \sum_{\substack{\pi_2 \in BNC(\chi)\\ \pi_2 \leq K_{BNC}(\pi_1)}} \kappa_{\pi_2}(z'') \\
 &=  \sum_{\pi_1 \in BNC(\chi)} \kappa_{\pi_1}(z') \varphi_{K_{BNC}(\pi_1)}(z'') \\
 &=  (\kappa_{z'} \ast m_{z''})(0_\chi, 1_\chi)
\end{align*}
Hence, as $m_z$ and $\kappa_{z'} \ast m_{z''}$ are multiplicative functions that agree on full lattices in $BNC$, and consequently on all intervals by bi-multiplicitivity.
\end{proof}

\begin{rem}
Note that the above generalizes the formula for the free cumulants of the multiplicative convolution of freely independent random variables in terms of their individual cumulants (\emph{cf}. Section 3.5 of \cite{nicaspeicher1997}).
This seems to suggest that when defining the multiplicative convolution of a bi-free pair of two-faced families one should multiply the right faces as if in the opposite algebra.
\end{rem}

\begin{rem}
Since convolution is abelian on multiplicative functions, we obtain that $(\{z'_{\ell} z''_{\ell}\}, \{z''_{r}z'_{r}\})$ and $(\{z''_{\ell} z'_{\ell}\}, \{z'_{r}z''_{r}\})$ have the same joint distributions.
\end{rem}


\section{An Operator Model for Pairs of Faces}
\label{sec:model}

In this section we will construct an operator model for a two-faced family in a non-commutative probability space.
This model will generalize the operator model usually considered in free probability introduced by Nica in \cite{Nica1996271}.

In Definition 3.2.1 of \cite{Nica1996271}, Nica's operator model is constructed via unbounded operators on a Fock space making use of the left creation and annihilation operators where each product of creation operators is weighted by a free cumulant of the random variables.
The operator model for a pair of faces in a non-commutative probability space will be constructed in Theorem \ref{operatormodelforapairoffaces}, with terms similarly weighted by the corresponding bi-free cumulants.
We use left annihilation operators in the same way as Nica's model, though we must use more complex operators than simply left and right creation, essentially to account for the fact that the order in which variables are annihilated does not correspond to the order in which they were added as strongly as in the free case.
Our model reduces to Nica's model when all variables are left (or right) variables.
Moreover, a model using only left and right creation and annihilation operators is unlikely, by discussions in \cite{1312.0269}.

Nica's operator model also gives a direct analogue to the $R$-series of a collection of random variables in a non-commutative probability space.
We introduce an analogous operator $\Theta_z$ in Theorem \ref{operatormodelforapairoffaces} which serves as the $R$-series of the two-faced family $z =\left( \{z_i\}_{i \in I}, \{z_j\}_{j \in J}\right)$.
In particular, if
\[
z' = \left( \{z'_i\}_{i \in I}, \{z'_j\}_{j \in J}\right) \mbox{ and } z'' = \left( \{z''_i\}_{i \in I}, \{z''_j\}_{j \in J}\right)
\]
is a bi-free pair of two-faced families, we can consider the single family
\[
z = \left( \{z'_i + z''_i\}_{i \in I}, \{z'_j + z''_j\}_{j \in J}\right)
\]
and construct the corresponding operator $\Theta_z$.
It will follow that
\[
\Theta_z - I = (\Theta_{z'} - I) + (\Theta_{z''} - I).
\]
Hence the operator $\Theta_z$ from Theorem \ref{operatormodelforapairoffaces} behaves like an $R$-series.

\subsection{Nica's Operator Model}
We will take a moment to recall Nica's operator model from \cite{1312.0269} to demystify our construction.
Given an index set $I$, let $\mc F\paren{\C^I}$ be the Fock space generated by $I$; that is,
$$\mc F\paren{\C^I} := \C\Omega\oplus\paren{\bigoplus_{\substack{k\geq1\\i_1,\ldots,i_k\in I}}\C(e_{i_1}\otimes\cdots\otimes e_{i_k})},$$
where $\set{e_i}_{i\in I}$ is an orthonormal basis of $\C^I$.
Recall that $\Omega$ is called the vacuum vector of $\mc F\paren{\C^I}$, and is thought of as a ``length zero tensor''; that is, for example, $\xi_1\otimes\cdots\otimes\xi_{k} = \Omega$ if $k = 0$.
To simplify notation, we will sometimes join impure tensors of varying lengths together with a tensor product symbol, which should be distributed across sums.
For example, $(\xi_1+\xi_2\otimes \xi_3)\otimes \xi_4 = \xi_1\otimes \xi_4 + \xi_2\otimes \xi_3\otimes \xi_4$, while $\xi\otimes\Omega = \xi = \Omega\otimes\xi$.

Recall next that the left creation operator corresponding to $e_i$, denoted $L_i$, is defined by $L_i(\xi_1\otimes\cdots\otimes\xi_k) = e_i\otimes\xi_1\otimes\cdots\otimes\xi_k$, while its adjoint $L_i^*$ is called the left annihilation operator corresponding to $e_i$.
We take $\omega : \mc L\paren{\mc F\paren{\C^I}} \to \C$ to be the vector state corresponding to $\Omega$, so $\omega(T) := \ang{T\Omega, \Omega}$.

Suppose that $X = \paren{X_i}_{i\in I}$ is a collection of random variables in a non-commutative probability space $(\A, \varphi)$, and consider the (unbounded) operator
$$\Theta_X := I_{\mc F\paren{\C^I}} + \sum_{k\geq1}\sum_{i_1, \ldots, i_k \in I} \kappa\paren{X_{i_1}, \ldots, X_{i_k}}L_{i_k}\cdots L_{i_1},$$
where $\kappa$ is the free cumulant functional.
Next, set
$$Z_i := L_i^*\Theta_X = L_i^* + \sum_{k\geq0}\sum_{i_1, \ldots, i_k \in I} \kappa\paren{X_{i_1}, \ldots, X_{i_k}, X_i}L_{i_k}\cdots L_{i_1}.$$
Then the joint distribution of $(Z_i)_{i\in I}$ with respect to $\omega$ is the same as that of $(X_i)_{i\in I}$ with respect to $\varphi$.
Observe that given $i_1, \ldots, i_n$ there is a bijection between the non-crossing partitions $NC(n)$ and the terms in $Z_{i_1}\cdots Z_{i_n}$ of non-zero trace: one takes the finest partition such that if the annihilation operator of $Z_{i_t}$ cancels a creation operator added by a term from the variable $Z_{i_s}$, then $t$ and $s$ lie in the same block.

\begin{ex}
Consider the product $T_1T_2T_3T_4T_5 = Z_1Z_2Z_3Z_3Z_1$, which contains the term of non-zero trace $\kappa(X_2, X_3)\kappa(X_1, X_3, X_1)L_1^*L_2^*(L_3^*L_3L_2)L_3^*(L_1^*L_1L_3L_1)$.
This corresponds to the non-crossing partition $\set{\set{1,4,5}\set{2,3}}$, as the term $\kappa(X_2, X_3)L_3^*L_3L_2$ was selected from $T_3$, while the surviving $L_2$ was annihilated by $T_2$; this accounts for the block $\set{2, 3}$.
Similarly, $\kappa(X_1, X_3, X_1)L_1^*L_1L_3L_1$ was added by $T_5$, and its remaining pieces were annihilated by $T_4$ and $T_1$, which gives us the block $\set{1, 4, 5}$.

On the other hand, to find the term corresponding to the non-crossing partition $\set{\set{1, 5}, \set{2, 4}, \set{3}}$, we note that $T_5$ must introduce an operator annihilated by $T_1$, $T_4$ must for $T_2$, and $T_3$ annihilate any term it adds.
That is, we have the term
$$L_1^*L_2^*\paren{\kappa(X_3)L_3^*L_3}\paren{\kappa(X_2, X_3)L_3^*L_3L_2}\paren{\kappa(X_1, X_1)L_1^*L_1L_1}.$$
\end{ex}


\subsection{Skeletons corresponding to bi-non-crossing partitions}

The operator model from \cite{Nica1996271} can be thought of as a systematic way of constructing all non-crossing partitions weighted by products of free cumulants.
Recall that non-crossing partitions may be viewed as bi-non-crossing partitions where all nodes are on the left-hand side.

\begin{defi}
Let $\alpha : \{1,\ldots, n\} \to I \sqcup J$.
For a bi-non-crossing partition $\pi \in BNC(\alpha)$, a \emph{skeleton on $\pi$} is a bi-non-crossing diagram of $\pi$ (as in Subsection \ref{sec:bncp}), labelled by $\alpha$, with a choice of each node being either closed or open subject to the constraint that any node below a closed node is also closed.
\end{defi}
\begin{ex}
If $\alpha$ and $\pi$ are as in Example \ref{ex:paulsneed}, the skeletons corresponding to $\pi$ are the following diagrams.
	\begin{align*}
	\begin{tikzpicture}[baseline]
	\node[left] at (-.5, 2) {$\alpha(1)$};
	\draw (-.5,2) circle (0.05);
	\node[left] at (-.5, 1.5) {$\alpha(2)$};
	\draw (-.5,1.5) circle (0.05);
	\node[right] at (.5, 1) {$\alpha(3)$};
	\draw (.5,1) circle (0.05);
	\node[left] at (-.5, .5) {$\alpha(4)$};
	\draw (-.5,.5) circle (0.05);
	\node[right] at (.5,0) {$\alpha(5)$};
	\draw (.5,0) circle (0.05);
	\draw[thick] (-.45,2) -- (0,2) -- (0, 1) -- (.45,1);
	\draw[thick] (-.45,1.5) -- (-.25,1.5) -- (-.25, 0) -- (.45,0);
	\draw[thick] (-.45,.5) -- (-.25,.5);
	\end{tikzpicture}
	\begin{tikzpicture}[baseline]
	\node[left] at (-.5, 2) {$\alpha(1)$};
	\draw (-.5,2) circle (0.05);
	\node[left] at (-.5, 1.5) {$\alpha(2)$};
	\draw (-.5,1.5) circle (0.05);
	\node[right] at (.5, 1) {$\alpha(3)$};
	\draw (.5,1) circle (0.05);
	\node[left] at (-.5, .5) {$\alpha(4)$};
	\draw (-.5,.5) circle (0.05);
	\node[right] at (.5,0) {$\alpha(5)$};
	\draw[fill=black] (.5,0) circle (0.05);
	\draw[thick] (-.45,2) -- (0,2) -- (0, 1) -- (.45,1);
	\draw[thick] (-.45,1.5) -- (-.25,1.5) -- (-.25, 0) -- (.45,0);
	\draw[thick] (-.45,.5) -- (-.25,.5);
	\end{tikzpicture}
	\begin{tikzpicture}[baseline]
	\node[left] at (-.5, 2) {$\alpha(1)$};
	\draw (-.5,2) circle (0.05);
	\node[left] at (-.5, 1.5) {$\alpha(2)$};
	\draw (-.5,1.5) circle (0.05);
	\node[right] at (.5, 1) {$\alpha(3)$};
	\draw (.5,1) circle (0.05);
	\node[left] at (-.5, .5) {$\alpha(4)$};
	\draw[fill=black] (-.5,.5) circle (0.05);
	\node[right] at (.5,0) {$\alpha(5)$};
	\draw[fill=black] (.5,0) circle (0.05);
	\draw[thick] (-.45,2) -- (0,2) -- (0, 1) -- (.45,1);
	\draw[thick] (-.45,1.5) -- (-.25,1.5) -- (-.25, 0) -- (.45,0);
	\draw[thick] (-.45,.5) -- (-.25,.5);
	\end{tikzpicture}
	\end{align*}
	\begin{align*}
	\begin{tikzpicture}[baseline]
	\node[left] at (-.5, 2) {$\alpha(1)$};
	\draw (-.5,2) circle (0.05);
	\node[left] at (-.5, 1.5) {$\alpha(2)$};
	\draw (-.5,1.5) circle (0.05);
	\node[right] at (.5, 1) {$\alpha(3)$};
	\draw[fill=black] (.5,1) circle (0.05);
	\node[left] at (-.5, .5) {$\alpha(4)$};
	\draw[fill=black] (-.5,.5) circle (0.05);
	\node[right] at (.5,0) {$\alpha(5)$};
	\draw[fill=black] (.5,0) circle (0.05);
	\draw[thick] (-.45,2) -- (0,2) -- (0, 1) -- (.45,1);
	\draw[thick] (-.45,1.5) -- (-.25,1.5) -- (-.25, 0) -- (.45,0);
	\draw[thick] (-.45,.5) -- (-.25,.5);
	\end{tikzpicture}
	\begin{tikzpicture}[baseline]
	\node[left] at (-.5, 2) {$\alpha(1)$};
	\draw (-.5,2) circle (0.05);
	\node[left] at (-.5, 1.5) {$\alpha(2)$};
	\draw[fill=black] (-.5,1.5) circle (0.05);
	\node[right] at (.5, 1) {$\alpha(3)$};
	\draw[fill=black] (.5,1) circle (0.05);
	\node[left] at (-.5, .5) {$\alpha(4)$};
	\draw[fill=black] (-.5,.5) circle (0.05);
	\node[right] at (.5,0) {$\alpha(5)$};
	\draw[fill=black] (.5,0) circle (0.05);
	\draw[thick] (-.45,2) -- (0,2) -- (0, 1) -- (.45,1);
	\draw[thick] (-.45,1.5) -- (-.25,1.5) -- (-.25, 0) -- (.45,0);
	\draw[thick] (-.45,.5) -- (-.25,.5);
	\end{tikzpicture}
	\begin{tikzpicture}[baseline]
	\node[left] at (-.5, 2) {$\alpha(1)$};
	\draw[fill=black] (-.5,2) circle (0.05);
	\node[left] at (-.5, 1.5) {$\alpha(2)$};
	\draw[fill=black] (-.5,1.5) circle (0.05);
	\node[right] at (.5, 1) {$\alpha(3)$};
	\draw[fill=black] (.5,1) circle (0.05);
	\node[left] at (-.5, .5) {$\alpha(4)$};
	\draw[fill=black] (-.5,.5) circle (0.05);
	\node[right] at (.5,0) {$\alpha(5)$};
	\draw[fill=black] (.5,0) circle (0.05);
	\draw[thick] (-.45,2) -- (0,2) -- (0, 1) -- (.45,1);
	\draw[thick] (-.45,1.5) -- (-.25,1.5) -- (-.25, 0) -- (.45,0);
	\draw[thick] (-.45,.5) -- (-.25,.5);
	\end{tikzpicture}
	\end{align*}
\end{ex}

\begin{defi}
We will refer to a skeleton where all nodes are closed circles as the \emph{completed skeleton}.
For a skeleton on $1_\alpha \in BNC(\alpha)$, the skeleton where all nodes are open will be referred to as the \emph{empty skeleton} corresponding to $\alpha$, while the skeleton where all but the bottom node is open will be referred to as the \emph{starter skeleton} corresponding to $\alpha$.
Any skeleton that is not empty will be referred to as a \emph{partially completed skeleton}.
\end{defi}

\begin{rem}
We will examine Nica's model in the language of skeletons, which we will think of as a bi-free situation where all variables come from the left face.
Let $\set{X_i}_{i\in I}$ be a family of non-commutative random variables, $\set{Z_i}_{i\in I}$ the corresponding operator model, and fix some $i_1, \ldots, i_n \in I$ and consider a product $(L_{i_1}^*T_1) \cdots (L_{i_n}^*T_n)$ where
$$T_k \in \set{I}\cup\set{\kappa\paren{X_{i_1'}, \ldots, X_{i_m'}}L_{i_m'}\cdots L_{i_1'} | m \geq 1, i'_1, \ldots, i'_m \in I}.$$
Note that $L_{i_k}^*T_k = 0$ unless the $T_k$ chosen is either $I$ or begins with $L_{i_k}$.
For $t \leq n$, we think of $(L_{i_t}^*T_t)\cdots (L_{i_n}^*T_n)\Omega$ as a partially completed skeleton weighted by a scalar which is a product of free cumulants.
There is not a bijection between partially completed skeletons and basis vectors of our Fock space as the partially completed skeleton will retain the information of how the vector was created.
Each annihilation operator acts on the skeleton by filling in the lowest open node if it is labelled appropriately (to make the node closed in the new skeleton), and otherwise weights the skeleton by zero (which removes the skeleton from consideration).
Note, then, that the closed nodes correspond to variables in the block which have been encountered, and the requirement that they be filled from bottom to top ensures that the ordering of variables matches the cumulant.
For example,
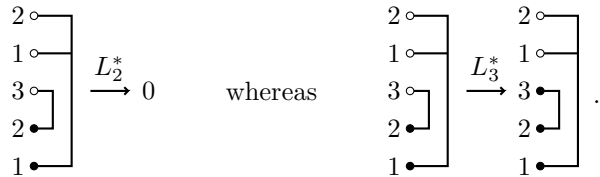
\[
	\begin{tikzpicture}[baseline]
	\node[left] at (-.5, 2) {2};
	\draw (-.5,2) circle (0.05);
	\node[left] at (-.5, 1.5) {1};
	\draw (-.5,1.5) circle (0.05);
	\node[left] at (-.5, 1) {3};
	\draw (-.5,1) circle (0.05);
	\node[left] at (-.5, .5) {2};
	\draw[fill=black] (-.5,.5) circle (0.05);
	\node[left] at (-.5,0) {1};
	\draw[fill=black] (-.5,0) circle (0.05);
	\draw[thick] (-.45,2) -- (0,2) -- (0, 0) -- (-.45,0);
	\draw[thick] (-.45,1.5) -- (0,1.5);
	\draw[thick] (-.45,1) -- (-.25,1) -- (-.25, 0.5) -- (-.45,.5);
	\node[above] at (.5, 1) {$L_2^*$};
	\draw[thick] (.25, 1) -- (.75, 1) -- (.7, .95);
	\draw[thick] (.75, 1) -- (.7, 1.05);
	\end{tikzpicture}
	\begin{tikzpicture}[baseline]
	\node[left] at (-.5, 1) {0};
	\end{tikzpicture}
	\qquad
	\begin{tikzpicture}[baseline]
	\node[left] at (-.5, 1) {whereas};
	\end{tikzpicture}
	\qquad
	\begin{tikzpicture}[baseline]
	\node[left] at (-.5, 2) {2};
	\draw (-.5,2) circle (0.05);
	\node[left] at (-.5, 1.5) {1};
	\draw (-.5,1.5) circle (0.05);
	\node[left] at (-.5, 1) {3};
	\draw (-.5,1) circle (0.05);
	\node[left] at (-.5, .5) {2};
	\draw[fill=black] (-.5,.5) circle (0.05);
	\node[left] at (-.5,0) {1};
	\draw[fill=black] (-.5,0) circle (0.05);
	\draw[thick] (-.45,2) -- (0,2) -- (0, 0) -- (-.45,0);
	\draw[thick] (-.45,1.5) -- (0,1.5);
	\draw[thick] (-.45,1) -- (-.25,1) -- (-.25, 0.5) -- (-.45,.5);
	\node[above] at (.5, 1) {$L_3^*$};
	\draw[thick] (.25, 1) -- (.75, 1) -- (.7, .95);
	\draw[thick] (.75, 1) -- (.7, 1.05);
	\end{tikzpicture}
	\begin{tikzpicture}[baseline]
	\node[left] at (-.5, 2) {2};
	\draw (-.5,2) circle (0.05);
	\node[left] at (-.5, 1.5) {1};
	\draw (-.5,1.5) circle (0.05);
	\node[left] at (-.5, 1) {3};
	\draw[fill=black] (-.5,1) circle (0.05);
	\node[left] at (-.5, .5) {2};
	\draw[fill=black] (-.5,.5) circle (0.05);
	\node[left] at (-.5,0) {1};
	\draw[fill=black] (-.5,0) circle (0.05);
	\draw[thick] (-.45,2) -- (0,2) -- (0, 0) -- (-.45,0);
	\draw[thick] (-.45,1.5) -- (0,1.5);
	\draw[thick] (-.45,1) -- (-.25,1) -- (-.25, 0.5) -- (-.45,.5);
	
	\node[below] at (.25,1) {.};
	\end{tikzpicture}
\]

Each product of creation operators $\kappa\paren{X_{i_1'}, \ldots, X_{i_m'}}L_{i_m'}\cdots L_{i_1'}$ adds an empty skeleton (corresponding to the creation operators chosen) to the skeleton under consideration directly above the highest closed node, and is weighted by the appropriate cumulant.
The lowest node of the new block is immediately filled by the following $L_{i_k}^*$.
For example, 
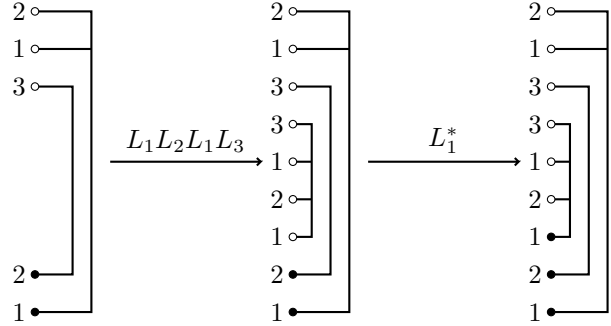
\begin{equation*}
\begin{tikzpicture}[baseline]
	\node[left] at (-.5, 4) {2};
	\draw (-.5,4) circle (0.05);
	\node[left] at (-.5, 3.5) {1};
	\draw (-.5,3.5) circle (0.05);
	\node[left] at (-.5, 3) {3};
	\draw (-.5,3) circle (0.05);
	\node[left] at (-.5, .5) {2};
	\draw[fill=black] (-.5,.5) circle (0.05);	
	\node[left] at (-.5, 0) {1};
	\draw[fill=black] (-.5,0) circle (0.05);
	\draw[thick] (-.45,4) -- (0.25,4) -- (0.25, 0) -- (-.45,0);
	\draw[thick] (-.45,3.5) -- (0.25,3.5);
	\draw[thick] (-.45,3) -- (0,3) -- (0, 0.5) -- (-.45,.5);
	\node[above] at (1.5, 2) {$L_1L_2L_1L_3$};
	\draw[thick] (.5, 2) -- (2.5, 2) -- (2.45, 1.95);
	\draw[thick] (2.5, 2) -- (2.45, 2.05);
	\end{tikzpicture}	
	\begin{tikzpicture}[baseline]
	\node[left] at (-.5, 4) {2};
	\draw (-.5,4) circle (0.05);
	\node[left] at (-.5, 3.5) {1};
	\draw (-.5,3.5) circle (0.05);
	\node[left] at (-.5, 3) {3};
	\draw (-.5,3) circle (0.05);
	\node[left] at (-.5, 2.5) {3};
	\draw (-.5,2.5) circle (0.05);
	\node[left] at (-.5, 2) {1};
	\draw (-.5,2) circle (0.05);
	\node[left] at (-.5, 1.5) {2};
	\draw (-.5,1.5) circle (0.05);
	\node[left] at (-.5, 1) {1};
	\draw (-.5,1) circle (0.05);
	\node[left] at (-.5, .5) {2};
	\draw[fill=black] (-.5,.5) circle (0.05);	
	\node[left] at (-.5, 0) {1};
	\draw[fill=black] (-.5,0) circle (0.05);
	\draw[thick] (-.45,4) -- (0.25,4) -- (0.25, 0) -- (-.45,0);
	\draw[thick] (-.45,3.5) -- (0.25,3.5);
	\draw[thick] (-.45,3) -- (0,3) -- (0, 0.5) -- (-.45,.5);
	\draw[thick] (-.45,2.5) -- (-.25,2.5) -- (-.25, 1) -- (-.45,1);
	\draw[thick] (-.45,2) -- (-.25,2);
	\draw[thick] (-.45,1.5) -- (-.25,1.5);
	\node[above] at (1.5, 2) {$L_1^*$};
	\draw[thick] (.5, 2) -- (2.5, 2) -- (2.45, 1.95);
	\draw[thick] (2.5, 2) -- (2.45, 2.05);
	\end{tikzpicture}
	\begin{tikzpicture}[baseline]
	\node[left] at (-.5, 4) {2};
	\draw (-.5,4) circle (0.05);
	\node[left] at (-.5, 3.5) {1};
	\draw (-.5,3.5) circle (0.05);
	\node[left] at (-.5, 3) {3};
	\draw (-.5,3) circle (0.05);
	\node[left] at (-.5, 2.5) {3};
	\draw (-.5,2.5) circle (0.05);
	\node[left] at (-.5, 2) {1};
	\draw (-.5,2) circle (0.05);
	\node[left] at (-.5, 1.5) {2};
	\draw (-.5,1.5) circle (0.05);
	\node[left] at (-.5, 1) {1};
	\draw[fill=black] (-.5,1) circle (0.05);
	\node[left] at (-.5, .5) {2};
	\draw[fill=black] (-.5,.5) circle (0.05);	
	\node[left] at (-.5, 0) {1};
	\draw[fill=black] (-.5,0) circle (0.05);
	\draw[thick] (-.45,4) -- (0.25,4) -- (0.25, 0) -- (-.45,0);
	\draw[thick] (-.45,3.5) -- (0.25,3.5);
	\draw[thick] (-.45,3) -- (0,3) -- (0, 0.5) -- (-.45,.5);
	\draw[thick] (-.45,2.5) -- (-.25,2.5) -- (-.25, 1) -- (-.45,1);
	\draw[thick] (-.45,2) -- (-.25,2);
	\draw[thick] (-.45,1.5) -- (-.25,1.5);
	\end{tikzpicture}
\end{equation*}

For a product $(L_{i_1}^*T_1) \cdots (L_{i_n}^*T_n)\Omega$, we will get precisely one partially completed skeleton.
For example,
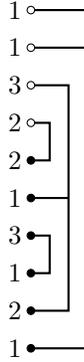
\begin{align*}
\begin{tikzpicture}[baseline]
\node[left] at (-.5, 4.5) {1};
	\draw (-.5,4.5) circle (0.05);	
	\node[left] at (-.5, 4) {1};
	\draw (-.5,4) circle (0.05);
	\node[left] at (-.5, 3.5) {3};
	\draw (-.5,3.5) circle (0.05);
	\node[left] at (-.5, 3) {2};
	\draw (-.5,3) circle (0.05);
	\node[left] at (-.5, 2.5) {2};
	\draw[fill=black] (-.5,2.5) circle (0.05);
	\node[left] at (-.5, 2) {1};
	\draw[fill=black] (-.5,2) circle (0.05);
	\node[left] at (-.5, 1.5) {3};
	\draw[fill=black] (-.5,1.5) circle (0.05);
	\node[left] at (-.5, 1) {1};
	\draw[fill=black] (-.5,1) circle (0.05);
	\node[left] at (-.5, .5) {2};
	\draw[fill=black] (-.5,.5) circle (0.05);	
	\node[left] at (-.5, 0) {1};
	\draw[fill=black] (-.5,0) circle (0.05);
	\draw[thick] (-.45,4.5) -- (0.25,4.5) -- (0.25, 0) -- (-.45,0);
	\draw[thick] (-.45,4) -- (0.25,4);
	\draw[thick] (-.45,3.5) -- (0,3.5) -- (0, 0.5) -- (-.45,.5);
		\draw[thick] (-.45,2) -- (0,2);
	\draw[thick] (-.45,3) -- (-.25,3) -- (-.25, 2.5) -- (-.45,2.5);
	\draw[thick] (-.45,1.5) -- (-.25,1.5) -- (-.25, 1) -- (-.45,1);	
	\end{tikzpicture}
\end{align*}
corresponds to the product
$$\paren{\kappa(X_2, X_2)L_2}L_1^*L_3^*\paren{\kappa(X_3, X_1)L_1^*L_1L_3}\paren{\kappa(X_3, X_1, X_2)L_2^*L_2L_1L_3}\paren{\kappa(X_1, X_1, X_1)L_1^*L_1L_1L_1}\Omega.$$
Notice that when the above operators are applied to $\Omega$ in the order listed, we obtain the vector
\[
\lambda e_2 \otimes e_3 \otimes e_1 \otimes e_1,
\]
where $\lambda$ is a product of cumulants.
The indices of the tensor can be seen in the partially completed skeleton by reading the open nodes from bottom to top.
In this manner, vectors $(L_{i_1}^*T_1) \cdots (L_{i_n}^*T_n)\Omega$ correspond to partially completed skeletons and the only products such that
\[
\ang{ (L_{i_1}^*T_1) \cdots (L_{i_n}^*T_n)\Omega, \Omega} \neq 0
\]
arise from completed skeletons.
It is easy to see that a completed skeleton corresponds to an element of $\pi \in NC(n)$.
These completed skeletons are weighted by the correct product of cumulants so that when we sum over all completed skeletons, we get
\[
\ang{Z_{i_1}\cdots Z_{i_n}\Omega, \Omega} = \sum_{\pi \in NC(n)} \kappa_\pi(X_{i_1}, \ldots, X_{i_n}) = \varphi(X_{i_1} \cdots X_{i_n}),
\]
as desired.
\end{rem}

\subsection{A construction}

We will now construct our operator model for pairs of faces, motivated by our realization of Nica's operator model.
Above, the model constructed all weighted non-crossing partitions by using creation operators to glue in full non-crossing blocks and annihilation operators to approve or reject non-crossing diagrams.
As the combinatorics of pairs of faces is dictated by bi-non-crossing partitions, we must construct the appropriate creation operators to glue together bi-non-crossing partitions.
However, unlike with non-crossing partitions where there is only one way to glue in a full block at any given point, there may be multiple or no ways to glue one bi-non-crossing skeleton into another.
As such, the description of the appropriate creation operators is more complicated.

\label{constructingtheoperatormodel}
Let $z = ((z_i)_{i \in I}, (z_j)_{j \in J})$ be a two-faced family in $(\mc A, \varphi)$.
As before, consider the Fock space $\mathcal{H} := \mathcal{F}\left(\mathbb{C}^{I\sqcup J}\right)$ with $\{e_k\}_{k \in I \sqcup J}$ an orthonormal basis.

For $\alpha : \{1,\ldots, n\} \to I \sqcup J$, we will define operators $T_\alpha \in \mathcal{L}(\mathcal{H})$ which should be thought of as playing the same roll as the operators $T_k$ in our discussion of Nica's model; that is, each adds an appropriate empty skeleton.
Though we will often speak of actions of these operators in terms of their actions on skeletons, one can return to the context of $\mc H$ by letting a partially completed skeleton correspond to the vector formed by taking the tensor product of the basis elements matching the labels of its open nodes, from bottom to top, and weighting it based on which cumulants have been chosen.
For example, the skeleton
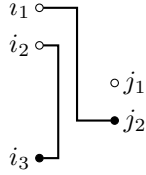
\[
\begin{tikzpicture}[baseline]
	\node[left] at (-.5, 2) {$i_1$};
	\draw (-.5,2) circle (0.05);
	\node[left] at (-.5, 1.5) {$i_2$};
	\draw (-.5,1.5) circle (0.05);
	\node[right] at (.5, 1) {$j_1$};
	\draw (.5,1) circle (0.05);
	\node[right] at (.5, .5) {$j_2$};
	\draw[fill=black] (.5,.5) circle (0.05);
	\node[left] at (-.5,0) {$i_3$};
	\draw[fill=black] (-.5,0) circle (0.05);
	\draw[thick] (-.45,2) -- (0,2) -- (0, 0.5) -- (.45,0.5);
	\draw[thick] (-.45,1.5) -- (-.25,1.5) -- (-.25, 0) -- (-.45,0);
\end{tikzpicture}
\]
corresponds to the vector $e_{j_1}\otimes e_{i_2} \otimes e_{i_1}$ and will be weighted by $\kappa(z_{i_2}, z_{i_3})\kappa(z_{i_1}, z_{j_2})\kappa(z_{j_1})$.
The key point here is that the only choices of future $Z_k$ which yield a non-zero $\Omega$ component when applied to such a vector have annihilation operators in the correct order.
In the above example, in order for this skeleton to make a contribution to the final term, we must act on it by $Z_{j_1}$, $Z_{i_2}$, and $Z_{i_1}$ in that order (though other variables may occur between them).
Since the closed nodes of the skeleton only effect the resulting quantity in terms of its weight and cannot affect the action of future operators (as indeed they must not, for the vector has forgotten them) we will sometimes truncate diagrams of skeletons to show only the open nodes.
It is implied that there may be significantly more nodes and blocks below the bottom of the diagrams that follow, but their representation is eschewed.
Likewise, in order to ensure that $T_\alpha$ is well-defined, we cannot have behaviour depending on which partial skeletons have been chosen, but only the choice of side and of labels of the open nodes.

For $n = 1$, we define $T_\alpha :=  L_{\alpha(1)}$.
In this setting, one may think of $T_\alpha$ as adding an empty skeleton in the lowest possible position with a single open node on the left or on the right depending on whether $\alpha(1)$ is in $I$ or $J$.
For example,
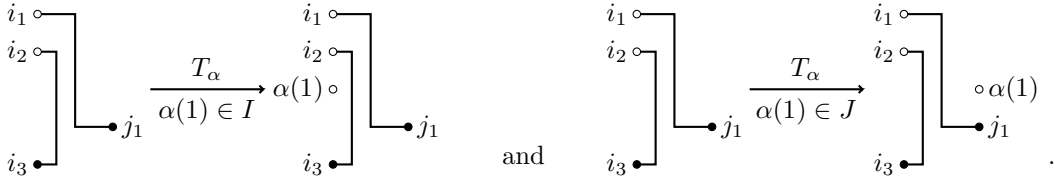
\begin{align*}
\begin{tikzpicture}[baseline]
	\node[left] at (-.5, 2) {$i_1$};
	\draw (-.5,2) circle (0.05);
	\node[left] at (-.5, 1.5) {$i_2$};
	\draw (-.5,1.5) circle (0.05);
	\node[right] at (.5, .5) {$j_1$};
	\draw[fill=black] (.5,.5) circle (0.05);
	\node[left] at (-.5,0) {$i_3$};
	\draw[fill=black] (-.5,0) circle (0.05);
	\draw[thick] (-.45,2) -- (0,2) -- (0, 0.5) -- (.45,0.5);
	\draw[thick] (-.45,1.5) -- (-.25,1.5) -- (-.25, 0) -- (-.45,0);
	\node[above] at (1.75, 1) {$T_\alpha$};
	\node[below] at (1.75, 1) {$\alpha(1) \in I$};
	\draw[thick] (1, 1) -- (2.5, 1) -- (2.45,0.95);
	\draw[thick] (2.5, 1) -- (2.45,1.05);
\end{tikzpicture}
\begin{tikzpicture}[baseline]
	\node[left] at (-.5, 2) {$i_1$};
	\draw (-.5,2) circle (0.05);
	\node[left] at (-.5, 1.5) {$i_2$};
	\draw (-.5,1.5) circle (0.05);
	\node[left] at (-.5, 1) {$\alpha(1)$};
	\draw (-.5,1) circle (0.05);
	\node[right] at (.5, .5) {$j_1$};
	\draw[fill=black] (.5,.5) circle (0.05);
	\node[left] at (-.5,0) {$i_3$};
	\draw[fill=black] (-.5,0) circle (0.05);
	\draw[thick] (-.45,2) -- (0,2) -- (0, 0.5) -- (.45,0.5);
	\draw[thick] (-.45,1.5) -- (-.25,1.5) -- (-.25, 0) -- (-.45,0);
\end{tikzpicture}\qquad\text{and}\qquad
\begin{tikzpicture}[baseline]
	\node[left] at (-.5, 2) {$i_1$};
	\draw (-.5,2) circle (0.05);
	\node[left] at (-.5, 1.5) {$i_2$};
	\draw (-.5,1.5) circle (0.05);
	\node[right] at (.5, .5) {$j_1$};
	\draw[fill=black] (.5,.5) circle (0.05);
	\node[left] at (-.5,0) {$i_3$};
	\draw[fill=black] (-.5,0) circle (0.05);
	\draw[thick] (-.45,2) -- (0,2) -- (0, 0.5) -- (.45,0.5);
	\draw[thick] (-.45,1.5) -- (-.25,1.5) -- (-.25, 0) -- (-.45,0);
	\node[above] at (1.75, 1) {$T_\alpha$};
	\node[below] at (1.75, 1) {$\alpha(1) \in J$};
	\draw[thick] (1, 1) -- (2.5, 1) -- (2.45,0.95);
	\draw[thick] (2.5, 1) -- (2.45,1.05);
\end{tikzpicture}
\begin{tikzpicture}[baseline]
	\node[left] at (-.5, 2) {$i_1$};
	\draw (-.5,2) circle (0.05);
	\node[left] at (-.5, 1.5) {$i_2$};
	\draw (-.5,1.5) circle (0.05);
	\node[right] at (.5, 1) {$\alpha(1)$};
	\draw (.5,1) circle (0.05);
	\node[right] at (.5, .5) {$j_1$};
	\draw[fill=black] (.5,.5) circle (0.05);
	\node[left] at (-.5,0) {$i_3$};
	\draw[fill=black] (-.5,0) circle (0.05);
	\draw[thick] (-.45,2) -- (0,2) -- (0, 0.5) -- (.45,0.5);
	\draw[thick] (-.45,1.5) -- (-.25,1.5) -- (-.25, 0) -- (-.45,0);
\end{tikzpicture}.
\end{align*}
Observe that $T_\alpha$ adds an open node in the lowest valid location (i.e., immediately above all closed nodes); this behaviour will be mimicked by the other $T_\alpha$ as well.
That is, the lowest open node added will always be added directly above the highest closed node.

Let $\Sigma : \mathcal{H} \oplus \mathcal{H} \to \mathcal{H}$ be defined by
\[
\Sigma\left(f_1\otimes\cdots\otimes f_n, f_{n+1} \otimes \cdots \otimes f_{n+m}\right) := \sum_\sigma f_{\sigma(1)}\otimes\cdots\otimes f_{\sigma(n+m)},
\]
where the sum is over all permutations $\sigma \in S_{n+m}$ so that $\sigma|_{[1, n]}$ and $\sigma|_{[n+1, n+m]}$ are increasing; that is, $\sigma$ interleaves the sets $\{1,\ldots, n\}$ and $\{n+1,\ldots, n+m\}$.
Note that $\Sigma(\xi, \Omega) = \xi = \Sigma(\Omega, \xi)$.
As an example,
\begin{align*}
\Sigma(e_{1}\otimes e_{2}, e_{3}\otimes e_4) =& \  e_1 \otimes e_2 \otimes e_3 \otimes e_4 + e_1 \otimes e_3 \otimes e_2 \otimes e_4 + e_3 \otimes e_1 \otimes e_2 \otimes e_4   \\
 & + e_1 \otimes e_3 \otimes e_4 \otimes e_2 + e_3 \otimes e_1 \otimes e_4\otimes e_2 + e_3 \otimes e_4 \otimes e_1 \otimes e_2.
\end{align*}
We will use $\Sigma$ to account for the fact that nodes on the right may be added with any order to nodes on the left to obtain a valid skeleton.

For $\alpha : \{1,\ldots, n\} \to I \sqcup J$ we define
\[
T_{\alpha}(\Omega) := L_{\alpha(n)} L_{\alpha(n-1)} \cdots L_{\alpha(1)}(\Omega) = e_{\alpha(1)}\otimes\cdots\otimes e_{\alpha(n)}.
\]
Note that this corresponds to taking a completed skeleton (possibly with no nodes), and adding the empty skeleton corresponding to $\alpha$ above it.

We will now define $T_{\alpha}$ for $n\geq 2$ on a tensors of basis elements, and extend by linearity to their span (which is dense $\mc H$).
We consider only the case $\alpha(n) \in I$, as the case when $\alpha(n) \in J$ will be similar.
Let $\eta = e_{\beta(m)}\otimes\cdots\otimes e_{\beta(1)} \in \mc H$, where $\beta : \set{1, \ldots, m} \to I \sqcup J$.

If $\beta^{-1}(I) = \emptyset$, we define $T_\alpha(\eta)$ as follows.
Let $k = \max\paren{\set0\cup \alpha^{-1}(J)}$, so that $k$ is the lowest of the nodes to be added by which falls on the right.
We define $T_\alpha(\eta)$ as follows:
\[
T_\alpha(\eta) := e_{\alpha(n)} \otimes \Sigma\paren{ e_{\alpha(n - 1)} \otimes \cdots \otimes e_{\alpha(k+1)}, e_{\beta(m)} \otimes \cdots \otimes e_{\beta(1)} } \otimes e_{\alpha(k)} \otimes \cdots \otimes e_{\alpha(1)}.
\]
This is mimicking the action of adding a new skeleton to the existing skeleton.
In order to ensure that no crossings are introduced, all new nodes on the right must be placed above all existing nodes on the right; before any new right nodes are added, though, nodes on the left can be added freely.
One should think of this as the sum of all valid partially completed skeletons where the old skeleton is below and to the right of starter skeleton corresponding to $\alpha$, with the node corresponding to $\alpha(n)$ in the lowest possible position.

\begin{ex}
If $\alpha : \{1,2,3,4\} \to I \sqcup J$ satisfies $\alpha^{-1}(I) = \{1,3,4\}$ and $\alpha^{-1}(J) = \{2\}$, and $j \in J$, then
\[
T_\alpha(e_{j}) = e_{\alpha(4)}\otimes e_{\alpha(3)} \otimes e_{j} \otimes e_{\alpha(2)} \otimes e_{\alpha(1)} +e_{\alpha(4)}\otimes e_{j} \otimes e_{\alpha(3)} \otimes  e_{\alpha(2)} \otimes e_{\alpha(1)}.
\]
This action corresponds to the following diagram:
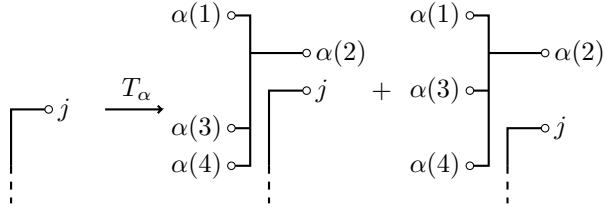
\[
\begin{tikzpicture}[baseline]
	\node[right] at (.5, 1.25) {$j$};
	\draw (.5,1.25) circle (0.05);
	\draw[thick] (.45,1.25) -- (0,1.25) -- (0, .5);
	\draw[thick,dashed] (0,.5) -- (0, 0);
	\node[above] at (1.675, 1.25) {$T_\alpha$};
	\draw[thick] (1.25, 1.25) -- (2, 1.25) -- (1.95, 1.2);
	\draw[thick] (2, 1.25) -- (1.95, 1.3);
	\end{tikzpicture}	
	\begin{tikzpicture}[baseline]
	\node[left] at (-.5, 2.5) {$\alpha(1)$};
	\draw (-.5,2.5) circle (0.05);
	\node[right] at (.5, 2) {$\alpha(2)$};
	\draw (.5,2) circle (0.05);	
	\node[right] at (.5, 1.5) {$j$};
	\draw (.5,1.5) circle (0.05);
	\node[left] at (-.5, 1) {$\alpha(3)$};
	\draw (-.5,1) circle (0.05);
	\node[left] at (-.5, .5) {$\alpha(4)$};
	\draw (-.5,.5) circle (0.05);
	\draw[thick] (-.45,2.5) -- (-.25,2.5) -- (-.25, .5) -- (-.45,.5);
	\draw[thick] (-.45, 1) -- (-.25, 1);
	\draw[thick] (.45, 2) -- (-.25, 2);
	\draw[thick] (.45,1.5) -- (0,1.5) -- (0, 0.5);
	\draw[thick,dashed] (0,.5) -- (0, 0);
	\node[above] at (1.5, 1.25) {$+$};
\end{tikzpicture}
\begin{tikzpicture}[baseline]
	\node[left] at (-.5, 2.5) {$\alpha(1)$};
	\draw (-.5,2.5) circle (0.05);
	\node[right] at (.5, 2) {$\alpha(2)$};
	\draw (.5,2) circle (0.05);	
	\node[left] at (-.5, 1.5) {$\alpha(3)$};
	\draw (-.5,1.5) circle (0.05);
	\node[right] at (.5, 1) {$j$};
	\draw (.5,1) circle (0.05);
	\node[left] at (-.5, 0.5) {$\alpha(4)$};
	\draw (-.5,.5) circle (0.05);
	\draw[thick] (-.45,2.5) -- (-.25,2.5) -- (-.25, .5) -- (-.45,.5);
	\draw[thick] (-.45, 1.5) -- (-.25, 1.5);
	\draw[thick] (.45, 2) -- (-.25, 2);
	\draw[thick] (.45,1) -- (0,1) -- (0, 0.5);
	\draw[thick,dashed] (0,.5) -- (0, 0);
\end{tikzpicture}
\]
The purpose of allowing multiple diagrams is that the cumulant corresponding to a bi-non-crossing diagram for a sequence of operators is equal to the same cumulant for the sequence of operators obtained by interchanging the $k$-th and $(k+1)$-th operators and the $k$-th and $(k+1)$-th nodes in the bi-non-crossing diagram provided $k$ and $k+1$ are in different blocks and on different sides of the diagram.
In the end, a sequence of annihilation operators can complete at most one skeleton and will produce the correct completed skeleton for a given sequence of operators.

As a further example, suppose $\alpha : \{1,2,3\} \to I$ and $j_1, j_2 \in J$.
Then
\begin{align*}
T_\alpha(e_{j_2}\otimes e_{j_1})
=&\ e_{\alpha(3)} \otimes e_{\alpha(2)} \otimes e_{\alpha(1)} \otimes e_{j_2} \otimes e_{j_1}
+ e_{\alpha(3)} \otimes e_{\alpha(2)} \otimes e_{j_2} \otimes e_{\alpha(1)} \otimes e_{j_1}\\
& + e_{\alpha(3)} \otimes e_{\alpha(2)} \otimes e_{j_2} \otimes e_{j_1} \otimes e_{\alpha(1)}
+ e_{\alpha(3)} \otimes e_{j_2} \otimes e_{\alpha(2)} \otimes e_{\alpha(1)} \otimes e_{j_1} \\
& + e_{\alpha(3)} \otimes e_{j_2} \otimes e_{\alpha(2)} \otimes e_{j_1} \otimes e_{\alpha(1)}
+ e_{\alpha(3)} \otimes e_{j_2} \otimes e_{j_1} \otimes e_{\alpha(2)} \otimes e_{\alpha(1)}
\end{align*}
This action corresponds to the following diagram:
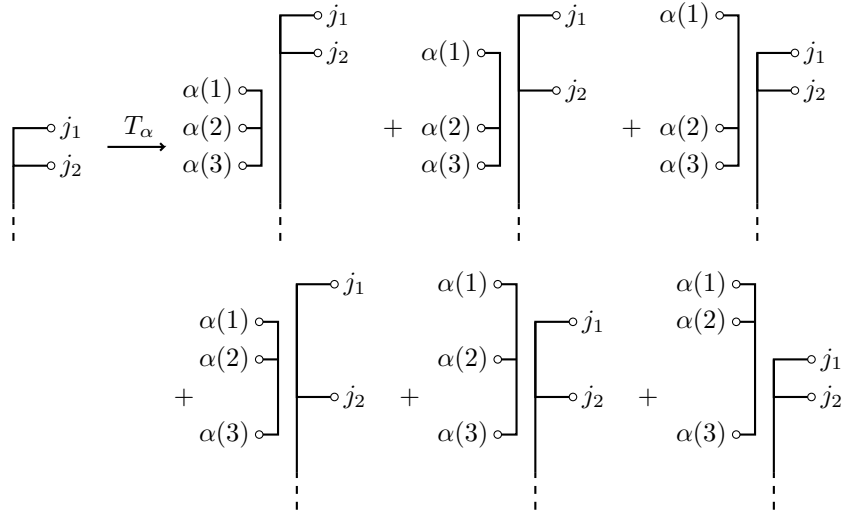
\begin{align*}
\begin{tikzpicture}[baseline]
	\draw[thick,dashed] (0, 0) -- (0,.5);
	\draw[thick] (.5,1.5) -- (0,1.5) -- (0,1) -- (.5,1);
	\draw[thick] (0,.5) -- (0,1);
	\draw[fill=white] (.5,1) circle (0.05);
	\node[right] at (.5,1) {$j_2$};
	\draw[fill=white] (.5,1.5) circle (0.05);
	\node[right] at (.5,1.5) {$j_1$};
	\node[above] at (1.675, 1.25) {$T_\alpha$};
	\draw[thick] (1.25, 1.25) -- (2, 1.25) -- (1.95, 1.2);
	\draw[thick] (2, 1.25) -- (1.95, 1.3);
\end{tikzpicture}&\ 
\begin{tikzpicture}[baseline]
	\draw[thick,dashed] (0, 0) -- (0,.5);
	\draw[thick] (.5,2.5) -- (0,2.5) -- (0,3) -- (.5,3);
	\draw[thick] (-.5,2) -- (-.25,2) -- (-.25,1) -- (-.5,1);
	\draw[thick] (-.5,1.5) -- (-.25,1.5);
	\draw[thick] (0,.5) -- (0,3);
	\draw[fill=white] (-.5,1) circle (0.05);
	\node[left] at (-.5,1) {$\alpha(3)$};
	\draw[fill=white] (-.5,1.5) circle (0.05);
	\node[left] at (-.5,1.5) {$\alpha(2)$};
	\draw[fill=white] (-.5,2) circle (0.05);
	\node[left] at (-.5,2) {$\alpha(1)$};
	\draw[fill=white] (.5,2.5) circle (0.05);
	\node[right] at (.5,2.5) {$j_2$};
	\draw[fill=white] (.5,3) circle (0.05);
	\node[right] at (.5,3) {$j_1$};
	\node[above] at (1.5, 1.25) {$+$};
\end{tikzpicture}
\begin{tikzpicture}[baseline]
	\draw[thick,dashed] (0, 0) -- (0,.5);
	\draw[thick] (.5,2) -- (0,2) -- (0,3) -- (.5,3);
	\draw[thick] (-.5,2.5) -- (-.25,2.5) -- (-.25,1) -- (-.5,1);
	\draw[thick] (-.5,1.5) -- (-.25,1.5);
	\draw[thick] (0,.5) -- (0,3);
	\draw[fill=white] (-.5,1) circle (0.05);
	\node[left] at (-.5,1) {$\alpha(3)$};
	\draw[fill=white] (-.5,1.5) circle (0.05);
	\node[left] at (-.5,1.5) {$\alpha(2)$};
	\draw[fill=white] (-.5,2.5) circle (0.05);
	\node[left] at (-.5,2.5) {$\alpha(1)$};
	\draw[fill=white] (.5,2) circle (0.05);
	\node[right] at (.5,2) {$j_2$};
	\draw[fill=white] (.5,3) circle (0.05);
	\node[right] at (.5,3) {$j_1$};
	\node[above] at (1.5, 1.25) {$+$};
\end{tikzpicture}
\begin{tikzpicture}[baseline]
	\draw[thick,dashed] (0, 0) -- (0,.5);
	\draw[thick] (.5,2) -- (0,2) -- (0,2.5) -- (.5,2.5);
	\draw[thick] (-.5,3) -- (-.25,3) -- (-.25,1) -- (-.5,1);
	\draw[thick] (-.5,1.5) -- (-.25,1.5);
	\draw[thick] (0,.5) -- (0,2.5);
	\draw[fill=white] (-.5,1) circle (0.05);
	\node[left] at (-.5,1) {$\alpha(3)$};
	\draw[fill=white] (-.5,1.5) circle (0.05);
	\node[left] at (-.5,1.5) {$\alpha(2)$};
	\draw[fill=white] (-.5,3) circle (0.05);
	\node[left] at (-.5,3) {$\alpha(1)$};
	\draw[fill=white] (.5,2) circle (0.05);
	\node[right] at (.5,2) {$j_2$};
	\draw[fill=white] (.5,2.5) circle (0.05);
	\node[right] at (.5,2.5) {$j_1$};
\end{tikzpicture}\\ &
\begin{tikzpicture}[baseline]
	\draw[thick,dashed] (0, 0) -- (0,.5);
	\draw[thick] (.5,1.5) -- (0,1.5) -- (0,3) -- (.5,3);
	\draw[thick] (-.5,2.5) -- (-.25,2.5) -- (-.25,1) -- (-.5,1);
	\draw[thick] (-.5,2) -- (-.25,2);
	\draw[thick] (0,.5) -- (0,3);
	\draw[fill=white] (-.5,1) circle (0.05);
	\node[left] at (-.5,1) {$\alpha(3)$};
	\draw[fill=white] (-.5,2) circle (0.05);
	\node[left] at (-.5,2) {$\alpha(2)$};
	\draw[fill=white] (-.5,2.5) circle (0.05);
	\node[left] at (-.5,2.5) {$\alpha(1)$};
	\draw[fill=white] (.5,1.5) circle (0.05);
	\node[right] at (.5,1.5) {$j_2$};
	\draw[fill=white] (.5,3) circle (0.05);
	\node[right] at (.5,3) {$j_1$};
	\node[above] at (1.5, 1.25) {$+$};
	\node[above] at (-1.5, 1.25) {$+$};
\end{tikzpicture}
\begin{tikzpicture}[baseline]
	\draw[thick,dashed] (0, 0) -- (0,.5);
	\draw[thick] (.5,1.5) -- (0,1.5) -- (0,2.5) -- (.5,2.5);
	\draw[thick] (-.5,3) -- (-.25,3) -- (-.25,1) -- (-.5,1);
	\draw[thick] (-.5,2) -- (-.25,2);
	\draw[thick] (0,.5) -- (0,2.5);
	\draw[fill=white] (-.5,1) circle (0.05);
	\node[left] at (-.5,1) {$\alpha(3)$};
	\draw[fill=white] (-.5,2) circle (0.05);
	\node[left] at (-.5,2) {$\alpha(2)$};
	\draw[fill=white] (-.5,3) circle (0.05);
	\node[left] at (-.5,3) {$\alpha(1)$};
	\draw[fill=white] (.5,1.5) circle (0.05);
	\node[right] at (.5,1.5) {$j_2$};
	\draw[fill=white] (.5,2.5) circle (0.05);
	\node[right] at (.5,2.5) {$j_1$};
	\node[above] at (1.5, 1.25) {$+$};
\end{tikzpicture}
\begin{tikzpicture}[baseline]
	\draw[thick,dashed] (0, 0) -- (0,.5);
	\draw[thick] (.5,1.5) -- (0,1.5) -- (0,2) -- (.5,2);
	\draw[thick] (-.5,3) -- (-.25,3) -- (-.25,1) -- (-.5,1);
	\draw[thick] (-.5,2.5) -- (-.25,2.5);
	\draw[thick] (0,.5) -- (0,2);
	\draw[fill=white] (-.5,1) circle (0.05);
	\node[left] at (-.5,1) {$\alpha(3)$};
	\draw[fill=white] (-.5,2.5) circle (0.05);
	\node[left] at (-.5,2.5) {$\alpha(2)$};
	\draw[fill=white] (-.5,3) circle (0.05);
	\node[left] at (-.5,3) {$\alpha(1)$};
	\draw[fill=white] (.5,1.5) circle (0.05);
	\node[right] at (.5,1.5) {$j_2$};
	\draw[fill=white] (.5,2) circle (0.05);
	\node[right] at (.5,2) {$j_1$};
\end{tikzpicture}
.
\end{align*}
\end{ex}

Now, suppose that $\beta^{-1}(I) \neq \emptyset$, and let $k = \max\paren{\beta^{-1}(I)}$.
This corresponds to a partially completed skeleton with open nodes on both the left and right, where the lowest open node on the left is the $k^\mathrm{th}$ from the top.
We set $T_{\alpha}(\eta) = 0$ if $\alpha(t) \in J$ for some $t$, since the partially completed skeleton has open nodes on the left and right we cannot add the empty skeleton of $\alpha$ without introducing a crossing, since the lowest node of $\alpha$ is on the left.
Otherwise $\alpha(t) \in I$ for all $t \in \set{1,\ldots, n}$, and we set
\[T_{\alpha}(\eta) := e_{\alpha(n)}\otimes \Sigma\left( e_{\alpha(n-1)} \otimes \cdots \otimes e_{\alpha(1)}, e_{\beta(m)} \otimes \cdots \otimes e_{\beta(k+1)} \right) \otimes e_{\beta(k)} \otimes \cdots \otimes e_{\beta(1)}.
\]
One can think of this as the sum of all valid partially completed skeletons where the empty skeleton of $\alpha$ sits below the lowest open node on the left of the old skeleton.

\begin{ex}
If $\alpha : \set{1, 2} \to I \sqcup J$ has $\alpha(2) \in I$, $\alpha(1) \in J$ and $i \in I$, then
$$T_\alpha(e_i) = 0.$$
This is because there is no way to glue the empty skeleton corresponding to $\alpha$ into the partially completed skeleton without introducing a crossing while placing the lowest node of $\alpha$ at the bottom of the diagram (directly above the highest closed node):
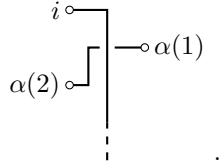
\[
\begin{tikzpicture}[baseline]
	\draw[thick,dashed] (0, 0) -- (0,.5);
	\draw[thick] (-.5,2) -- (0,2) -- (0,.5);
	\draw[thick] (-.5,1) -- (-.25, 1) -- (-.25,1.5) -- (-.1, 1.5);
	\draw[thick] (.1,1.5) -- (.5, 1.5);
	\draw[fill=white] (-.5,1) circle (0.05);
	\node[left] at (-.5,1) {$\alpha(2)$};
	\draw[fill=white] (-.5,2) circle (0.05);
	\node[left] at (-.5,2) {$i$};
	\draw[fill=white] (.5,1.5) circle (0.05);
	\node[right] at (.5,1.5) {$\alpha(1)$};
\end{tikzpicture}
.
\]

If $\alpha : \{1,2\} \to I$ and $i \in I$, $j, j' \in J$ then
\[
T_\alpha \left(e_j\otimes e_i\otimes e_{j'}\right) = e_{\alpha(2)} \otimes e_{\alpha(1)} \otimes e_{j}\otimes e_{i} \otimes e_{j'} + e_{\alpha(2)} \otimes e_{j}\otimes e_{\alpha(1)} \otimes  e_{i} \otimes e_{j'}.
\]
This action corresponds to the following diagram:
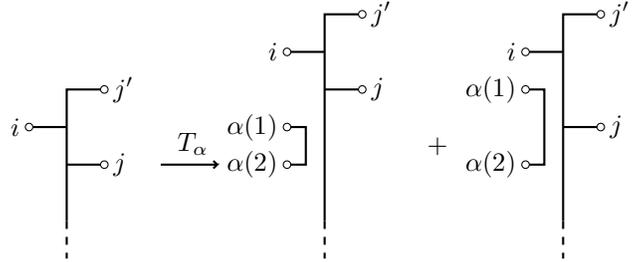
\[
\begin{tikzpicture}[baseline]
	\node[right] at (.5, 2.25) {$j'$};
	\draw (.5,2.25) circle (0.05);
	\node[left] at (-.5, 1.75) {$i$};
	\draw (-.5,1.75) circle (0.05);	
	\node[right] at (.5, 1.25) {$j$};
	\draw (.5,1.25) circle (0.05);
	\draw[thick] (.45,2.25) -- (0,2.25) -- (0, .5);
	\draw[thick] (-.45, 1.75) -- (0, 1.75);
	\draw[thick] (.45, 1.25) -- (0, 1.25);
	\draw[thick,dashed] (0, .5) -- (0, 0);
	\node[above] at (1.675, 1.25) {$T_\alpha$};
	\draw[thick] (1.25, 1.25) -- (2, 1.25) -- (1.95, 1.2);
	\draw[thick] (2, 1.25) -- (1.95, 1.3);
\end{tikzpicture}
\begin{tikzpicture}[baseline]
	\node[right] at (.5, 3.25) {$j'$};
	\draw (.5,3.25) circle (0.05);
	\node[left] at (-.5, 2.75) {$i$};
	\draw (-.5,2.75) circle (0.05);	
	\node[right] at (.5, 2.25) {$j$};
	\draw (.5,2.25) circle (0.05);
	\node[left] at (-.5, 1.75) {$\alpha(1)$};
	\draw (-.5,1.75) circle (0.05);
	\node[left] at (-.5, 1.25) {$\alpha(2)$};
	\draw (-.5,1.25) circle (0.05);
	\draw[thick] (.45,3.25) -- (0,3.25) -- (0, .5);
	\draw[thick] (-.45, 2.75) -- (0, 2.75);
	\draw[thick] (.45, 2.25) -- (0, 2.25);
	\draw[thick] (-.45, 1.75) -- (-.25, 1.75) -- (-.25, 1.25) -- (-.45, 1.25);
	\draw[thick,dashed] (0, .5) -- (0, 0);
	\node[above] at (1.5, 1.25) {$+$};
\end{tikzpicture}
\begin{tikzpicture}[baseline]
	\node[right] at (.5, 3.25) {$j'$};
	\draw (.5,3.25) circle (0.05);
	\node[left] at (-.5, 2.75) {$i$};
	\draw (-.5,2.75) circle (0.05);	
	\node[left] at (-.5, 2.25) {$\alpha(1)$};
	\draw (-.5,2.25) circle (0.05);
	\node[right] at (.5, 1.75) {$j$};
	\draw (.5,1.75) circle (0.05);
	\node[left] at (-.5, 1.25) {$\alpha(2)$};
	\draw (-.5,1.25) circle (0.05);
	\draw[thick] (.45,3.25) -- (0,3.25) -- (0, .5);
	\draw[thick] (-.45, 2.75) -- (0, 2.75);
	\draw[thick] (.45, 1.75) -- (0, 1.75);
	\draw[thick,dashed] (0, .5) -- (0, 0);
	\draw[thick] (-.45,2.25) -- (-.25,2.25) -- (-.25, 1.25) -- (-.45,1.25);
	\end{tikzpicture}
.
\]
\end{ex}

As $T_{\alpha}$ has been defined on an orthonormal basis, we may extend by linearity to obtain a densely defined operator on $\mathcal{H}$; note that $T_\alpha$ may not be bounded due to the action of $\Sigma$.
On the other hand, if $\alpha : \{1, \ldots, n\} \to I$ then $T_{\alpha}$ acts on the Fock subspace generated by $\set{e_i}_{i\in I}$ as $ L_{\alpha(n)} \cdots L_{\alpha(1)}$.
Thus, if one considers only left variables, the resulting operators are precisely those of Nica's model.

We define $T_\alpha$ in a similar manner when $\alpha(n) \in J$.

\subsection{The operator model for pairs of faces}
With the above construction, the operator model for a pair of faces is at hand.

\begin{thm}
\label{operatormodelforapairoffaces}
Let $z = \left( \{z_i\}_{i \in I}, \{z_j\}_{j \in J}\right)$ be a pair of faces in a non-commutative probability space $(\mathcal{A}, \varphi)$.
With notation as in Construction \ref{constructingtheoperatormodel}, consider the (unbounded) operator
\[
\Theta_z := I + \sum_{n\geq 1}\sum_{\alpha : \{1,\ldots, n\} \to I \sqcup J} \kappa_\alpha(z) T_{\alpha},
\]
and for $k \in I \sqcup J$, set $Z_k := L_k^*\Theta_z$.
If $T \in \mathrm{alg}(\{Z_k\}_{k \in I \sqcup J})$ then $\langle T\Omega, \Omega\rangle$ is well-defined.
Moreover, if $\omega(T) = \langle T\Omega,\Omega\rangle$, the joint distribution of $\{Z_k\}_{k \in I \sqcup J}$ with respect to $\omega$ is the same as the joint distribution of $z$ with respect to $\varphi$.
\end{thm}

Before we begin the proof, we give the following example.
\begin{ex}
\label{exampletoshowoperatormodelworks}
In this example, let $I = \{1\}$ and $J = \{2\}$.
We will examine how the completed skeleton below is constructed for $Z_1Z_2Z_1Z_1Z_2Z_2Z_1Z_2Z_1Z_1$.
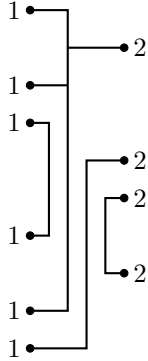
\begin{align*}
\begin{tikzpicture}[baseline]
	\node[left] at (-.75, 4.5) {$1$};
	\draw[fill=black] (-.75,4.5) circle (0.05);
	\node[right] at (.5, 4) {$2$};
	\draw[fill=black] (.5,4) circle (0.05);
	\node[left] at (-.75, 3.5) {$1$};
	\draw[fill=black] (-.75,3.5) circle (0.05);
	\node[left] at (-.75, 3) {$1$};
	\draw[fill=black] (-.75,3) circle (0.05);
	\node[right] at (.5, 2.5) {$2$};
	\draw[fill=black] (.5,2.5) circle (0.05);
	\node[right] at (.5, 2) {$2$};
	\draw[fill=black] (.5,2) circle (0.05);
	\node[left] at (-.75, 1.5) {$1$};
	\draw[fill=black] (-.75,1.5) circle (0.05);
	\node[right] at (.5, 1) {$2$};
	\draw[fill=black] (.5,1) circle (0.05);
	\node[left] at (-.75, .5) {$1$};
	\draw[fill=black] (-.75,.5) circle (0.05);
	\node[left] at (-.75,0) {$1$};
	\draw[fill=black] (-.75,0) circle (0.05);
	\draw[thick] (-0.7,4.5) -- (-.25,4.5) -- (-.25, 0.5) -- (-.7, 0.5);
	\draw[thick] (-.25, 4) -- (.45, 4);
	\draw[thick] (-.25, 3.5) -- (-.7, 3.5);
	\draw[thick] (-0.7, 3) -- (-.5, 3) -- (-.5, 1.5) -- (-.7, 1.5);
	\draw[thick] (0.45, 1) -- (.25, 1) -- (.25, 2) -- (.45, 2);
	\draw[thick] (0.45, 2.5) -- (0, 2.5) -- (0, 0) -- (-.7, 0);
	\end{tikzpicture}
\end{align*}
First $\kappa_{(21)}(z) L_1^*T_{(21)}$ is applied to get the partially completed skeleton
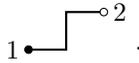
\begin{align*}
\begin{tikzpicture}[baseline]
	\node[right] at (.5, .5) {$2$};
	\draw (.5,.5) circle (0.05);
	\node[left] at (-.5,0) {$1$};
	\draw[fill=black] (-.5,0) circle (0.05);
	\draw[thick] (.45,.5) -- (0,.5) -- (0, 0) -- (-.45,0);
	\end{tikzpicture}.
\end{align*}
Then $\kappa_{(1211)}(z) L_1^*T_{(1211)}$ is applied to obtain the following collection of partially completed skeletons
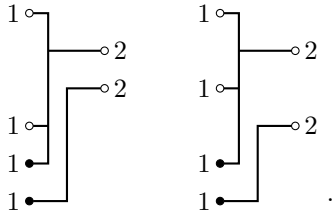
\[
\begin{tikzpicture}[baseline]
	\node[left] at (-.5, 2.5) {1};
	\draw (-.5,2.5) circle (0.05);
	\node[right] at (.5, 2) {2};
	\draw (.5,2) circle (0.05);
	\node[right] at (.5, 1.5) {2};
	\draw (.5,1.5) circle (0.05);
	\node[left] at (-.5, 1) {1};
	\draw(-.5,1) circle (0.05);
	\node[left] at (-.5, .5) {1};
	\draw[fill=black] (-.5,.5) circle (0.05);
	\node[left] at (-.5,0) {1};
	\draw[fill=black] (-.5,0) circle (0.05);
	\draw[thick] (-.45,2.5) -- (-.25,2.5) -- (-0.25, .5) -- (-.45,.5);
	\draw[thick] (-.45,1) -- (-.25,1);
	\draw[thick] (.45,2) -- (-.25,2);
	\draw[thick] (-.45,0) -- (0,0) -- (0, 1.5) -- (.45,1.5);
	\end{tikzpicture}
	\qquad
	\begin{tikzpicture}[baseline]
	\node[left] at (-.5, 2.5) {1};
	\draw (-.5,2.5) circle (0.05);
	\node[right] at (.5, 2) {2};
	\draw (.5,2) circle (0.05);
	\node[left] at (-.5, 1.5) {1};
	\draw(-.5,1.5) circle (0.05);
	\node[right] at (.5, 1) {2};
	\draw (.5,1) circle (0.05);
	\node[left] at (-.5, .5) {1};
	\draw[fill=black] (-.5,.5) circle (0.05);
	\node[left] at (-.5,0) {1};
	\draw[fill=black] (-.5,0) circle (0.05);
	\draw[thick] (-.45,2.5) -- (-.25,2.5) -- (-0.25, .5) -- (-.45,.5);
	\draw[thick] (-.45,1.5) -- (-.25,1.5);
	\draw[thick] (.45,2) -- (-.25,2);
	\draw[thick] (-.45,0) -- (0,0) -- (0, 1) -- (.45,1);
	\end{tikzpicture}.
\]
Applying $\kappa_{(22)}L_2^*T_{(22)}$ then gives the following collection of partially completed skeletons (where the first two below are from the first above and the third below is from the second above)
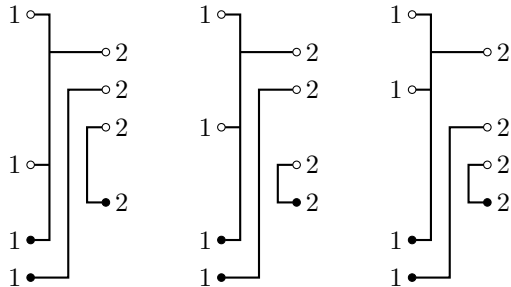
\[
\begin{tikzpicture}[baseline]
	\node[left] at (-.5, 3.5) {1};
	\draw (-.5,3.5) circle (0.05);
	\node[right] at (.5, 3) {2};
	\draw (.5,3) circle (0.05);
	\node[right] at (.5, 2.5) {2};
	\draw (.5,2.5) circle (0.05);
	\node[right] at (.5, 2) {2};
	\draw (.5,2) circle (0.05);	
	\node[left] at (-.5, 1.5) {1};
	\draw(-.5,1.5) circle (0.05);
	\node[right] at (.5, 1) {2};
	\draw[fill=black] (.5,1) circle (0.05);
	\node[left] at (-.5, .5) {1};
	\draw[fill=black] (-.5,.5) circle (0.05);
	\node[left] at (-.5,0) {1};
	\draw[fill=black] (-.5,0) circle (0.05);
	\draw[thick] (-.45,3.5) -- (-.25,3.5) -- (-0.25, .5) -- (-.45,.5);
	\draw[thick] (-.45,1.5) -- (-.25,1.5);
	\draw[thick] (.45,3) -- (-.25,3);
	\draw[thick] (-.45,0) -- (0,0) -- (0, 2.5) -- (.45,2.5);
	\draw[thick] (.45, 1) -- (.25, 1) -- (.25, 2) -- (.45, 2);
\end{tikzpicture}
\qquad
\begin{tikzpicture}[baseline]
	\node[left] at (-.5, 3.5) {1};
	\draw (-.5,3.5) circle (0.05);
	\node[right] at (.5, 3) {2};
	\draw (.5,3) circle (0.05);
	\node[right] at (.5, 2.5) {2};
	\draw (.5,2.5) circle (0.05);
	\node[left] at (-.5, 2) {1};
	\draw(-.5,2) circle (0.05);
	\node[right] at (.5, 1.5) {2};
	\draw (.5,1.5) circle (0.05);
	\node[right] at (.5, 1) {2};
	\draw[fill=black] (.5,1) circle (0.05);
	\node[left] at (-.5, .5) {1};
	\draw[fill=black] (-.5,.5) circle (0.05);
	\node[left] at (-.5,0) {1};
	\draw[fill=black] (-.5,0) circle (0.05);
	\draw[thick] (-.45,3.5) -- (-.25,3.5) -- (-0.25, .5) -- (-.45,.5);
	\draw[thick] (-.45,2) -- (-.25,2);
	\draw[thick] (.45,3) -- (-.25,3);
	\draw[thick] (-.45,0) -- (0,0) -- (0, 2.5) -- (.45,2.5);
	\draw[thick] (.45, 1) -- (.25, 1) -- (.25, 1.5) -- (.45, 1.5);
	\end{tikzpicture}
\qquad
	\begin{tikzpicture}[baseline]
	\node[left] at (-.5, 3.5) {1};
	\draw (-.5,3.5) circle (0.05);
	\node[right] at (.5, 3) {2};
	\draw (.5,3) circle (0.05);
	\node[left] at (-.5, 2.5) {1};
	\draw(-.5,2.5) circle (0.05);
	\node[right] at (.5, 2) {2};
	\draw (.5,2) circle (0.05);
	\node[right] at (.5, 1.5) {2};
	\draw (.5,1.5) circle (0.05);
	\node[right] at (.5, 1) {2};
	\draw[fill=black] (.5,1) circle (0.05);
	\node[left] at (-.5, .5) {1};
	\draw[fill=black] (-.5,.5) circle (0.05);
	\node[left] at (-.5,0) {1};
	\draw[fill=black] (-.5,0) circle (0.05);
	\draw[thick] (-.45,3.5) -- (-.25,3.5) -- (-0.25, .5) -- (-.45,.5);
	\draw[thick] (-.45,2.5) -- (-.25,2.5);
	\draw[thick] (.45,3) -- (-.25,3);
	\draw[thick] (-.45,0) -- (0,0) -- (0, 2) -- (.45,2);
	\draw[thick] (.45, 1) -- (.25, 1) -- (.25, 1.5) -- (.45, 1.5);
	\end{tikzpicture}
\]
and applying $\kappa_{(11)}L_1^*T_{(11)}$ then gives the following collection of partially completed skeletons (where the first below is from the first above, the second and third below are from the second above, and the last three are from the third above).
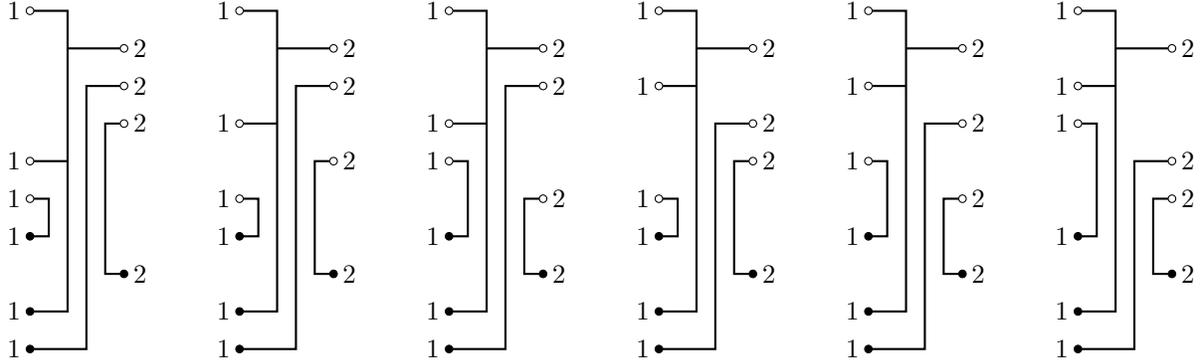
\[
\begin{tikzpicture}[baseline]
	\node[left] at (-.75, 4.5) {1};
	\draw (-.75,4.5) circle (0.05);
	\node[right] at (.5, 4) {2};
	\draw (.5,4) circle (0.05);
	\node[right] at (.5, 3.5) {2};
	\draw (.5,3.5) circle (0.05);
	\node[right] at (.5, 3) {2};
	\draw (.5,3) circle (0.05);	
	\node[left] at (-.75, 2.5) {1};
	\draw(-.75,2.5) circle (0.05);
	\node[left] at (-.75, 2) {1};
	\draw(-.75,2) circle (0.05);
	\node[left] at (-.75, 1.5) {1};
	\draw[fill=black](-.75,1.5) circle (0.05);
	\node[right] at (.5, 1) {2};
	\draw[fill=black] (.5,1) circle (0.05);
	\node[left] at (-.75, .5) {1};
	\draw[fill=black] (-.75,.5) circle (0.05);
	\node[left] at (-.75,0) {1};
	\draw[fill=black] (-.75,0) circle (0.05);
	\draw[thick] (-.7,4.5) -- (-.25,4.5) -- (-0.25, .5) -- (-.7,.5);
	\draw[thick] (-.7,2.5) -- (-.25,2.5);
	\draw[thick] (.45,4) -- (-.25,4);
	\draw[thick] (-.7,0) -- (0,0) -- (0, 3.5) -- (.45,3.5);
	\draw[thick] (.45, 1) -- (.25, 1) -- (.25, 3) -- (.45, 3);
	\draw[thick] (-.7, 1.5) -- (-.5, 1.5) -- (-.5, 2) -- (-.7, 2);
\end{tikzpicture}
\qquad
\begin{tikzpicture}[baseline]
	\node[left] at (-.75, 4.5) {1};
	\draw (-.75,4.5) circle (0.05);
	\node[right] at (.5, 4) {2};
	\draw (.5,4) circle (0.05);
	\node[right] at (.5, 3.5) {2};
	\draw (.5,3.5) circle (0.05);
	\node[left] at (-.75, 3) {1};
	\draw(-.75,3) circle (0.05);
	\node[right] at (.5, 2.5) {2};
	\draw (.5,2.5) circle (0.05);	
	\node[left] at (-.75, 2) {1};
	\draw(-.75,2) circle (0.05);
	\node[left] at (-.75, 1.5) {1};
	\draw[fill=black](-.75,1.5) circle (0.05);
	\node[right] at (.5, 1) {2};
	\draw[fill=black] (.5,1) circle (0.05);
	\node[left] at (-.75, .5) {1};
	\draw[fill=black] (-.75,.5) circle (0.05);
	\node[left] at (-.75,0) {1};
	\draw[fill=black] (-.75,0) circle (0.05);
	\draw[thick] (-.7,4.5) -- (-.25,4.5) -- (-0.25, .5) -- (-.7,.5);
	\draw[thick] (-.7,3) -- (-.25,3);
	\draw[thick] (.45,4) -- (-.25,4);
	\draw[thick] (-.7,0) -- (0,0) -- (0, 3.5) -- (.45,3.5);
	\draw[thick] (.45, 1) -- (.25, 1) -- (.25, 2.5) -- (.45, 2.5);
	\draw[thick] (-.7, 1.5) -- (-.5, 1.5) -- (-.5, 2) -- (-.7, 2);
\end{tikzpicture}
\qquad
\begin{tikzpicture}[baseline]
	\node[left] at (-.75, 4.5) {1};
	\draw (-.75,4.5) circle (0.05);
	\node[right] at (.5, 4) {2};
	\draw (.5,4) circle (0.05);
	\node[right] at (.5, 3.5) {2};
	\draw (.5,3.5) circle (0.05);
	\node[left] at (-.75, 3) {1};
	\draw(-.75,3) circle (0.05);
	\node[left] at (-.75, 2.5) {1};
	\draw(-.75,2.5) circle (0.05);
	\node[right] at (.5, 2) {2};
	\draw (.5,2) circle (0.05);	
	\node[left] at (-.75, 1.5) {1};
	\draw[fill=black](-.75,1.5) circle (0.05);
	\node[right] at (.5, 1) {2};
	\draw[fill=black] (.5,1) circle (0.05);
	\node[left] at (-.75, .5) {1};
	\draw[fill=black] (-.75,.5) circle (0.05);
	\node[left] at (-.75,0) {1};
	\draw[fill=black] (-.75,0) circle (0.05);
	\draw[thick] (-.7,4.5) -- (-.25,4.5) -- (-0.25, .5) -- (-.7,.5);
	\draw[thick] (-.7,3) -- (-.25,3);
	\draw[thick] (.45,4) -- (-.25,4);
	\draw[thick] (-.7,0) -- (0,0) -- (0, 3.5) -- (.45,3.5);
	\draw[thick] (.45, 1) -- (.25, 1) -- (.25, 2) -- (.45, 2);
	\draw[thick] (-.7, 1.5) -- (-.5, 1.5) -- (-.5, 2.5) -- (-.7, 2.5);
\end{tikzpicture}
\qquad
\begin{tikzpicture}[baseline]
	\node[left] at (-.75, 4.5) {1};
	\draw (-.75,4.5) circle (0.05);
	\node[right] at (.5, 4) {2};
	\draw (.5,4) circle (0.05);
	\node[left] at (-.75, 3.5) {1};
	\draw(-.75,3.5) circle (0.05);
	\node[right] at (.5, 3) {2};
	\draw (.5,3) circle (0.05);
	\node[right] at (.5, 2.5) {2};
	\draw (.5,2.5) circle (0.05);	
	\node[left] at (-.75, 2) {1};
	\draw(-.75,2) circle (0.05);
	\node[left] at (-.75, 1.5) {1};
	\draw[fill=black](-.75,1.5) circle (0.05);
	\node[right] at (.5, 1) {2};
	\draw[fill=black] (.5,1) circle (0.05);
	\node[left] at (-.75, .5) {1};
	\draw[fill=black] (-.75,.5) circle (0.05);
	\node[left] at (-.75,0) {1};
	\draw[fill=black] (-.75,0) circle (0.05);
	\draw[thick] (-.7,4.5) -- (-.25,4.5) -- (-0.25, .5) -- (-.7,.5);
	\draw[thick] (-.7,3.5) -- (-.25,3.5);
	\draw[thick] (.45,4) -- (-.25,4);
	\draw[thick] (-.7,0) -- (0,0) -- (0, 3) -- (.45,3);
	\draw[thick] (.45, 1) -- (.25, 1) -- (.25, 2.5) -- (.45, 2.5);
	\draw[thick] (-.7, 1.5) -- (-.5, 1.5) -- (-.5, 2) -- (-.7, 2);
\end{tikzpicture}
\qquad
\begin{tikzpicture}[baseline]
	\node[left] at (-.75, 4.5) {1};
	\draw (-.75,4.5) circle (0.05);
	\node[right] at (.5, 4) {2};
	\draw (.5,4) circle (0.05);
	\node[left] at (-.75, 3.5) {1};
	\draw(-.75,3.5) circle (0.05);
	\node[right] at (.5, 3) {2};
	\draw (.5,3) circle (0.05);
	\node[left] at (-.75, 2.5) {1};
	\draw(-.75,2.5) circle (0.05);
	\node[right] at (.5, 2) {2};
	\draw (.5,2) circle (0.05);	
	\node[left] at (-.75, 1.5) {1};
	\draw[fill=black](-.75,1.5) circle (0.05);
	\node[right] at (.5, 1) {2};
	\draw[fill=black] (.5,1) circle (0.05);
	\node[left] at (-.75, .5) {1};
	\draw[fill=black] (-.75,.5) circle (0.05);
	\node[left] at (-.75,0) {1};
	\draw[fill=black] (-.75,0) circle (0.05);
	\draw[thick] (-.7,4.5) -- (-.25,4.5) -- (-0.25, .5) -- (-.7,.5);
	\draw[thick] (-.7,3.5) -- (-.25,3.5);
	\draw[thick] (.45,4) -- (-.25,4);
	\draw[thick] (-.7,0) -- (0,0) -- (0, 3) -- (.45,3);
	\draw[thick] (.45, 1) -- (.25, 1) -- (.25, 2) -- (.45, 2);
	\draw[thick] (-.7, 1.5) -- (-.5, 1.5) -- (-.5, 2.5) -- (-.7, 2.5);
\end{tikzpicture}
\qquad
\begin{tikzpicture}[baseline]
	\node[left] at (-.75, 4.5) {1};
	\draw (-.75,4.5) circle (0.05);
	\node[right] at (.5, 4) {2};
	\draw (.5,4) circle (0.05);
	\node[left] at (-.75, 3.5) {1};
	\draw(-.75,3.5) circle (0.05);
	\node[left] at (-.75, 3) {1};
	\draw(-.75,3) circle (0.05);
	\node[right] at (.5, 2.5) {2};
	\draw (.5,2.5) circle (0.05);
	\node[right] at (.5, 2) {2};
	\draw (.5,2) circle (0.05);	
	\node[left] at (-.75, 1.5) {1};
	\draw[fill=black](-.75,1.5) circle (0.05);
	\node[right] at (.5, 1) {2};
	\draw[fill=black] (.5,1) circle (0.05);
	\node[left] at (-.75, .5) {1};
	\draw[fill=black] (-.75,.5) circle (0.05);
	\node[left] at (-.75,0) {1};
	\draw[fill=black] (-.75,0) circle (0.05);
	\draw[thick] (-.7,4.5) -- (-.25,4.5) -- (-0.25, .5) -- (-.7,.5);
	\draw[thick] (-.7,3.5) -- (-.25,3.5);
	\draw[thick] (.45,4) -- (-.25,4);
	\draw[thick] (-.7,0) -- (0,0) -- (0, 2.5) -- (.45,2.5);
	\draw[thick] (.45, 1) -- (.25, 1) -- (.25, 2) -- (.45, 2);
	\draw[thick] (-.7, 1.5) -- (-.5, 1.5) -- (-.5, 3) -- (-.7, 3);
	\end{tikzpicture}
\]
Applying $L_2^*$ then gives the following collection of partially completed skeletons (where the first, second, and fourth diagrams above were destroyed)
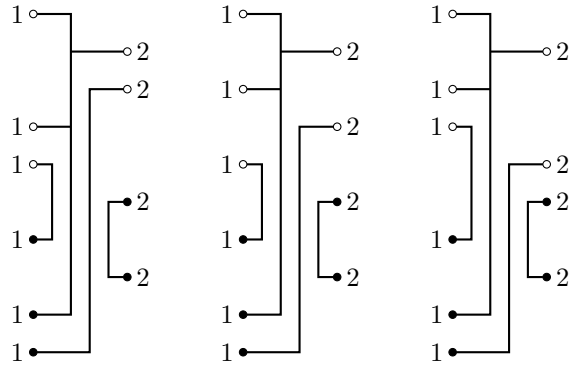
\[
\begin{tikzpicture}[baseline]
	\node[left] at (-.75, 4.5) {1};
	\draw (-.75,4.5) circle (0.05);
	\node[right] at (.5, 4) {2};
	\draw (.5,4) circle (0.05);
	\node[right] at (.5, 3.5) {2};
	\draw (.5,3.5) circle (0.05);
	\node[left] at (-.75, 3) {1};
	\draw(-.75,3) circle (0.05);
	\node[left] at (-.75, 2.5) {1};
	\draw(-.75,2.5) circle (0.05);
	\node[right] at (.5, 2) {2};
	\draw[fill=black] (.5,2) circle (0.05);	
	\node[left] at (-.75, 1.5) {1};
	\draw[fill=black](-.75,1.5) circle (0.05);
	\node[right] at (.5, 1) {2};
	\draw[fill=black] (.5,1) circle (0.05);
	\node[left] at (-.75, .5) {1};
	\draw[fill=black] (-.75,.5) circle (0.05);
	\node[left] at (-.75,0) {1};
	\draw[fill=black] (-.75,0) circle (0.05);
	\draw[thick] (-.7,4.5) -- (-.25,4.5) -- (-0.25, .5) -- (-.7,.5);
	\draw[thick] (-.7,3) -- (-.25,3);
	\draw[thick] (.45,4) -- (-.25,4);
	\draw[thick] (-.7,0) -- (0,0) -- (0, 3.5) -- (.45,3.5);
	\draw[thick] (.45, 1) -- (.25, 1) -- (.25, 2) -- (.45, 2);
	\draw[thick] (-.7, 1.5) -- (-.5, 1.5) -- (-.5, 2.5) -- (-.7, 2.5);
\end{tikzpicture}
\qquad
\begin{tikzpicture}[baseline]
	\node[left] at (-.75, 4.5) {1};
	\draw (-.75,4.5) circle (0.05);
	\node[right] at (.5, 4) {2};
	\draw (.5,4) circle (0.05);
	\node[left] at (-.75, 3.5) {1};
	\draw(-.75,3.5) circle (0.05);
	\node[right] at (.5, 3) {2};
	\draw (.5,3) circle (0.05);
	\node[left] at (-.75, 2.5) {1};
	\draw(-.75,2.5) circle (0.05);
	\node[right] at (.5, 2) {2};
	\draw[fill=black] (.5,2) circle (0.05);	
	\node[left] at (-.75, 1.5) {1};
	\draw[fill=black](-.75,1.5) circle (0.05);
	\node[right] at (.5, 1) {2};
	\draw[fill=black] (.5,1) circle (0.05);
	\node[left] at (-.75, .5) {1};
	\draw[fill=black] (-.75,.5) circle (0.05);
	\node[left] at (-.75,0) {1};
	\draw[fill=black] (-.75,0) circle (0.05);
	\draw[thick] (-.7,4.5) -- (-.25,4.5) -- (-0.25, .5) -- (-.7,.5);
	\draw[thick] (-.7,3.5) -- (-.25,3.5);
	\draw[thick] (.45,4) -- (-.25,4);
	\draw[thick] (-.7,0) -- (0,0) -- (0, 3) -- (.45,3);
	\draw[thick] (.45, 1) -- (.25, 1) -- (.25, 2) -- (.45, 2);
	\draw[thick] (-.7, 1.5) -- (-.5, 1.5) -- (-.5, 2.5) -- (-.7, 2.5);
\end{tikzpicture}
\qquad
\begin{tikzpicture}[baseline]
	\node[left] at (-.75, 4.5) {1};
	\draw (-.75,4.5) circle (0.05);
	\node[right] at (.5, 4) {2};
	\draw (.5,4) circle (0.05);
	\node[left] at (-.75, 3.5) {1};
	\draw(-.75,3.5) circle (0.05);
	\node[left] at (-.75, 3) {1};
	\draw(-.75,3) circle (0.05);
	\node[right] at (.5, 2.5) {2};
	\draw (.5,2.5) circle (0.05);
	\node[right] at (.5, 2) {2};
	\draw[fill=black] (.5,2) circle (0.05);	
	\node[left] at (-.75, 1.5) {1};
	\draw[fill=black](-.75,1.5) circle (0.05);
	\node[right] at (.5, 1) {2};
	\draw[fill=black] (.5,1) circle (0.05);
	\node[left] at (-.75, .5) {1};
	\draw[fill=black] (-.75,.5) circle (0.05);
	\node[left] at (-.75,0) {1};
	\draw[fill=black] (-.75,0) circle (0.05);
	\draw[thick] (-.7,4.5) -- (-.25,4.5) -- (-0.25, .5) -- (-.7,.5);
	\draw[thick] (-.7,3.5) -- (-.25,3.5);
	\draw[thick] (.45,4) -- (-.25,4);
	\draw[thick] (-.7,0) -- (0,0) -- (0, 2.5) -- (.45,2.5);
	\draw[thick] (.45, 1) -- (.25, 1) -- (.25, 2) -- (.45, 2);
	\draw[thick] (-.7, 1.5) -- (-.5, 1.5) -- (-.5, 3) -- (-.7, 3);
\end{tikzpicture}
\]
and applying $L_2^*$ removes all but the last diagram to give
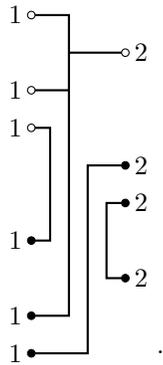
\[
\begin{tikzpicture}[baseline]
	\node[left] at (-.75, 4.5) {1};
	\draw (-.75,4.5) circle (0.05);
	\node[right] at (.5, 4) {2};
	\draw (.5,4) circle (0.05);
	\node[left] at (-.75, 3.5) {1};
	\draw(-.75,3.5) circle (0.05);
	\node[left] at (-.75, 3) {1};
	\draw(-.75,3) circle (0.05);
	\node[right] at (.5, 2.5) {2};
	\draw[fill=black] (.5,2.5) circle (0.05);
	\node[right] at (.5, 2) {2};
	\draw[fill=black] (.5,2) circle (0.05);	
	\node[left] at (-.75, 1.5) {1};
	\draw[fill=black](-.75,1.5) circle (0.05);
	\node[right] at (.5, 1) {2};
	\draw[fill=black] (.5,1) circle (0.05);
	\node[left] at (-.75, .5) {1};
	\draw[fill=black] (-.75,.5) circle (0.05);
	\node[left] at (-.75,0) {1};
	\draw[fill=black] (-.75,0) circle (0.05);
	\draw[thick] (-.7,4.5) -- (-.25,4.5) -- (-0.25, .5) -- (-.7,.5);
	\draw[thick] (-.7,3.5) -- (-.25,3.5);
	\draw[thick] (.45,4) -- (-.25,4);
	\draw[thick] (-.7,0) -- (0,0) -- (0, 2.5) -- (.45,2.5);
	\draw[thick] (.45, 1) -- (.25, 1) -- (.25, 2) -- (.45, 2);
	\draw[thick] (-.7, 1.5) -- (-.5, 1.5) -- (-.5, 3) -- (-.7, 3);
	\end{tikzpicture}.
\]
Applying $L_1^*L_2^*L_1^*L_1^*$ then gives us the desired diagram.
We also see the diagram was weighted by 
$$\kappa_{(21)}(z)\kappa_{(1211)}(z)\kappa_{(22)}\kappa_{(11)}(z)$$
which is the correct product of bi-free cumulants for this bi-non-crossing partition.
\end{ex}

\begin{proof}[Proof of Theorem \ref{operatormodelforapairoffaces}]
Let $\alpha : \{1, \ldots, n\} \to I \sqcup J$.
To see that
\[
\omega(Z_{\alpha(1)} \cdots Z_{\alpha(n)}) = \varphi(z_{\alpha(1)} \cdots z_{\alpha(n)}),
\]
we must demonstrate that the sum of over all
\[
A_{k} \in \{L_{\alpha(k)}^*\} \cup \{\kappa_\beta(z) L_{\alpha(k)}^*T_{\beta} \, \mid \, \beta : \{1,\ldots, m\} \to I \sqcup J\}
\]
of
\[
\langle A_{1} \cdots A_{n}\Omega, \Omega\rangle
\]
is precisely $\varphi(z_{\alpha(1)} \cdots z_{\alpha(n)})$.
(Note that $L_{\alpha(k)}^*T_\beta = 0$ unless $\beta(m) = \alpha(k)$.)
This suffices as these are precisely the terms that appear in expanding the product $Z_{\alpha(1)}\cdots Z_{\alpha(n)}$.
By construction $A_{1} \cdots A_{n}$ acting on $\Omega$ corresponds to creating a (sequence of) partially completed skeletons and $\langle A_{1} \cdots A_{n}\Omega, \Omega\rangle$ will be the weight of the skeleton if the skeleton is complete and otherwise will be zero.
Since
\[
\varphi(z_{\alpha(1)} \cdots z_{\alpha(n)}) = \sum_{\pi \in BNC(\alpha)}\kappa_{\pi}(z).
\]
it suffices to show that there is a bijection between completed skeletons and elements $\pi$ of $BNC(\alpha)$, and that the weight of the skeleton is the corresponding cumulant.

Observe that after $A_k$ is applied, the bottom $n-k+1$ nodes of the partially completed skeleton will be closed, as $A_k$ itself either closed an open node which was already present or added a new block containing one closed node and zero or more open nodes.
In particular, the $(n-k+1)$-th node from the bottom must be on the side corresponding to $\alpha(k)$ since it was closed by $L_{\alpha(k)}^*$.
Thus when we have applied $A_1\cdots A_n$, any skeleton surviving has precisely $n$ nodes and structure arising from $\alpha$.

From a bi-non-crossing partition $\pi \in BNC(\alpha)$, we can recover the choice of $A_1, \ldots, A_n$ which produces it.
To do so, for each block $V = \set{k_1 < \ldots < k_t}$, we let $A_{k_i} = L_{k_i}^*$ for $i \neq t$, and if $\beta(i) = \alpha(k_i)$, we set $A_{k_t} = \kappa_\beta(z)L_{k_t}^*T_\beta$.
Indeed, the partially created skeletons created by $A_k\cdots A_n$ agree with $\pi$ on the bottom $n-k+1$ nodes.
Moreover, given any other product $A_1'\cdots A_n'$ which differs from $A_1\cdots A_n$, consider the greatest index $k$ so that $A_k'\neq A_k$.
Then all partially completed skeletons in $A_k'\cdots A_n'$ and $A_k\cdots A_n$ agree in structure for their bottom $n-k$ nodes, while the next either starts a new block in one case but not the other or starts new blocks of different shapes.
Finally, note that if $\beta$ corresponds to the block $V \in \pi$ as above, then $\kappa_\beta(z) = \kappa_{\pi|_V}(z)$ and so the total weight on the skeleton is precisely $\kappa_\pi(z)$.
\end{proof}

\begin{rem}
In Theorem 7.4 of \cite{voiculescu2013freei}, an operator model for the bi-free central limit distributions was given as sums of creation and annihilation operators on a Fock space.
It is interesting that the operator model from Theorem \ref{operatormodelforapairoffaces} uses different operators.
Indeed for $i, i' \in I$ and $j \in J$, one can check that
\[
T_{(i,i')} = \sum_{n\geq 0} \sum_{\alpha : \{1,\ldots, n\} \to J} L_{i'}L_{\alpha(1)} \cdots L_{\alpha(n)} L_{i} L^*_{\alpha(n)} \cdots L^*_{\alpha(1)}
\]
and
\[
T_{(j,i')} = L_{i'}R_j P
\]
where $P$ is the projection onto the Fock subspace of $\mathcal{H}$ generated by $\{e_j\}_{j \in J}$ and $R_j$ is the right creation operator corresponding to $e_j$.
Therefore, if $c_{k_1, k_2} = \varphi(z_{k_1}z_{k_2})$ for $k_1, k_2 \in I \sqcup J$ with $z$ a bi-free central limit distribution, Theorem \ref{operatormodelforapairoffaces} produces the operators
\[
Z_k = L_k^* + \sum_{k' \in I \sqcup J} c_{k', k} L_{k}^*T_{(k',k)}
\]
which are very different from $L_k + L_k^*$ (if $k \in I$) and $R_k+R_k^*$ (if $k \in J$) proposed in \cite{voiculescu2013freei}.
The main issues with the model involving $\{L_i, L^*_i, R_j, R^*_j \, \mid \, i \in I, j \in J\}$ is that the vectors obtained by applying the algebra generated by these operators to $\Omega$ do not generate the full Fock space - indeed they only generate vectors of the form
\[
e_{i_1} \otimes \cdots \otimes e_{i_n} \otimes e_{j_m} \otimes \cdots \otimes e_{j_1}
\]
where $n,m \geq 0$, $i_1, \ldots, i_n \in I$, and $j_1, \ldots, j_m \in J$.
It is not difficult to see that the vectors obtained by the algebra generated  $\{L_i^*, L_j^*, T_{(i,i)}, T_{(j,j)} \, \mid \, i \in I, j \in J\}$ applied to $\Omega$ generate the full Fock space.
\end{rem}

\end{document}